\documentclass[11pt]{amsart}

\usepackage{amscd,amssymb,amsopn,amsmath,amsthm,mathrsfs,graphics,amsfonts,enumerate,verbatim,calc,
stmaryrd}
\usepackage[all]{xypic}
\usepackage[hyphens]{url}

\usepackage[OT2,OT1]{fontenc}
\newcommand\cyr{%
\renewcommand\rmdefault{wncyr}%
\renewcommand\sfdefault{wncyss}%
\renewcommand\encodingdefault{OT2}%
\normalfont
\selectfont}
\DeclareTextFontCommand{\textcyr}{\cyr}

\usepackage{amssymb,amsmath}

\DeclareFontFamily{OT1}{rsfs}{}
\DeclareFontShape{OT1}{rsfs}{n}{it}{<-> rsfs10}{}
\DeclareMathAlphabet{\mathscr}{OT1}{rsfs}{n}{it}

\newcommand{\cf}{{\itshape cf.} }

\newcommand{\etale}{{\'e}tale }

\numberwithin{equation}{section}
\newtheorem{theorem}{Theorem}[section]
\newtheorem{lemma}[theorem]{Lemma}
\newtheorem{lemmadefinition}[theorem]{Lemma-Definition}
\newtheorem{proposition}[theorem]{Proposition}
\newtheorem{corollary}[theorem]{Corollary}
\newtheorem{maintheorem}{Main Theorem}

\newtheorem{Problem}{Problem}

\theoremstyle{definition}
\newtheorem{definition}[theorem]{Definition}
\newtheorem{remark}[theorem]{Remark}
\theoremstyle{remark}
\newtheorem{construction}[theorem]{Construction}

\newtheorem{example}[theorem]{Example}
\newtheorem*{acknowledgement}{Acknowledgement}

\newcommand{\im}{\operatorname{Im}}
\renewcommand{\ker}{\operatorname{Ker}}

\newcommand{\Spec}{\operatorname{Spec}}

\newcommand{\id}{\operatorname{id}}

\newcommand{\Hom}{\operatorname{Hom}}

\newcommand{\Ann}{\operatorname{Ann}}

\newcommand{\codim}{\operatorname{codim}}

\newcommand{\coker}{\operatorname{Coker}}

\newcommand{\Tr}{\operatorname{Tr}}

\newcommand{\Pic}{\operatorname{Pic}}
\newcommand{\Cl}{\operatorname{Cl}}

\newcommand{\et}{\operatorname{\acute{e}t}}

\newcommand{\bFEt}{{\bf F.\acute{E}t}}

\newcommand{\vpl}{\operatornamewithlimits{\varprojlim}}

\newcommand{\vil}{\operatornamewithlimits{\varinjlim}}

\newcommand{\fm}{\frak{m}}

\newcommand{\fq}{\frak{q}}

\newcommand{\fn}{\frak{n}}

\newcommand{\bfe}{{\bf e}}

\newcommand{\mbN}{\mathbb{N}}

\newcommand{\tower}{(\{ R_i\}_{i \geq 0}, \{t_i \}_{i \geq 0} )}

\begin{document}
\title[Perfectoid towers, tilts and the \'etale cohomology groups]
{Perfectoid towers and their tilts : with an application to the \'etale cohomology groups of local log-regular rings}

\author[S.Ishiro]{Shinnosuke Ishiro}
\address{National Institute of Technology, Gunma College, 580 Toriba-machi, Maebashi-shi, Gunma 371-8530, Japan}
\email{shinnosukeishiro@gmail.com}

\author[K.Nakazato]{Kei Nakazato}
\address{Proxima Technology Inc., 5-24-16, Ueno, Taito, 
Tokyo 110-0005, Japan}
\email{keinakazato31@gmail.com}

\author[K. Shimomoto]{Kazuma Shimomoto}
\address{Department of Mathematics, Institute of Science Tokyo, 2-12-1 Ookayama, Meguro, Tokyo 152-8551, Japan}
\email{shimomotokazuma@gmail.com}

\thanks{2020 {\em Mathematics Subject Classification\/}:13B02, 13B40, 13F35, 14A21, 14G45}

\keywords{
Perfectoid tower,
tilting,
Frobenius maps,
\etale cohomology,
log-regularity.}

\begin{abstract}
To initiate a systematic study on the applications of perfectoid methods to Noetherian rings, we introduce the notions of perfectoid towers and their tilts.
%, and examine their properties.
%Using these, we establish a comparison theorem on finiteness of \'etale cohomology groups of a perfectoid tower and of the tilt.
%We also specialize this to prove the finiteness of the prime-to-$p$-torsion subgroup of the divisor class group of a local log-regular ring that appears in logarithmic geometry in the mixed characteristic case.
We mainly show that the tilting operation preserves several homological invariants and finiteness properties.
Using this, we also provide a comparison result on \'etale cohomology groups under the tilting. 
As an application, we prove finiteness of the prime-to-$p$-torsion subgroup of the divisor class group of a local log-regular ring that appears in logarithmic geometry in the mixed characteristic case.
\end{abstract}

\maketitle

\setcounter{tocdepth}{3}
\tableofcontents

\section{Introduction}

In recent years, the perfectoid technique became one of the most effective tools in commutative ring theory and singularity theory in mixed characteristic.
The \textit{tilting operation} $S\rightsquigarrow S^{\flat}$ for a perfectoid ring $S$ is a central notion in this method, which makes a bridge between objects in mixed characteristic and objects in positive characteristic.
However, perfectoid rings themselves are too big to fit into Noetherian ring theory. 
Hence, for applications, one often requires distinguished Noetherian ring extensions that approximate perfectoids. 
Indeed, in many earlier works (such as \cite{CK19}, \cite{KS20} and \cite{GR22}), one constructs a highly ramified tower of regular local rings or local log-regular rings: 
$$
R_{0}\subseteq R_{1}\subseteq R_{2}\subseteq \cdots
$$
that converges to a (pre)perfectoid ring. Our purposes in this paper are to axiomatize the above towers and establish a kind of Noetherization of perfectoid theory. As an application, we show a finiteness result on the divisor class groups of local log-regular rings.

Fix a prime $p$. We first notice that the highly ramified towers in the positive characteristic case are of the form:
$$
R\subseteq R^{1/p}\subseteq R^{1/p^2}\subseteq\cdots. 
$$ 
This type of towers also appears when one considers the perfect closure of a reduced $\mathbb{F}_{p}$-algebra.
Thus we formulate this class as a tower-theoretic analogue of perfect $\mathbb{F}_{p}$-algebras, and call them \textit{perfect towers} (Definition \ref{ptower}). 
Next, we introduce \textit{perfectoid towers} as a generalization of perfect towers, which includes the towers applied so far (cf.\ Proposition \ref{smalltilt} and Example \ref{examplelogstilt}). 
A perfectoid tower is given by a direct system of rings $R_{0}\xrightarrow{t_{0}}R_{1}\xrightarrow{t_1}\cdots$ satisfying seven axioms in Definition \ref{invqperf} and Definition \ref{stilt}. 
If we assume that each $R_i$ is Noetherian, then these axioms are essential to cope with two main difficulties which we explain below.
%(although the axioms depend on what adic topology is considered, here we deal with only $p$-adic one for simplicity). 

The first difficulty is that the residue ring $R_i/pR_i$ on each layer is not necessarily semiperfect. 
We overcome it by the axioms (b), (c), and (d); these ensure the existence of a surjective ring map $F_i : R_{i+1}/pR_{i+1} \to R_i/pR_i$ which gives a decomposition of the Frobenius endomorphism. 
We call $F_{i}$ the \textit{$i$-th Frobenius projection}, and define a ring $R_j^{s.\flat}$ $(j\geq 0)$ as the inverse limit of Frobenius projections starting at $R_{j}/pR_{j}$. 
Then the resulting tower $R^{s.\flat}_{0}\xrightarrow{t^{s.\flat}_{0}}R^{s.\flat}_{1}\xrightarrow{t^{s.\flat}_1}\cdots$ is perfect, and thus we obtain the tilting operation $(\{R_i\}_{i \geq 0},  \{t_i\}_{i \geq 0}) \rightsquigarrow (\{R^{s.\flat}_i\}_{i \geq 0},  \{t^{s.\flat}_i\}_{i \geq 0})$. 
We remark that this strategy is an axiomatization of the principal arguments in \cite{Sh11}. 

The second one is that each $R^{s.\flat}_{i}$ could be imperfect. 
Because of this, the Witt ring $W(R^{s.\flat}_{i})$ is often uncontrollable. 
On the other hand, the definition of Bhatt-Morrow-Scholze's perfectoid rings (\cite{BMS18}) contains an axiom involving Fontaine's theta map $\theta_{S}: W(S^{\flat}) \to S$ (see Definition \ref{integralperfectoid} (3)), where perfectness of $S^\flat$ is quite effective. 
Our axioms (f) and (g) are the substitutes for it; these require the Frobenius projections to behave well, especially on the $p$-torsion parts.  
This idea is closely related to Gabber and Ramero's characterization of perfectoid rings (\cite[Corollary 16.3.75]{GR22}; see also Theorem \ref{FontainePerfectoid}). 
Indeed, we apply it to deduce that the completed direct limit of a perfectoid tower is a perfectoid ring (Corollary \ref{smalltiltproperty1}).

We then verify fundamental properties of the tilting operation for towers. 
For example, the tilt $(\{R^{s.\flat}_i\}_{i \geq 0},  \{t^{s.\flat}_i\}_{i \geq 0})$ is a perfectoid tower with respect to an ideal $I^{s.\flat}_{0}\subseteq R^{s.\flat}_{0}$ which is the kernel of the $0$-th projection $R_0^{s.\flat} \to R_0/pR_0$ (Proposition \ref{stiltstilt}). 
It induces an isomorphism between two perfectoid objects of different characteristics modulo the defining ideals (Lemma \ref{921WedN}). 
Moreover, this operation preserves several finiteness properties such as Noetherianness on each layer (Proposition \ref{smalltiltproperty2}). 
A key to deducing these statements is the following result (see Remark \ref{rMonNm306} for homological interpretation). 

\begin{maintheorem}[see Theorem \ref{exactstilt}]\label{mt1}
$I^{s.\flat}_{0}$ is a principal ideal.
%and $(I^{s.\flat}_{i+1})^{p}=I^{s.\flat}_{i}R_{i+1}^{s.\flat}$. 
Moreover, we have isomorphisms of (possibly) non-unital rings $(R_i^{s.\flat})_{I^{s.\flat}_0\textnormal{-tor}}\cong (R_{i})_{p\textnormal{-tor}}$ $(i\geq 0)$ that are compatible with $\{t^{s.\flat}_i\}_{i\geq 0}$ and $\{t_i\}_{i\geq 0}$. 
\end{maintheorem}
Under certain normality assumptions, we obtain a comparison result on the finiteness of \etale cohomology groups under tilting for towers (Proposition \ref{TiltEtaleCohHensel}). 
%This proposition is formulated as a generalization of a crucial part of the proof of \cite[Theorem 3.1.3]{KS20} (indeed, the regularity assumption is dropped).
This proposition is considered to rework the crucial part of the proof of \cite[Theorem 3.1.3]{KS20} in a systematic way.
Actually, our proposition applies beyond the regular case.

%substancial, for the purpose of utility, 

As a typical example, we investigate certain towers of \textit{local log-regular rings}; 
this class of rings is defined by Kazuya Kato, and is central to logarithmic geometry (we refer to \cite{GR22}, \cite{Ka89} and \cite{Ogus18} for the readers interested in logarithmic geometry). 
By Kato's structure theorem, a complete local log-regular ring $(R, \mathcal{Q}, \alpha)$ of mixed characteristic is of the form $C(k)\llbracket \mathcal{Q}\oplus \mathbb{N}^r \rrbracket/(p-f)$ where $C(k)$ is a Cohen ring of the residue field $k$ of $R$ (see Theorem \ref{CohenLogReg}). 
Gabber and Ramero gave a systematic way to build a perfectoid tower (in our sense) over it, which consists of local log-regular rings (Construction \ref{logtower}). 
In this paper, we reveal that its tilt also consists of local log-regular rings, and arises from $C(k)\llbracket\mathcal{Q}\oplus \mathbb{N}^{r}\rrbracket/(p)$ (Theorem \ref{TiltingLogRegular}). 
It says that these two rings on the starting layers fit into a Noetherian variant of the tilting correspondence in perfectoid theory (e.g.\ $\mathbb{Z}_{p}$ corresponds to $\mathbb{F}_{p}\llbracket x\rrbracket$). 

%\begin{maintheorem}[Theorem \ref{TiltingLogRegular}]
%\footnote{Theorem \ref{TiltingLogRegular} deals with more general cases. }]
%\label{mainthm4}
%Let $(\{ R_i \}_{i \geq 0} , \{ t_i \}_{i \geq 0})$ be a perfectoid tower which consists of complete local log-regular rings $\{(R_i,\mathcal{Q}^{(i)},\alpha_i)\}_{i\geq 0}$ given by Proposition \ref{claimlog}. 
%Let $(\{R^{s.\flat}_i\}_{i \geq 0} ,\{ t^{s.\flat}_i\}_{i \geq 0})$ be the tilt of $(\{R_i\}_{i \geq 0} ,\{ t_i\}_{i \geq 0})$ associated to $(R_{0}, (p))$. 
%Let $r$ be the dimension of $R_0/I_{\alpha_0}$. 
%Then for any $j \ge 0$, the following assertions hold.
%\begin{enumerate}
%\item
%There exists a homomorphism of monoids $\alpha^{s.\flat}_{j} : \mathcal{Q}^{s.\flat}_j \to R_{j}^{s.\flat}$ such that $(R_j^{s.\flat}, \mathcal{Q}^{s.\flat}_{j}, \alpha^{s.\flat}_{j})$ is a local log-regular ring. Moreover, we obtain the isomorphism $R_j^{s.\flat} \cong k\llbracket \mathcal{Q}^{(j)} \oplus (\mathbb{N}^r)^{(j)}\rrbracket$ where $k$ is the residue field of $R_0$.
%\item
%Assume further that $k_{j+1}$ is $F$-finite. 
%The ring map $t^{s.\flat}_j : R_j^{s.\flat} \to R_{j+1}^{s.\flat}$ is module-finite and $(R_j)^{s.\flat}$ is $F$-finite.
%\end{enumerate}
%\end{maintheorem}

%As a practical direction, we give a comparison result for \'etale cohomology groups of (Henselian) perfectoid towers under tilting (Theorem \ref{TiltEtaleCohHensel}).
%This plays an important role in the proof of Main Theorem \ref{mainthm5} mentioned below.

We regard Theorem \ref{TiltingLogRegular} to be of fundamental importance in the search on the singularities of Noetherian rings via perfectoid methods.
For instance, we can investigate the \textit{divisor class groups} of  local log-regular rings.\footnote{K.\ Kato proved that a local log-regular ring is a normal domain (\cite{Ka89}).}
%by combining it with the comparison result of the finiteness of \'etale cohomology of perfectoid towers and of their tilts (Proposition \ref{TiltEtaleCohHensel}), we can investigate the \textit{divisor class groups} of  local log-regular rings.\footnote{K.\ Kato proved that a local log-regular ring is a normal domain (\cite{Ka89}).}
The divisor class group of a Noetherian normal domain is an important invariant, but it is often hard to compute.\footnote{Every abelian group is realized as a divisor class group of some Dedekind domain (due to Claborn's result \cite{Cla66}).}
On the other hand, Polstra recently proved a remarkable result stating that the torsion subgroup of the divisor class group of an $F$-finite strongly $F$-regular domain is finite (\cite{Pol20}).
Based on this result, we obtain the following finiteness theorem.

%\begin{theorem}[Polstra]\label{PolstraThm}
%Let $(R,\fm)$ be a Noetherian local $\mathbb{F}_p$-algebra, where we set $\mathbb{F}_p:=\mathbb{Z}/p\mathbb{Z}$ for a prime number $p$. Suppose that $R$ is an $F$-finite strongly $F$-regular domain. Then the torsion subgroup of the divisor class group of $R$ is finite.
%\end{theorem}

%We specialize Proposition \ref{TiltEtaleCohHensel} to a study of divisor class groups, and obtain the following finiteness theorem. 
%It can be regarded as a mixed characteristic analogue of Polstra's theorem.
% \ref{PolstraThm}. 

\begin{maintheorem}[Theorem \ref{torsiondivisorclass}]
\label{mainthm5}
Let $(R,\mathcal{Q},\alpha)$ be a local log-regular ring of mixed characteristic with perfect residue field $k$ of characteristic $p>0$, and denote by $\Cl(R)$ the divisor class group with its torsion subgroup $\Cl(R)_{\rm{tor}}$.
%Then the following assertions hold.
%\begin{enumerate}
%\item
%Assume that $R \cong W(k)\llbracket\mathcal{Q}\rrbracket$ for a fine, sharp and saturated monoid $\mathcal{Q}$, where $W(k)$ is the ring of Witt vectors over $k$. Then $\Cl(R)_{\rm{tor}} \otimes \mathbb{Z}[\frac{1}{p}]$ is a finite group. In other words, the $\ell$-primary subgroup of $\Cl(R)_{\rm{tor}}$ is finite for all primes $\ell \ne p$ and vanishes for almost all primes $\ell \ne p$.
%\item
Assume that $\widehat{R^{\rm{sh}}}[\frac{1}{p}]$ is locally factorial, where $\widehat{R^{\rm{sh}}}$ is the completion of the strict Henselization $R^{\rm{sh}}$. Then $\Cl(R)_{\rm{tor}} \otimes \mathbb{Z}[\frac{1}{p}]$ is a finite group. In other words, the $\ell$-primary subgroup of $\Cl(R)_{\rm{tor}}$ is finite for all primes $\ell \ne p$ and vanishes for almost all primes $\ell \ne p$.
%\end{enumerate}
\end{maintheorem}

Our approach to the above theorem is a combination of Theorem \ref{TiltingLogRegular} and Proposition \ref{TiltEtaleCohHensel}.
%the comparison result of the finiteness of \'etale cohomology of perfectoid towers and of their tilts (Proposition \ref{TiltEtaleCohHensel}).
%reduction to Polstra's theorem by using tilting.
%ここに手法を入れる。
%In the first case of the theorem, the proof uses only a few known results in commutative algebra. 
%On the other hand, in the second case, 
%In Main Theorem \ref{mainthm5}, Theorem \ref{TiltEtaleCohHensel} plays a crucial role. 
Although we formulated the above theorem only in mixed characteristic, it has an analogue in characteristic $p>0$, which is relatively easy as follows from the fact that $F$-finite log-regular rings are strongly $F$-regular (Lemma \ref{F-regularLog}) combined with Polstra's theorem. 

For a local log-regular ring $(R, \mathcal{Q}, \alpha)$, Gabber and Ramero constructed the isomorphism $\Cl(Q) \cong \Cl(R)$ where $\Cl(\mathcal{Q})$ is the divisor class group of the associated monoid (\cite[Corollary 12.6.43]{GR22}). It induces the finite generation of $\Cl(R)$.\footnote{The first named author recently provided an elementary proof of \cite[Corollary 12.6.43]{GR22}. See \cite{Ish24}.}

Recently, H. Cai, S. Lee, L. Ma, K. Schwede, and K. Tucker proved that the torsion part of the divisor class group of a BCM-regular ring is finite (see \cite[Theorem 7.0.10.]{CLMSK22}). 
Since they also proved that local log-regular rings are BCM-regular, their result recovers Main Theorem \ref{mainthm5}. 
Although their approach relies on the evaluation of a certain inequality with the perfectoid signature which is defined in \cite{CLMSK22} as an analogue of $F$-signature, it does not use a reduction to positive characteristic and is therefore essentially different from our approach.

This paper is organized as follows.
In \S \ref{SectLog}, we discuss several properties of monoids and local log-regular rings needed in later sections.
%We investigate the properties of the monoid which consists of $p$-th power roots of elements of a given monoid.
%Many properties of 
%我々は後で必要になるlog-regular ringの環論的な性質やモノイドの性質を調べる。
We also record a shorter proof of the result that \textit{local log-regular rings are splinter} based on the direct summand theorem in \S\ref{LRsplinter}. 
In \S \ref{sectionPerf}, we introduce the notions of perfect towers, perfectoid towers, and their tilts.
The most part of this section is devoted to studying fundamental properties of them; in particular, \S\ref{subsecTilt} deals with Main Theorem \ref{mt1}. 
%To define them as tower-theoretic analogues of perfect rings, perfectoid rings, and their tilts, we also introduce the notions of purely inseparable towers and inverse perfection of purely inseparable towers.
%Perfect towers are a class of towers of rings whose direct limit is a perfect $\mathbb{F}_p$-algebra.
%Purely inseparable towers are defined as a naive class of tower of rings that can be defined the inverse perfection of a tower.
%After the definition, we discuss fundamental properties of purely inseparable towers (resp. perfectoid towers) and their inverse perfections (resp. tilts).
In the last subsection \S\ref{smalltiltlog}, we provide explicit examples of perfectoid towers consisting of local log-regular rings, and compute their tilts.
In \S 4, we give a proof for Main Theorem \ref{mainthm5} using the tilting operation, which is an application of \S \ref{SectLog} and \S \ref{sectionPerf}. 
In Appendix, we review the notion of \textit{maximal sequences} associated to certain differential modules due to Gabber and Ramero \cite{GR22}.
This plays an important role in the construction of perfectoid towers of local log-regular rings (Construction \ref{logtower}).

\ \\

\textbf{Convention}: Throughout this paper, we follow the convention stated below. 
\begin{itemize}
\item
We consistently fix a prime $p>0$. If we need to refer to another prime different from $p$, we denote it by $\ell$. 
\item
All rings are assumed to be commutative and unital (unless otherwise stated; cf.\ Theorem \ref{exactstilt} (2)). We mean by a \textit{ring map} a unital ring homomorphism. 
\item
A local ring is a (not necessarily Noetherian) ring with a unique maximal ideal. When a ring $R$ is local, then we use $\fm_R$ (or simply $\fm$ if no confusion is likely) to denote its unique maximal ideal. 
We say that a ring map $f: R\to S$ is \textit{local} if $R$ and $S$ are local rings and $f^{-1}(\mathfrak{m}_{S})=\mathfrak{m}_{R}$. 
\item
Unless otherwise stated, a pair $(A, I)$ consisting of a ring $A$ and an ideal $I\subseteq A$ will be simply called a \textit{pair}.
\item
The Frobenius endomorphism on an $\mathbb{F}_{p}$-algebra $R$ is denoted by $F_R$.
If there is no confusion, we denote it by $\textnormal{Frob}$. 
\end{itemize}

\begin{acknowledgement}
The authors are grateful to Professor K. Fujiwara for his continued support and comments. 
Our gratitude also goes to Professor Shunsuke Takagi, Shunya Saito, Ryo Ishizuka, and Kazuki Hayashi for their advice on the content of this paper. 
Finally, we sincerely thank the referees for their careful reading and valuable suggestions.
The second-named author was partially supported by JSPS Grant-in-Aid for Early-Career Scientists 23K12952. The third-named author was partially supported by JSPS Grant-in-Aid for Scientific Research(C) 18K03257 and Scientific Research(C) 23K03077.
\end{acknowledgement}

\section{Log-regularity}\label{SectLog}
%In this section, we discuss the theory of log rings with a focus on log-regularity. 
%As a preparation to dealing with log rings, we first give a brief review and some remarks on commutative monoids in \S\ref{monoidReview}. 
%It also includes a study on a certain type of extension of monoids appeared in Gabber-Ramero's construction of perfectoid towers. 
%In particular, Lemma \ref{finite424} and Lemma \ref{injective1} in \S\ref{intshmon} are used in our proof of Main Theorem \ref{mainthm5}. 
%In \S\ref{LogRings}, we review the definition of local log-regular rings and their notable properties, such as Kato's structure theorem (Theorem \ref{CohenLogReg}). 
%We also recall the notion of strong $F$-regularity, and describe how it is related to log-regularity. 
%As a consequence, we give a proof for the first part of Main Theorem \ref{mainthm5} (see Proposition \ref{torsiondivisorunramified}). 
%In \S\ref{LRsplinter}, we discuss Gabber-Ramero's result which claims that \textit{any local log-regular ring is a splinter} (Theorem \ref{log-splinter}). We give an alternative proof for it using the Direct Summand Theorem. 
In this section, we discuss several properties of monoids and local log-regular rings.
In \S\ref{monoidReview}, we review basic terms on monoids, and examine the behavior of $p$-times maps which are effectively used in Gabber-Ramero's treatment of perfectoid towers (see Construction \ref{logtower}).
In \S\ref{LogRings}, we review the definition of local log-regular rings and crucial results by K. Kato, and study the relationship with strong $F$-regularity.
In \S\ref{LRsplinter}, we recall Gabber-Ramero's result which claims that \textit{any local log-regular ring is a splinter} (Theorem \ref{log-splinter}), and give an alternative proof for it using the Direct Summand Theorem by Y. Andr\'e \cite{Andre18} (its derived variant is proved by B. Bhatt \cite{Bh18}).

\subsection{Preliminaries on monoids}\label{monoidReview}
\subsubsection{Basic terms}
Here we review the definition of several notions on monoids. 

\begin{definition}
A \textit{monoid} is a semigroup with a unit. A \textit{homomorphism of monoids} is a semigroup homomorphism between monoids that sends a unit to a unit. 
\end{definition}

Throughout this paper, all monoids are assumed to be commutative. 
We denote by $\bold{Mnd}$ the category whose objects are (commutative) monoids and whose morphisms are homomorphisms of monoids.  

We denote a unit by $0$. 
Let $\mathcal{Q}$ be a monoid and $\mathcal{Q}^{*}$ denote the set of all $p \in \mathcal{Q}$ such that there exists $q \in \mathcal{Q}$ such that $p+q=0$. 
Let $\mathcal{Q}^{gp}$ denote the set of elements $a-b$ for all $a,b \in \mathcal{Q}$ where $a-b = a'-b'$ if and only if there exists $c \in \mathcal{Q}$ such that $a+b'+c=a'+b+c$. 
By definition, $\mathcal{Q}^{gp}$ is an abelian group. 
The following conditions yield good classes of monoids.

\begin{definition}\label{propertymon}
Let $\mathcal{Q}$ be a monoid.
\begin{enumerate}
\item
$\mathcal{Q}$ is called \textit{integral} if for $x, x'$ and $y \in Q$, $x+y=x' +y$ implies $x=x'$.

\item $\mathcal{Q}$ is called \textit{fine} if it is finitely generated and integral.

\item $\mathcal{Q}$ is called \textit{sharp} if $Q^{*}=0$.

\item $\mathcal{Q}$ is called \textit{saturated} if the following conditions hold.
\begin{enumerate}
    \item 
    $Q$ is integral. 
    \item
    For any $x \in \mathcal{Q}^{gp}$, if $nx \in \mathcal{Q}$ for some $n \geq 1$, then $x \in \mathcal{Q}$.
\end{enumerate}
%whenever $x \in \mathcal{Q}^{gp}$ is such that $nx \in \mathcal{Q}$ for some $n \geq 1$, then $x \in Q$.
%\footnote{Since $\mathcal{Q}$ is integral, $\iota_{\mathcal{Q}^{gp}}  : \mathcal{Q} \to \mathcal{Q}^{gp}~;~q \mapsto q-0$ is injective (see \cite[Chapter {\bf I}, Proposition 1.3.3]{Ogus18}). In Definition \ref{}(4), we identify $\mathcal{Q}$ with its image in $\mathcal{Q}^{gp}$}
\end{enumerate}
\end{definition}
For an integral monoid $\mathcal{Q}$, the map $\iota_{\mathcal{Q}}  : \mathcal{Q} \to \mathcal{Q}^{gp}~;~q \mapsto q-0$ is injective (see \cite[Chapter {\bf I}, Proposition 1.3.3]{Ogus18}). In Definition \ref{propertymon} (4), we identify $\mathcal{Q}$ with its image in $\mathcal{Q}^{gp}$.

Next we recall the definition of a module over a monoid.\footnote{This is called a \textit{$\mathcal{Q}$-set in \cite{Ogus18}}. We call it as above to follow the convention of the terminology in commutative ring theory.}

\begin{definition}[$\mathcal{Q}$-module]
Let $\mathcal{Q}$ be a monoid. 
\begin{enumerate}
\item
A \emph{$\mathcal{Q}$-module} is a set $M$ equipped with a binary operation 
$$
\mathcal{Q}\times{M}\to M\ ;\ (q, x)\mapsto q+x
$$
having the following properties: 
\begin{enumerate}
\item
$0+x=x$ for any $x\in M$;
\item
$(p+q)+x=p+(q+x)$  for any $p,q\in \mathcal{Q}$ and $x\in M$. 
\end{enumerate}
\item
A \emph{homomorphism of $\mathcal{Q}$-modules} is a (set-theoretic) map $f: M\to N$ between $\mathcal{Q}$-modules such that 
$f(q+x)=q+f(x)$ for any $q\in \mathcal{Q}$ and $x\in M$.
We denote by $\mathcal{Q}\textnormal{-}\bold{Mod}$ the category of $\mathcal{Q}$-modules and homomorphisms of $\mathcal{Q}$-modules. 
\end{enumerate}
\end{definition}

For a monoid $\mathcal{Q}$ and a family of $\mathcal{Q}$-modules $\{ M_i \}_{i \in I}$, we denote by $\coprod_{i \in I} M_i$ the disjoint union with the binary operation induced by that of each $M_i$.
Then it is the coproduct in $\mathcal{Q}\textnormal{-}\bold{Mod}$.

\begin{definition}[Monoid algebras]
    Let $R$ be a ring and let $\mathcal{Q}$ be a monoid.
    Then the \textit{monoid algebra} $R[\mathcal{Q}]$ is the $R$-algebra which is the free $R$-module $R^{\oplus\mathcal{Q}} $, endowed with the unique ring structure induced by the homomorphism of monoids 
    $$\mathcal{Q} \to R[\mathcal{Q}]\ ;\ q \mapsto e^q.$$
\end{definition}

%We refer the reader to the definition of a monoid algebra $R[\mathcal{Q}]$ to \cite{Ogus18}.
%We denote by $e^q$ (resp. $e^{\mathcal{Q}}$) the image of an element $q$ of $\mathcal{Q}$ (resp. the monoid $\mathcal{Q}$) in $R[\mathcal{Q}]$.
For a monoid $\mathcal{Q}$, one obtains the functor 
\begin{equation}\label{Mnd-Mod}
\mathcal{Q}\textnormal{-}\bold{Mod}\to R[\mathcal{Q}]\textnormal{-}\bold{Mod}\ ;\ M\mapsto R[M],
\end{equation}
which is a left adjoint of the forgetful functor $R[\mathcal{Q}]\textnormal{-}\bold{Mod}\to\mathcal{Q}\textnormal{-}\bold{Mod}$. 
Notice that (\ref{Mnd-Mod}) preserves coproducts (we use this property to prove Proposition \ref{Ogus1.4.2.7}).

%%% Def環境をcoproductが何を意味しているのかで書く．
%%% 有限個ではなくて，一般の添字集合で構成して，それがQ-module categoryのcoprodctになっているという書き方に変更．

%\begin{definition}[Coproduct of $\mathcal{Q}$-modules]
%Let $Q$ be a monoid and let $M_1$ and $M_2$ be $\mathcal{Q}$-modules.
%Then we define the coproduct of $\mathcal{Q}$-modules $M_1 \coprod M_2$ as a set of coproducts with the binary operation which is induced by the binary operations of $M_1$ and $M_2$.
%\end{definition}

Like ideals (resp. prime ideals, the Krull dimension) of a ring, an ideal (resp. prime ideals, the dimension) of a monoid is defined as follows. 

\begin{definition}
Let $\mathcal{Q}$ be a monoid.

\begin{enumerate}
\item
%For any monoid $\mathcal{Q}$, $\mathcal{Q}$ itself has an obvious structure as a $\mathcal{Q}$-module. 
A $\mathcal{Q}$-submodule of $\mathcal{Q}$ is called an \textit{ideal of $\mathcal{Q}$}. 
%In other words, a subset $I \subseteq \mathcal{Q}$ is an ideal if and only if for $k \in I$ and $q \in \mathcal{Q}$, $k+q \in I$. 
\item
An ideal $I$ is called \textit{prime} if $I \neq \mathcal{Q}$ and $p+q \in I$ implies $p \in I$ or $q \in I$. Remark that the empty set $\emptyset$ is a prime ideal of $\mathcal{Q}$. 
%\end{enumerate}
%\end{definition}
%Like the Krull dimension of a ring, the dimension of a monoid is defined as follows. 
%\begin{definition}
\item
The \textit{dimension} of a monoid $\mathcal{Q}$ is the maximal length $d$ of the ascending chain\footnote{In this paper, the symbol $\subset$ is used to indicate \emph{proper} inclusion for making an analogy to the inequality symbols as in \cite{Ogus18}.} of prime ideals
\begin{center}
$\emptyset=\fq_{0} \subset \fq_{1} \subset \cdots \subset \fq_{d}=\mathcal{Q}^{+}$,
\end{center}
where $\mathcal{Q}^{+}$ is the set of non-unit elements of $\mathcal{Q}$ (i.e.\ $\mathcal{Q}^{+}=\mathcal{Q}\setminus \mathcal{Q}^{*}$). 
We also denote it by $\dim \mathcal{Q}$.
\end{enumerate}
\end{definition}

Next we review a good class of homomorphisms of monoids, called \textit{exact homomorphisms}. 

\begin{definition}[Exact homomorphisms]
\label{exactmorpshim}
Let $\mathcal{P}$ and $\mathcal{Q}$ be monoids. 
\begin{enumerate}
\item
A homomorphism of monoids $\varphi: \mathcal{P}\to \mathcal{Q}$ is said to be \emph{exact} if the diagram of monoids: 
\[\xymatrix{
\mathcal{P}\ar[d]\ar[r]^{\varphi}&\mathcal{Q}\ar[d]\\
{\mathcal{P}}^{gp}\ar[r]^{\varphi^{gp}}&\mathcal{Q}^{gp}
}\]
is cartesian. 
\item
An \emph{exact submonoid} of $\mathcal{Q}$ is a submonoid $\mathcal{Q}'$ of $\mathcal{Q}$ such that the inclusion map $\mathcal{Q}'\hookrightarrow \mathcal{Q}$ is exact 
(in other words, $(\mathcal{Q}')^{gp}\cap\mathcal{Q}=\mathcal{Q}'$). 
\end{enumerate}
\end{definition}

There is a quite useful characterization of exact submonoids (Proposition \ref{Ogus1.4.2.7}). 
To see this, we recall a graded decomposition of a $\mathcal{Q}$-module attached to a submonoid. For a monoid $\mathcal{Q}$ and a submonoid $\mathcal{Q}'\subseteq \mathcal{Q}$, we denote by $\mathcal{Q}\to\mathcal{Q}/\mathcal{Q}'$ the cokernel of the inclusion map $\mathcal{Q}'\hookrightarrow \mathcal{Q}$.

\begin{definition}\label{qgdecomp}
Let $\mathcal{Q}$ be an integral monoid, and let $\mathcal{Q}'\subseteq \mathcal{Q}$ be a submonoid. Then for any $g\in \mathcal{Q}/\mathcal{Q}'$, we denote by $\mathcal{Q}_{g}$ a $\mathcal{Q}'$-module defined as follows. 
\begin{itemize}
\item
As a set, $\mathcal{Q}_{g}$ is the inverse image of $g\in \mathcal{Q}/\mathcal{Q}'$ under the cokernel $\mathcal{Q}\to \mathcal{Q}/\mathcal{Q}'$ of $\mathcal{Q}' \hookrightarrow\mathcal{Q}$. 
\item
The operation $\mathcal{Q}'\times \mathcal{Q}_{g}\to \mathcal{Q}_{g}$ is defined by the rule: $(q, x)\mapsto q+x$ 
(where $q+x$ denotes the sum of $q$ and $x$ in $\mathcal{Q}$). 
\end{itemize} 
\end{definition}
By definition, $\mathcal{Q}=\coprod_{g\in \mathcal{Q}/\mathcal{Q}'}\mathcal{Q}_{g}$ in $\mathcal{Q}'\textnormal{-}\bold{Mod}$.
%as sets. 
%The right-hand side is viewed as the coproduct of $\mathcal{Q}'$-modules $\{\mathcal{Q}_{g}\}_{g\in \mathcal{Q}/\mathcal{Q}'}$, and hence a $\mathcal{Q}/\mathcal{Q}'$-graded decomposition of the $\mathcal{Q}'$-module $\mathcal{Q}$. 
Using this, one can refine a characterization of exact embeddings described in \cite[Chapter {\bf I}, Proposition 4.2.7]{Ogus18}.

\begin{proposition}[cf. \ {\cite[Chapter {\bf I}, Proposition 4.2.7]{Ogus18}}]
\label{Ogus1.4.2.7}
Let $\mathcal{Q}$ be an integral monoid, and let $\mathcal{Q}'\subseteq \mathcal{Q}$ be a submonoid. 
Let $\theta: \mathcal{Q}'\hookrightarrow \mathcal{Q}$ be the inclusion map, and let $\mathbb{Z}[\theta]: \mathbb{Z}[\mathcal{Q}']\to \mathbb{Z}[\mathcal{Q}]$ be the induced ring map. 
Set $G:=\mathcal{Q}/\mathcal{Q}'$. 
Then the following assertions hold. 
\begin{enumerate}
\item
The $\mathbb{Z}[\mathcal{Q}']$-module $\mathbb{Z}[\mathcal{Q}]$ admits a $G$-graded decomposition $\mathbb{Z}[ \mathcal{Q}]=\bigoplus_{g\in G}\mathbb{Z}[ \mathcal{Q}_g]$. 

\item
The  following conditions are equivalent. 
\begin{enumerate}
\item
The inclusion map $\theta: \mathcal{Q}'\hookrightarrow\mathcal{Q}$ is exact. 
In other words, $(\mathcal{Q}')^{gp}\cap\mathcal{Q}=\mathcal{Q}'$. 
\item
$\mathcal{Q}_{0}=\mathcal{Q}'$. 
\item
$\mathbb{Z}[\theta]$ splits as a $\mathbb{Z}[\mathcal{Q}']$-linear map. 
\item
$\mathbb{Z}[\theta]$ is equal to the canonical embedding $\mathbb{Z}[\mathcal{Q}_{0}]\hookrightarrow \bigoplus_{g\in G}\mathbb{Z}[ \mathcal{Q}_g]$. 
\item
$\mathbb{Z}[\theta]$ is universally injective. 
\end{enumerate}

\end{enumerate}
\end{proposition}
\begin{proof}
(1): By applying the functor (\ref{Mnd-Mod}) (that admits a right adjoint) to the decomposition $\mathcal{Q}=\coprod_{g\in G}\mathcal{Q}_{g}$, we find that the assertion follows. 

(2): Since $\mathcal{Q}_{0}=(\mathcal{Q}')^{gp}\cap\mathcal{Q}$ as sets by definition, the equivalence (a)$\Leftrightarrow$(b) follows. 
The assertion (a)$\Leftrightarrow$(c)$\Leftrightarrow$(e) is none other than \cite[Chapter {\bf I}, Proposition 4.2.7]{Ogus18}. 
Moreover, (d) implies (c) obviously. 
Thus it suffices to show the implication (b)$\Rightarrow$(d). 
Assume that (b) is satisfied. Then one can decompose $\mathcal{Q}$ into the direct sum of $\mathcal{Q}'$-modules $\coprod_{g\in G}\mathcal{Q}_{g}$ with $\mathcal{Q}_0=\mathcal{Q}'$. 
Hence the inclusion map $\mathcal{Q}'\hookrightarrow \mathcal{Q}$ is equal to the canonical embedding $\mathcal{Q}_{0}\hookrightarrow \coprod_{g\in G}\mathcal{Q}_{g}$. 
Thus the induced homomorphism $\mathbb{Z}[\theta]: \mathbb{Z}[\mathcal{Q}_{0}]\hookrightarrow \mathbb{Z}[\coprod_{g\in G}\mathcal{Q}_{g}]$ satisfies (d), as desired. 
\end{proof}

\begin{remark}\label{rmk301Wed}
In the situation of Proposition \ref{Ogus1.4.2.7}, assume that the condition (d) is satisfied. 
Then the split surjection $\pi: \mathbb{Z}[\mathcal{Q}]\to \mathbb{Z}[\mathcal{Q}']$ has the property that $\pi (e^{\mathcal{Q}})= e^{\mathcal{Q}'}$ by the construction of the $G$-graded decomposition $\mathbb{Z}[ \mathcal{Q}]=\bigoplus_{g\in G}\mathbb{Z}[ \mathcal{Q}_g]$. 
Moreover, $\pi (e^{\mathcal{Q}^{+}})\subseteq e^{(\mathcal{Q}')^{+}}$ because $\mathcal{Q}^{+}\cap \mathcal{Q}'\subseteq (\mathcal{Q}')^{+}$. 
We use this fact in our proof for Theorem \ref{log-splinter}. 
\end{remark}

Proposition \ref{Ogus1.4.2.7} implies the following useful lemma.

\begin{lemma}
\label{monoidsplit}
Let $\mathcal{Q}$ be a fine, sharp, and saturated monoid. Let $A$ be a ring. Then there is an embedding of monoids $\mathcal{Q} \hookrightarrow \mathbb{N}^d$ such that the induced map of monoid algebras
\begin{equation}\label{24613N}
A[\mathcal{Q}] \to A[\mathbb{N}^d]
\end{equation}
splits as an $A[\mathcal{Q}]$-linear map. 
\end{lemma}

\begin{proof}
Since $\mathcal{Q}$ is saturated, there exists an embedding $\mathcal{Q}$ into some $\mathbb{N}^d$ as an exact submonoid in view of \cite[Chapter {\bf I}, Corollary 2.2.7]{Ogus18}. Then by Proposition \ref{Ogus1.4.2.7}, %\cite[Proposition 4.2.7]{Ogus18}, 
the associated map of monoid algebras
\begin{equation}
\label{splitmap}
\mathbb{Z}[\mathcal{Q}] \to \mathbb{Z}[\mathbb{N}^d]
\end{equation}
splits as a $\mathbb{Z}[\mathcal{Q}]$-linear map. By tensoring $(\ref{splitmap})$ with $A$, we get the desired split map.
\end{proof}

\subsubsection{$c$-times maps on integral monoids}\label{intshmon}
For an integral monoid $\mathcal{Q}$, we denote by $\mathcal{Q}_\mathbb{Q}$ the submonoid of $\mathcal{Q}^{gp}\otimes_\mathbb{Z}\mathbb{Q}$ defined as 
$$
\mathcal{Q}_\mathbb{Q}:=\{x \otimes r\in\mathcal{Q}^{gp}\otimes_\mathbb{Z}\mathbb{Q}\ |\ x\in \mathcal{Q},\ r\in\mathbb{Q}_{\geq0}\}.
$$
Using this, one can define the following monoid which plays a central role in Gabber-Ramero's construction of perfectoid towers consisting of local log-regular rings. 

\begin{definition}
\label{cthpowermonoid}
Let $\mathcal{Q}$ be an integral monoid. Let $c$ and $i$ be non-negative integers with $c>0$. 
\begin{enumerate}
\item
We denote by $\mathcal{Q}^{(i)}_{c}$ the submonoid of $\mathcal{Q}_{\mathbb{Q}}$ defined as 
$$
\mathcal{Q}^{(i)}_{c}:=\{\gamma\in \mathcal{Q}_\mathbb{Q}\ |\ {c^i}\gamma\in \mathcal{Q}\}.
$$
\item
We denote by $\iota^{(i)}_{c}: \mathcal{Q}^{(i)}_{c}\hookrightarrow \mathcal{Q}^{(i+1)}_{c}$ the inclusion map, and by $\mathbb{Z}[\iota^{(i)}_{c}]: \mathbb{Z}[\mathcal{Q}^{(i)}_{c}]\to \mathbb{Z}[\mathcal{Q}^{(i+1)}_{c}]$ the induced ring map. 
\end{enumerate}
\end{definition}

In the rest of this subsection, we fix a positive integer $c>0$. 
To prove several properties of $\mathcal{Q}^{(i)}_{c}$, 
the following one is important as a starting point.  
\begin{lemma}
\label{finite1}
Let $\mathcal{Q}$ be an integral monoid. Then for 
%every $c>0$ and 
every $i\geq 0$, the following assertions hold. 
\begin{enumerate}
\item 
$\mathcal{Q}^{(i)}_c$ is integral.
\item
$\mathcal{Q}^{(i+1)}_{c}=(\mathcal{Q}^{(i)}_c)^{(1)}_{c}$. 
\item
The $c$-times map on $\mathcal{Q}_{\mathbb{Q}}$ restricts to an isomorphism of monoids: 
$$
f_{c}: \mathcal{Q}^{(i+1)}_{c}\xrightarrow{\cong} \mathcal{Q}^{(i)}_{c}\ ;\ \gamma\mapsto c\gamma. 
$$

\end{enumerate}
\end{lemma}

\begin{proof}
(1): Since $\mathcal{Q}^{gp}\otimes_\mathbb{Z}\mathbb{Q}$ is an integral monoid, so is $\mathcal{Q}_c^{(i)}$.

(2): Since any $g\in(\mathcal{Q}_c^{(i)})^{gp}$ satisfies $c^ig\in \mathcal{Q}^{gp}$, 
the inclusion map $\mathcal{Q}^{gp}\hookrightarrow (\mathcal{Q}_c^{(i)})^{gp}$ becomes an isomorphism $\varphi: \mathcal{Q}^{gp}\otimes_\mathbb{Z}\mathbb{Q}\xrightarrow{\cong}(\mathcal{Q}_c^{(i)})^{gp}\otimes_\mathbb{Z}\mathbb{Q}$ by extension of scalars along the flat ring map $\mathbb{Z}\to \mathbb{Q}$. 
The restriction $\widetilde{\varphi}: \mathcal{Q}_\mathbb{Q}\hookrightarrow (\mathcal{Q}_c^{(i)})_\mathbb{Q}$ of $\varphi$ is also an isomorphism, and 
one can easily check that $\widetilde{\varphi}$ restricts to the desired canonical isomorphism $\mathcal{Q}_{c}^{(i+1)}\xrightarrow{\cong} (\mathcal{Q}^{(i)}_{c})^{(1)}_c$.

(3): It is easy to see that the $c$-times map on  $\mathcal{Q}_{\mathbb{Q}}$ restricts to the homomorphism of monoids $f_{c}$. 
Since the abelian group $\mathcal{Q}_{\mathbb{Q}}=\mathcal{Q}^{gp}\otimes_{\mathbb{Z}}\mathbb{Q}$ is torsion-free, $f_{c}$ is injective. Moreover, any element $\gamma$ in 
$\mathcal{Q}^{{(i)}}_{c}$ is of the form $x\otimes r$ for some $x\in \mathcal{Q}^{gp}$ and $r\in \mathbb{Q}$, which satisfy that $c(x\otimes\frac{r}{c})=\gamma$ and $c^{i+1}(x\otimes\frac{r}{c})\in \mathcal{Q}$. Hence $f_{c}$ is also surjective, as desired. 
\end{proof}

Let us inspect monoid-theoretic aspects of the inclusion $\iota^{(i)}_{c}: \mathcal{Q}^{(i)}_{c}\hookrightarrow \mathcal{Q}^{(i+1)}_{c}$.

\begin{lemma}\label{PropertyQc}
    Let $\mathcal{Q}$ be an integral monoid, and let $\mathbf{P} \in \{ \textit{fine, sharp, saturated }\}$. 
    If $\mathcal{Q}$ satisfies $\mathbf{P}$, then $\mathcal{Q}^{(i)}_c$ also satisfies $\mathbf{P}$ for every $i\geq 0$. 
\end{lemma}

\begin{proof}
Assume that $\mathcal{Q}$ is sharp. Pick $x, y\in \mathcal{Q}_c^{(i)}$ such that $x+y=0$. 
Then $c^{i}x=0$ because $\mathcal{Q}$ is sharp. 
Since $\mathcal{Q}_c^{(i)}$ is a submonoid of the torsion-free group $\mathcal{Q}^{gp}\otimes_{\mathbb{Z}}\mathbb{Q}$, we have $x=0$. 
Next, if $\mathcal{Q}$ is fine or saturated, then it suffices to show the case $i=1$ by Lemma \ref{finite1} (2).
If $\mathcal{Q}$ is fine, then there exists a finite system of generators $\{x_1,\ldots, x_r\}$ of $\mathcal{Q}$. 
Hence $\mathcal{Q}^{(1)}_{c}$ also has a finite system of generators $\{x_j\otimes\frac{1}{c}\}_{j=1,\ldots, r}$.
Finally, assume that $\mathcal{Q}^{(1)}_c$ is saturated.
Pick an element $x$ of $(\mathcal{Q}^{(1)}_{c})^{gp}$ such that $nx\in\mathcal{Q}^{(1)}_{c}$. 
Then the element $cx$ of $\mathcal{Q}^{gp}$ satisfies $n(c x)=c (nx)\in\mathcal{Q}$.  Hence $c x\in \mathcal{Q}$ because $\mathcal{Q}$ is saturated. 
\end{proof}

The assumption of fineness on $\mathcal{Q}$ induces several finiteness properties. 

\begin{lemma}\label{finite424}
Let $\mathcal{Q}$ be a fine monoid. 
Then for %every $c>0$ and 
every $i\geq 0$, the following assertions hold. 
\begin{enumerate}
\item
The ring map $\mathbb{Z}[\iota^{(i)}_{c}]: \mathbb{Z}[\mathcal{Q}^{(i)}_{c}]\to \mathbb{Z}[\mathcal{Q}^{(i+1)}_c]$ is module-finite. 
\item
$\mathcal{Q}^{(i+1)}_{c}/\mathcal{Q}^{(i)}_{c}\cong (\mathcal{Q}^{(i+1)}_{c})^{gp}/(\mathcal{Q}^{(i)}_{c})^{gp}$ as monoids. Moreover, $\mathcal{Q}^{(i+1)}_{c}/\mathcal{Q}^{(i)}_{c}$ forms a finite abelian group. 
\item
For a prime $p>0$, we have $|\mathcal{Q}^{(i+1)}_{p}/\mathcal{Q}^{(i)}_{p}|=p^s$ for some $s\geq 0$. 
\end{enumerate}
\end{lemma}
\begin{proof}
In view of Lemma \ref{finite1} (2), it suffices to deal with the case when $i=0$ only. Here notice that $\mathcal{Q}^{(0)}_{c}=\mathcal{Q}$.

(1): 
Let $\{\frac{1}{c}x_1,\ldots, \frac{1}{c}x_r\}$ be the system of generators of $\mathcal{Q}^{(1)}_c$ obtained in the proof of Lemma \ref{PropertyQc} where $\frac{1}{c}x_j:=x_j\otimes\frac{1}{c}$. 
Then the $\mathbb{Z}[\mathcal{Q}]$-algebra $\mathbb{Z}[\mathcal{Q}^{(1)}_{c}]$ is generated by $\{e^{\frac{1}{c}x_1},\ldots, e^{\frac{1}{c}x_r}\}$, and each $e^{\frac{1}{c}x_j}\in \mathbb{Z}[\mathcal{Q}^{(1)}_{c}]$ is integral over $\mathbb{Z}[\mathcal{Q}]$. Hence $\mathbb{Z}[\iota^{(0)}_{c}]$ is module-finite, as desired. 

(2): 
By \cite[Chapter {\bf I}, Proposition 1.3.3]{Ogus18}, $\mathcal{Q}^{(1)}_{c}/\mathcal{Q}$ is identified with the image of the composition 
\begin{equation}\label{Ogus133}
\mathcal{Q}^{(1)}_{c}\hookrightarrow(\mathcal{Q}^{(1)}_{c})^{gp}\twoheadrightarrow(\mathcal{Q}^{(1)}_{c})^{gp}/\mathcal{Q}^{gp}.
\end{equation} 
Since $\mathcal{Q}^{(1)}_{c}$ is generated by $\frac{1}{c}x_1,\ldots, \frac{1}{c}x_r$, $(\mathcal{Q}^{(1)}_{c})^{gp}$ is generated by $\frac{1}{c}x_1,\ldots, \frac{1}{c}x_r, -\frac{1}{c}x_1,\ldots, -\frac{1}{c}x_r$ as a monoid. 
On the other hand, we have $-\frac{1}{c}x_j\equiv (c-1)\frac{1}{c}x_j\mod \mathcal{Q}^{gp}$ 
for $j=1,\ldots, r$. Hence $(\mathcal{Q}^{(1)}_{c})^{gp}/\mathcal{Q}^{gp}$ is generated by $\{\frac{1}{c}x_j\mod \mathcal{Q}^{gp}\}_{j=1,\ldots, r}$ as a monoid. 
Therefore, the composite map (\ref{Ogus133}) is surjective, and $(\mathcal{Q}^{(1)}_{c})^{gp}/\mathcal{Q}^{gp}$ is a finitely generated torsion abelian group. 
Thus, $\mathcal{Q}^{(1)}_{c}/\mathcal{Q}$ coincides with $(\mathcal{Q}^{(1)}_{c})^{gp}/\mathcal{Q}^{gp}$, which is a finite abelian group, as desired. 

(3): 
Since there exists a surjective group homomorphism
$$
f: \underbrace{\mathbb{Z}/p\mathbb{Z}\times\cdots\times\mathbb{Z}/p\mathbb{Z}}_{r}\twoheadrightarrow (\mathcal{Q}^{(1)}_{p})^{gp}/\mathcal{Q}^{gp}\ ;\ 
(\overline{n_1},\ldots, \overline{n_r})\mapsto n_1\biggl(\frac{1}{p}{x_1}\biggr)+\cdots +n_r\biggl(\frac{1}{p}x_r\biggr)\mod \mathcal{Q}^{gp}\ ,
$$
we have $p^r=|(\mathcal{Q}^{(1)}_{p})^{gp}/\mathcal{Q}^{gp}||\ker({f)}|$. Hence $|(\mathcal{Q}^{(1)}_{p})^{gp}/\mathcal{Q}^{gp}|=p^s$ for some $s\geq 0$. Thus the assertion follows from (2). 
\end{proof}

By assuming saturatedness, one finds the exactness of $\iota^{(i)}_{c}: \mathcal{Q}^{(i)}_{c}\hookrightarrow \mathcal{Q}^{(i+1)}_{c}$. 

\begin{lemma}\label{FineSharpSat}
Let $\mathcal{Q}$ be a saturated monoid. 
Then for %every $c>0$ and 
every $i\geq 0$, $\iota^{(i)}_{c}: \mathcal{Q}_c^{(i)}\hookrightarrow \mathcal{Q}_{c}^{(i+1)}$ is exact (i.e.\ $\mathcal{Q}^{(i+1)}_{c}\cap (\mathcal{Q}^{(i)}_{c})^{gp}=\mathcal{Q}^{(i)}_{c}$).  
\end{lemma}

\begin{proof}
It suffices to show that $\mathcal{Q}^{(i+1)}_{c}\cap (\mathcal{Q}^{(i)}_{c})^{gp}\subseteq \mathcal{Q}^{(i)}_{c}$. 
Pick an element $a\in\mathcal{Q}^{(i+1)}_{c}\cap (\mathcal{Q}^{(i)}_{c})^{gp}$. Then $ca\in \mathcal{Q}^{(i)}_{c}$.  Since $\mathcal{Q}^{(i)}_{c}$ is saturated by Lemma \ref{PropertyQc}, it implies that $a\in \mathcal{Q}^{(i)}_{c}$, as desired. 
\end{proof}

If further $\mathcal{Q}$ is fine, one can learn more about $\mathbb{Z}[\iota^{(i)}_{c}]: \mathbb{Z}[\mathcal{Q}^{(i)}_{c}]\to \mathbb{Z}[\mathcal{Q}^{(i+1)}_{c}]$ using the exactness of $\iota^{(i)}_{c}$ assured by Lemma \ref{FineSharpSat}. 

\begin{lemma}
\label{injective1}
Let $\mathcal{Q}$ be a fine and saturated monoid. %Let $c$ and $i$ be non-negative integers with $c>0$. 
For every $i\geq 0$, set $G_{i}:=\mathcal{Q}^{(i+1)}_{c}/\mathcal{Q}^{(i)}_c$ (which is a finite abelian group by Lemma \ref{finite424} (2)) and $K_{i}:=\textnormal{Frac}(\mathbb{Z}[\mathcal{Q}^{(i)}_{c}])$. 
Then the following assertions hold. 
\begin{enumerate}
\item
For any $g\in G_{i}$, we have an isomorphism of $\mathbb{Z}[\mathcal{Q}_{c}^{(i)}]$-modules $
\mathbb{Z}[(\mathcal{Q}_{c}^{(i+1)})_{g}]\otimes_{\mathbb{Z}[\mathcal{Q}_{c}^{(i)}]}K_{i}\cong K_{i}$. 
\item
The base extension 
$K_{i}\to \mathbb{Z}[\mathcal{Q}^{(i+1)}_{c}]\otimes_{\mathbb{Z}[\mathcal{Q}^{(i)}_{c}]}K_{i}$ of $\mathbb{Z}[\iota^{(i)}_{c}]$ is isomorphic to 
the split injection  
$$
K_{i}\hookrightarrow (K_{i})^{\oplus |G_{i}|}\ ;\ a\mapsto (a,0,\ldots, 0)
$$
as a $K_{i}$-linear map. In particular, $
\dim_{K_{i}}\bigl(\mathbb{Z}[\mathcal{Q}^{(i+1)}_{c}]\otimes_{\mathbb{Z}[\mathcal{Q}^{(i)}_{c}]}K_{i}\bigr)=|\mathcal{Q}^{(i+1)}_{c}/\mathcal{Q}^{(i)}_{c}|$. 
\end{enumerate}
\end{lemma}

\begin{proof}
In view of Lemma \ref{finite1} (2) and Lemma \ref{PropertyQc}, it suffices to show the assertions only for the case when $i=0$. 

(1): Let $y_{g}\in \mathcal{Q}^{(1)}_{c}$ be an element whose image in $\mathcal{Q}^{(1)}_{c}/\mathcal{Q}$ is equal to $g$. 
Then we obtain an injective homomorphism of $\mathcal{Q}$-modules 
\begin{equation}\label{actionQ-mod}
\iota_{g}: \mathcal{Q}\hookrightarrow (\mathcal{Q}^{(1)}_c)_{g}\ ; \ x\mapsto x+y_{g},
\end{equation}
which induces an injective $\mathbb{Z}[\mathcal{Q}]$-linear map $\mathbb{Z}[\iota_{g}]: \mathbb{Z}[\mathcal{Q}]\hookrightarrow \mathbb{Z}[(\mathcal{Q}^{(1)}_c)_{g}]$. 
Thus it suffices to show that $\coker (\mathbb{Z}[\iota_{g}])\otimes_{\mathbb{Z}[\mathcal{Q}]} K_{0}=(0)$, 
i.e., $\coker (\mathbb{Z}[\iota_{g}])$ is a torsion  $\mathbb{Z}[\mathcal{Q}]$-module. On the other hand, we also have a homomorphism of $\mathcal{Q}$-modules  
$$
(\mathcal{Q}^{(1)}_{c})_{g}\to \mathcal{Q}^{gp}\ ;\ y\mapsto y-y_{g}, 
$$
which induces an embedding of $\mathbb{Z}[\mathcal{Q}]$-modules $\coker (\mathbb{Z}[\iota_{g}])\hookrightarrow \mathbb{Z}[\mathcal{Q}^{gp}]/\mathbb{Z}[\mathcal{Q}]$. 
Since $\mathbb{Z}[\mathcal{Q}^{gp}]/\mathbb{Z}[\mathcal{Q}]$ is $\mathbb{Z}[\mathcal{Q}]$-torsion, the assertion follows.

(2): It immediately follows from the combination of Lemma \ref{FineSharpSat}, Proposition \ref{Ogus1.4.2.7} (2), and the assertion (1) of this lemma. 
\end{proof}

%Denote by $X_{Zar}$ the (small) Zariski site on the scheme $X$.

%\begin{definition}
%\label{LogSchemeStr1}
%Let $X$ be a scheme. We say that $X$ comes with a \textit{pre-logarithmic structure} (pre-log structure for short) if there is a sheaf of monoids $\mathcal{M}_X$ on $X_{Zar}$, together with a morphism of monoids $\alpha:\mathcal{M}_X \to \mathcal{O}_X$. Here, we view the structure sheaf $\mathcal{O}_X$ on $X_{Zar}$ as a multiplicative monoid. We say that a pre-log structure is a \textit{log structure} if $\alpha$ induces an isomorphism of sheaves of monoids $\alpha^{-1}(\mathcal{O}_X^\times) \cong \mathcal{O}_X^\times$.
%\end{definition}

\subsection{Local log-regular rings}\label{LogRings}
\subsubsection{Definition of local log-regular rings}
We review the definition and fundamental properties of local log-regular rings.
Unless otherwise stated, we always assume that the monoid structure of a commutative ring is specified by the multiplicative structure.

\begin{definition}[{\cite[Chapter {\bf III}, Definition 1.2.3]{Ogus18}}]
\label{Logring}
Let $R$ be a ring and let $\mathcal{Q}$ be a monoid with a homomorphism $\alpha : \mathcal{Q} \to R$ of monoids. Then we say that the triple $(R,\mathcal{Q},\alpha)$ is a \textit{log ring}. Moreover, we say that $(R,\mathcal{Q},\alpha)$ is a \textit{local log ring} if $(R,\mathcal{Q},\alpha)$ is a log ring, where $R$ is a local ring and $\alpha^{-1}(R^{\times}) = \mathcal{Q}^{*}$.
\end{definition}

In order to preserve the locality of a log structure, we need the locality of a ring map.
\begin{lemma}
\label{locallogextension}
Let $(R, \mathcal{Q}, \alpha)$ be a local log ring and let $(S, \fm_{S})$ be a local ring with a local ring map $\phi : R \to S$. Then $(S, \mathcal{Q}, \phi \circ \alpha)$ is also a local log ring.
\end{lemma}

\begin{proof}
By the locality of $\phi$, we obtain the equality $ (\phi \circ \alpha)^{-1}(S^{\times}) = \mathcal{Q}^{*}$, as desired.
\end{proof}

Now we define \textit{log-regular rings} according to \cite{Ogus18}.

\begin{definition}
\label{LogSchemeStr2}
Let $(R,\mathcal{Q},\alpha)$ be a local log ring, where $R$ is Noetherian and $\overline{\mathcal{Q}}:=\mathcal{Q}/\mathcal{Q}^*$ is fine and saturated. Let $I_{\alpha}$ be the ideal of $R$ generated by the set $\alpha(\mathcal{Q}^{+})$. Then $(R,\mathcal{Q},\alpha)$ is called a \textit{log-regular ring} if the following conditions hold. 
\begin{enumerate}
\item
$R/I_{\alpha}$ is a regular local ring.

\item
$\dim R = \dim R/I_{\alpha} + \dim\mathcal{Q}$.
\end{enumerate}
\end{definition}

\begin{remark}\label{rmk2125}
Note that for a monoid $\mathcal{Q}$ such that $\overline{\mathcal{Q}}$ is fine and saturated, the natural projection $\pi: \mathcal{Q} \twoheadrightarrow \overline{\mathcal{Q}}$ splits (see \cite[Lemma 6.2.10]{GR22}). 
Thus, in the situation of Definition \ref{LogSchemeStr2}, $\alpha$ extends to the homomorphism of monoids $\overline{\alpha}: \overline{\mathcal{Q}}\to R$ along $\pi$. 
Namely, we obtain another local log-regular ring $(R, \overline{\mathcal{Q}}, \overline{\alpha})$ with the same underlying ring, where $\overline{\mathcal{Q}}$ is fine, sharp, and saturated. 
\end{remark}

In his monumental paper \cite{Ka89}, Kato considered log structures of schemes on the \'etale sites, and he then considered them on the Zariski sites \cite{Ka94}. However, we do not need any deep part of logarithmic geometry and the present paper focuses on the local study of schemes with log structures. We should remark that if $k$ is any fixed field and $\mathcal{Q} \subseteq \mathbb{N}^d$ is a fine and saturated monoid, then the monoid algebra $k[\mathcal{Q}]$ is known as an  \textit{affine normal semigroup ring} which is actively studied in combinatorial commutative algebra (see the book \cite{MS04}). The following theorem is a natural extension of the Cohen-Macaulay property for the classical toric singularities over a field proved by Hochster \cite{Ho72}.

\begin{theorem}[{\cite[Theorem 4.1]{Ka94}}]
\label{CMnormal}
Every local log-regular ring is Cohen-Macaulay and normal.
\end{theorem}

Let $R$ be a ring and let $\mathcal{Q}$ be a fine sharp monoid. We denote by $R[\mathcal{Q}^{+}]$ the ideal of $R[\mathcal{Q}]$ generated by elements $\sum_{q \in \mathcal{Q}^{+}} a_{q}e^q$, where $a_{q}$ is an element of $R$. Then we denote by $R\llbracket\mathcal{Q}\rrbracket$ the adic completion of $R[\mathcal{Q}]$ with respect to the ideal $R[\mathcal{Q}^{+}]$. 

As to the structure of complete local log-regular rings, we have the following result analogous to the classical Cohen's structure theorem, originally proved in \cite{Ka94}. We borrow the presentation from \cite[Chapter {\bf III}, Theorem 1.11.2]{Ogus18}.

\begin{theorem}[Kato]
\label{CohenLogReg}
Let $(R,\mathcal{Q},\alpha)$ be a local log ring such that $R$ is Noetherian and $\mathcal{Q}$ is fine, sharp, and saturated. 
Let $k$ be the residue field of $R$ and $\fm_R$ its maximal ideal. 
Let $r$ be the dimension of $R/I_\alpha$.
Then the following assertions hold.
\begin{enumerate}
\item
Suppose that $R$ contains a field. Then $(R,\mathcal{Q},\alpha)$ is log-regular if and only if there exists a commutative diagram:
$$
\begin{CD}
\mathcal{Q} @>>> k\llbracket\mathcal{Q} \oplus \mathbb{N}^r\rrbracket \\
@V\alpha VV @V\psi VV \\
R @>>> \widehat{R}
\end{CD}
$$
where $\widehat{R}$ is the completion along the maximal ideal and $\psi$ is an isomorphism of rings. 

\item
Assume that $R$ is of mixed characteristic $p>0$. Let $C(k)$ be a Cohen ring of $k$. Then $(R,\mathcal{Q},\alpha)$ is log-regular if and only if there exists a commutative diagram: 
$$
\begin{CD}
\mathcal{Q} @>>> C(k)\llbracket\mathcal{Q} \oplus \mathbb{N}^r\rrbracket \\
@V\alpha VV @V\psi VV \\
R @>>> \widehat{R}
\end{CD}
$$
where $\widehat{R}$ is the completion along the maximal ideal and $\psi$ is a surjective ring map with $\ker(\psi)=(\theta)$ for some element $\theta \in \fm_{\widehat{R}}$ whose constant term is $p$. 
Moreover, for any element $\theta' \in \ker (\psi)$ whose constant term is $p$, $\ker (\psi) = (\theta')$ holds.
\end{enumerate}
\end{theorem}

\begin{proof}
The assertion (1) and the first part of (2) are \cite[Chapter {\bf III}, Theorem 1.11.2]{Ogus18}.
Pick an element $\theta' \in \ker (\psi)$ whose constant term is $p$.
%Note that $\psi : C(k)\llbracket \mathcal{Q} \oplus \mathbb{N}^r \rrbracket \to \widehat{R}$ is surjective and $\psi((\mathcal{Q} \oplus \mathbb{N}^r)^+) \widehat{R}$ is the maximal ideal of $\widehat{R}$.
%This implies $p \in \ker \psi + C(k)\llbracket(\mathcal{Q} \oplus \mathbb{N}^r)^+\rrbracket$.
%Since $f' \in \ker \psi$,  $p - f' \in C(k)\llbracket(\mathcal{Q} \oplus \mathbb{N}^r)^+\rrbracket$.
Note that $\theta'$ is a regular element that is not invertible.
By \cite[Chapter {\bf III}, Proposition 1.10.13]{Ogus18}, $C(k)\llbracket\mathcal{Q} \oplus \mathbb{N}^r \rrbracket/ (\theta')$ is a domain of $\dim \mathcal{Q} + r = \dim R = \dim \widehat{R}$. Thus $\ker (\psi) = (\theta')$ holds.\footnote{This argument is due to Ogus. See the proof of \cite[Chapter {\bf III}, Theorem 1.11.2 (2)]{Ogus18}.}
\end{proof}

The completion of a normal affine semigroup ring with respect to the ideal generated by elements of the semigroup is a typical example of local log-regular rings:

\begin{lemma}
\label{logregularexample}
Let $\mathcal{Q}$ be a fine, sharp and saturated monoid and let $k$ be a field. Then $(k\llbracket\mathcal{Q}\rrbracket, \mathcal{Q} , \iota)$ is a local log-regular ring, where $\iota : \mathcal{Q} \hookrightarrow k\llbracket\mathcal{Q}\rrbracket$ is the natural injection.
\end{lemma}

\begin{proof}
By \cite[Chapter {\bf I}, Proposition 3.6.1]{Ogus18}, $(k\llbracket\mathcal{Q}\rrbracket, \mathcal{Q}, \iota)$ is a local log ring. Now applying Theorem \ref{CohenLogReg}, it is a local log-regular ring. 
\end{proof}

\subsubsection{Log-regularity and strong $F$-regularity}
Strongly $F$-regular rings are one of the important classes appearing in the study of $F$-singularities.
Let us recall the definition.
\begin{definition}[Strong $F$-regularity]
Let $R$ be a Noetherian reduced $\mathbb{F}_p$-algebra that is $F$-finite. Let $F_*^eR$ be the same as $R$ as its underlying abelian groups with its $R$-module structure via restriction of scalars via the $e$-th iterated Frobenius endomorphism $F_R^e$ on $R$. Then we say that $R$ is \textit{strongly $F$-regular}, if for any element $c \in R$ that is not in any minimal prime of $R$, there exist an  $e>0$ and a map $\phi \in \Hom_R(F_*^eR,R)$ such that $\phi(F_*^ec)=1$.
\end{definition}

It is known that strongly $F$-regular rings are Cohen-Macaulay and normal (for example, see \cite[Proposition 4.4 and Theorem 4.6]{MP}). 
Let us show that log-regularity implies strong $F$-regularity (in positive characteristic cases). 

\begin{lemma}
\label{F-regularLog}
Let $(R,\mathcal{Q},\alpha)$ be a local log-regular ring of characteristic $p>0$ such that $R$ is $F$-finite. Then $R$ is strongly $F$-regular.
\end{lemma}

\begin{proof}
The completion of $R$ with respect to its maximal ideal is isomorphic to the completion of $k[\mathcal{Q}\oplus \mathbb{N}^r]$, and $\mathcal{Q}$ is a fine, sharp and saturated monoid by Theorem \ref{CohenLogReg} and \cite[Chapter {\bf I}, Proposition 3.4.1]{Ogus18}. Then it follows from Lemma \ref{monoidsplit} that $\mathcal{Q}\oplus \mathbb{N}^r$ can be embedded into $\mathbb{N}^d$ for $d >0$, and $k[\mathcal{Q}\oplus \mathbb{N}^r] \to k[\mathbb{N}^d] \cong k[x_1,\ldots,x_d]$ splits as a $k[\mathcal{Q}\oplus \mathbb{N}^r]$-linear map. Applying \cite[Theorem 3.1]{HH89}, we see that $k[\mathcal{Q}\oplus \mathbb{N}^r]$ is
strongly $F$-regular. After completion, the complete local ring $k\llbracket\mathcal{Q}\oplus \mathbb{N}^r\rrbracket$ is strongly $F$-regular in view of \cite[Theorem 3.6]{Aberbach01}. Then by faithful flatness of $R \to k\llbracket\mathcal{Q}\oplus \mathbb{N}^r\rrbracket$, \cite[Theorem 3.1]{HH89} applies to yield strong $F$-regularity of $R$.
\end{proof}

Under the hypothesis in the following proposition, one can easily establish the finiteness of the torsion part of the divisor class group, which is the first assertion of Theorem \ref{torsiondivisorclass}.

\begin{proposition}
\label{torsiondivisorunramified}
Assume that $R \cong C(k)\llbracket\mathcal{Q}\rrbracket$, where $C(k)$ is a Cohen ring with $F$-finite residue field $k$ and $\mathcal{Q}$ is a fine, sharp, and saturated monoid. Let $\Cl(R)_{\rm{tor}}$ be the torsion subgroup of $\Cl(R)$. Then $\Cl(R)_{\rm{tor}} \otimes \mathbb{Z}_{(\ell)}$ is finite for all $\ell \ne p$, and vanishes for almost all $\ell \ne p$.
\end{proposition}

\begin{proof}
Since $R \cong C(k)\llbracket\mathcal{Q}\rrbracket$, we have 
$$
R/pR \cong k\llbracket\mathcal{Q}\rrbracket,
$$
which is a local $F$-finite log-regular ring. There is an induced map $\Cl(R) \to \Cl(R/pR)$. By restriction, we have $\Cl(R)_{\rm{tor}} \to \Cl(R/pR)_{\rm{tor}}$. Then Lemma \ref{F-regularLog} together with Polstra's result \cite{Pol20} says that $\Cl(R/pR)_{\rm{tor}}$ is finite. Let $C_\ell$ be the maximal $\ell$-subgroup of $\Cl(R)_{\rm{tor}}$. Since $\ell \ne p$, we find that the map $\Cl(R)_{\rm{tor}} \to \Cl(R/pR)_{\rm{tor}}$ restricted to $C_\ell$ is injective in view of \cite[Theorem 1.2]{GrWe94}. In conclusion, $C_\ell$ is finite for all $\ell \ne p$, and $C_\ell$ vanishes for almost all $\ell \ne p$, as desired. 
\end{proof}

\subsection{Log-regularity and splinters}
\label{LRsplinter}
Local log-regular rings have another notable property; they are \textit{splinters}.
Let us recall the definition of splinters.

\begin{definition}
A Noetherian ring $A$ is a \textit{splinter} if every finite ring map $f:A \to B$ such that $\Spec(B) \to \Spec(A)$ is surjective admits an $A$-linear map $h:B \to A$ such that $h \circ f=\id_A$.
\end{definition}

%The Direct Summand Theorem asserts that regular rings are splinters. However, it is usually hard to know how splinters behave under various algebraic operations, such as completion, Henselization, polynomial extension, and so on (for examples, see \cite[Theorem B]{DaTu20} or \cite[Theorem C]{DaTu20}).
In general, it is not easy to see which algebraic operations preserve splinters.
In fact, it remains unsolved whether polynomial rings over a splinter are splinters (see \cite[Question 1']{DaTu20}).
Regarding these issues, Datta and Tucker proved remarkable results (\cite[Theorem B]{DaTu20}, \cite[Theorem C]{DaTu20}, or \cite[Example 3.2.1]{DaTu20}).
See also Murayama's work \cite{Murayama21} for the study of purity of ring extensions.
%For example, the following result was proved only recently \cite[Theorem C]{DaTu20}.

%\begin{theorem}[Datta-Tucker]
%Let $(R,\fm,k)$ be a Noetherian local ring and let $\widehat{R}$ be the $\fm$-adic completion such that the canonical map $R \to \widehat{R}$ has geometrically regular fibers. Then $R$ is a splinter if and only if $\widehat{R}$ is a splinter.
%\end{theorem}
%The paper \cite{DaTu20} contains many interesting results on splinters. 

In order to prove the splinter property, we need a lemma on splitting a map under completion.

\begin{lemma}
\label{splitlemma1}
Let $R$ be a ring and let $f:M \to N$ be an $R$-linear map. Consider a decreasing filtration by $R$-submodules $\{M_\lambda\}_{\lambda \in \Lambda}$ of $M$ and a decreasing filtration by $R$-submodules $\{N_\lambda\}_{\lambda \in \Lambda}$ of $N$ such that $f(M_\lambda) \subseteq N_\lambda$ for each $\lambda \in \Lambda$. Set
$$
\widehat{M}:=\varprojlim_{\lambda \in \Lambda} M/M_\lambda~\mbox{and}~\widehat{N}:=\varprojlim_{\lambda \in \Lambda} N/N_\lambda,
$$
respectively. Finally, assume that $f$ is a split injection that admits an $R$-linear map $g:N \to M$ such that $g \circ f=\id_M$, $g(N_\lambda) \subseteq M_\lambda$ for each $\lambda \in \Lambda$. Then $f$ extends to a split injection $\widehat{M} \to \widehat{N}$.
\end{lemma}

\begin{proof}
By assumption, there is an induced map
$$
M/M_\lambda \xrightarrow{\overline{f}} N/N_\lambda \xrightarrow{\overline{g}} M/M_\lambda
$$
which is an identity on $M/M_\lambda$. Taking inverse limits, we get an identity map $\widehat{M} \to \widehat{N} \to \widehat{M}$, which proves the lemma.
\end{proof}

The next result is originally due to Gabber and Ramero \cite[Theorem 17.3.12]{GR22},\footnote{One notices that the treatment of logarithmic geometry in \cite{GR22} is topos-theoretic, while \cite{Ka94} considers mostly the Zariski sites.} and we give an alternative and short proof, using the Direct Summand Theorem by Y.Andr\'e \cite{Andre18}.

\begin{theorem}
\label{log-splinter}
A local log-regular ring $(R,\mathcal{Q},\alpha)$ is a splinter.
\end{theorem}

\begin{proof}
First, we prove the theorem when $R$ is complete. 
By Remark \ref{rmk2125}, we may assume that $\mathcal{Q}$ is fine, sharp, and saturated. 
By Theorem \ref{CohenLogReg}, we have
$$
R \cong k\llbracket\mathcal{Q} \oplus \mathbb{N}^r\rrbracket,~\mbox{or}~R \cong C(k)\llbracket\mathcal{Q} \oplus \mathbb{N}^r\rrbracket/(p-f),
$$
depending on whether $R$ contains a field or not. Let us consider the mixed characteristic case. By Lemma \ref{monoidsplit}, there is a split injection $C(k)[\mathcal{Q} \oplus \mathbb{N}^r] \to C(k)[\mathbb{N}^d]$ for some $d>0$, which comes from an injection $\delta:\mathcal{Q} \oplus \mathbb{N}^r \to \mathbb{N}^d$ that realizes $\delta(\mathcal{Q} \oplus \mathbb{N}^r)$ as an exact submonoid of $\mathbb{N}^d$. After dividing out by the ideal $(p-f)$, we find that the map
$$
C(k)\llbracket\mathcal{Q} \oplus \mathbb{N}^r\rrbracket/(p-f) \to C(k)\llbracket\mathbb{N}^d\rrbracket/(p-f)
$$
splits as a $C(k)\llbracket\mathcal{Q} \oplus \mathbb{N}^r\rrbracket/(p-f)$-linear map by Remark \ref{rmk301Wed} and Lemma \ref{splitlemma1}. 
Hence, $R$ becomes a direct summand of the complete regular local ring $A:=C(k)\llbracket x_1,\ldots,x_d\rrbracket/(p-f)$. By invoking \cite[Proposition 2.2.8]{DaTu20} and the Direct Summand Theorem \cite{Andre18}, we see that $R$ is a splinter. The case where $R=k \llbracket\mathcal{Q} \oplus \mathbb{N}^r \rrbracket$ can be treated similarly.

%Pick a map $\alpha:A \to R$ that splits $R \to A$. Consider a module-finite extension $R \to S$ such that $S$ is a domain. We want to show that this map splits. Now there is a commutative diagram:
%$$
%\begin{CD}
%R^+ @>>> A^+ \\
%@AAA @AAA \\
%S @>\gamma>> B \\
%@AAA @AAA \\
%R @>>> A \\
%\end{CD}
%$$
%where $R^+$ (resp. $A^+$) is the absolute integral closure of $R$ (resp. $A$), and $B$ is a subring of $A^+$ that is constructed as the chain of $S$ and $A$, thus being finite over $A$. By the Direct Summand Theorem \cite{Andre18}, there is a map $\beta:B \to A$ that splits $A \to B$. Therefore, the composite map $S \xrightarrow{\gamma} B \xrightarrow{\beta} A \xrightarrow{\alpha} R$ splits $R \to S$, as desired. 

Next let us consider the general case. Then the completion map $R \to \widehat{R}$ is faithfully flat and $\widehat{R}$ is a complete local log-regular ring (see Theorem \ref{CohenLogReg}). Hence applying the complete case as above and \cite[Proposition 2.2.8]{DaTu20} shows that $R$ is a splinter, as desired.
%Let $R \to S$ be a module-finite extension with $S$ being a domain, and let $\widehat{R}$ be as in Theorem \ref{CohenLogReg}. By applying the functor $(~) \otimes_R \widehat{R}$ to the exact sequence $0 \to R \to S \to S/R \to 0$, we get an exact sequence: $0 \to \widehat{R} \to S \otimes_R \widehat{R} \to S/R \otimes_R \widehat{R} \to 0$. We have proved that $\widehat{R}$ is a splinter, so the induced sequence
%$$
%0 \to \Hom_{\widehat{R}}(S/R \otimes_R \widehat{R},\widehat{R}) \to %\Hom_{\widehat{R}}(S \otimes_R \widehat{R},\widehat{R}) \to %\Hom_{\widehat{R}}(\widehat{R},\widehat{R}) \to 0
%$$
%is exact. By the faithful flatness of $\widehat{R}$ over $R$, the above exact sequence induces the exact sequence:
%$$
%0 \to \Hom_{R}(S/R,R) \to \Hom_{R}(S,R) \to \Hom_{R}(R,R) \to 0,
%$$
%and we conclude.
\end{proof}

\section{Perfectoid towers and small tilts}\label{sectionPerf}

In this section, we establish a tower-theoretic framework to deal with perfectoid objects using the notion of \textit{perfectoid towers}.
We first introduce the class of \textit{perfect towers} (Definition \ref{ptower}) in \S\ref{subsecPerf}, and then define \textit{inverse perfection of towers} (Definition \ref{def3.5}) in \S\ref{PureInsep}. These notions are tower-theoretic variants of perfect $\mathbb{F}_{p}$-algebras and inverse perfection of rings, respectively. 
In \S\ref{ssaxiomsPerf}, we give a set of axioms for perfectoid towers. 
%We also discuss a method of the elimination of certain torsionness of perfectoid towers. 
In \S\ref{subsecTilt}, we adopt the process of inverse perfection for perfectoid towers as a new tilting operation. 
Indeed, we verify the invariance of several good properties under the tilting; Main Theorem \ref{mt1} is discussed here.
%One obtains the two types of short exact sequences, which are called the \textit{first exact sequence} and the \textit{second exact sequence}. 
In \S\ref{perfd-perfd}, we describe the relationship between perfectoid towers and perfectoid rings. 
This subsection also includes an alternative characterization of perfectoid rings without $\mathbb{A}_{\textnormal{inf}}$. 
In \S\ref{smalltiltlog}, we calculate the tilts of perfectoid towers consisting of local log-regular rings.

\subsection{Perfect towers}\label{subsecPerf}
First of all, we consider the category of \emph{towers of rings}. 
\begin{definition}[Towers of rings]\label{towerdef}\ 
\begin{enumerate}
\item
A \textit{tower of rings} is a direct system of rings of the form 
\[\xymatrix{
R_{0}\ar[r]^{t_{0}}&R_{1}\ar[r]^{t_{1}}&R_{2}\ar[r]^{t_{2}}&\cdots\ar[r]^{t_{i-1}}&R_{i}\ar[r]^{t_{i}}&\cdots,
}\]
and we denote it by $(\{R_i\}_{i \geq 0}, \{t_i\}_{i \geq 0})$ or $\{ R_0 \xrightarrow{t_0} R_1\xrightarrow{t_1} \cdots \}$. 
\item
A morphism of towers of rings $f: (\{R_{i}\}_{i\geq 0}, \{t_{i}\}_{i\geq 0})\to (\{R'_{i}\}_{i\geq 0},\{t'_{i}\}_{i\geq 0})$ is defined as a collection of ring maps 
$\{f_{i}: R_{i}\to R'_{i}\}_{i\geq 0}$ that is compatible with the transition maps; in other words, $f$ represents the commutative diagram
\[\xymatrix{
R_{0}\ar[r]\ar[d]_{f_{0}}&R_{1}\ar[r]\ar[d]_{f_{1}}&R_{2}\ar[r]\ar[d]_{f_{2}}&\cdots\ar[r]&R_{i}\ar[r]\ar[d]_{f_{i}}&\cdots\\
R'_{0}\ar[r]&R'_{1}\ar[r]&R'_{2}\ar[r]&\cdots\ar[r]&R'_{i}\ar[r]&\cdots.
}\]
\end{enumerate}
\end{definition}

For a tower of rings $(\{R_{i}\}_{i\geq 0}, \{t_{i}\}_{i\geq 0})$, we often denote by $R_{\infty}$ an inductive limit $\varinjlim_{i\geq 0}R_{i}$. 
Clearly, an isomorphism of towers of rings $(\{R_{i}\}_{i\geq 0}, \{t_{i}\}_{i\geq 0})\to (\{R'_{i}\}_{i\geq 0}, \{t'_{i}\}_{i\geq 0})$ induces the isomorphism of rings 
$R_{\infty}\xrightarrow{\cong} R'_{\infty}$. 
For every $i\geq 0$, we regard $R_{i+1}$ as an $R_{i}$-algebra via the transition map $t_{i}$.

Recall that the direct perfection of an $\mathbb{F}_{p}$-algebra $R$, which we denote by $R^\textnormal{perf}$, is the direct limit of the tower $(\{R_i\}_{i\geq 0}, \{t_{i}\}_{i\geq 0})$ where $R_{i} = R$ and $t_{i} = F_{R}$ for every $i\geq 0$. 
We denote by $\phi_{R}: R \to R^{\textnormal{perf}}$ the natural map $R_{0} \to \varinjlim_{i\geq 0}R_i$.
If $R$ is reduced, this tower can be regarded as ring extensions obtained by adjoining $p^i$-th roots (cf.\ Example \ref{perfecttower}). 
We formulate such towers as follows, and call them \textit{perfect towers}.

\begin{definition}[Perfect towers]\label{ptower}
A \emph{perfect $\mathbb{F}_{p}$-tower} (or, \emph{perfect tower} as an abbreviated form) is a tower of rings that is isomorphic to a tower:  
\begin{equation}\label{FrobTower}
\xymatrix{
R \ar[r]^{F_{R}} & R \ar[r]^{F_{R}} & R \ar[r]^{F_R} & \cdots 
}
\end{equation}
where $R$ is a reduced $\mathbb{F}_p$-algebra. 
\end{definition}

\begin{example}\label{perfecttower}
Let $R$ be a reduced $\mathbb{F}_p$-algebra.
Let $R^{1/p^i}$ be the ring of $p^i$-th roots of elements of $R$ for every $i\geq 0$.\footnote{For more details of the ring of $p$-th roots of elements of a reduced ring, we refer to \cite{MP}}
Then the tower $R \xrightarrow{t_{0}} R^{1/p} \xrightarrow{t_1} R^{1/p^{2}} \xrightarrow{t_2} \cdots$ is a perfect tower. 
Indeed, we have an isomorphism  $F_i : R^{1/p^{i+1}} \to R^{1/p^{i}}$ ; $x \mapsto x^{p}$. 
By putting $F_{0,i+1}:=F_{0}\circ\cdots\circ F_{i}$, we obtain the following commutative ladder:
\[
\xymatrix{
%line1
R^{1/p^0} \ar[d]^{F_{0,0}} \ar[r]^{t_0}&
R^{1/p} \ar[d]^{F_{0,1}} \ar[r]^{t_1} &
\cdots \ar[r]^{t_{i - 1}} &
R^{1/p^i} \ar[r]^{t_i} \ar[d]^{F_{0, i}}& \cdots  \\ % \ar[r]^{\id_R} \ar[d]^{\id_R} &
R \ar[r]_{F_R} &
R \ar[r]_{F_R} & 
\cdots \ar[r]_{F_R} & 
R\ar[r]_{F_R}  & \cdots .
}
\]
Hence the assertion follows. 
\end{example}

%Let us define Noetherian towers.

%\begin{definition}\label{Noethtower}
%Let $(\{R_i \}_{i \geq 0} , \{ t_i\}_{i \geq 0})$ be a tower of rings.
%Then the tower is called a \textit{Noetherian tower} if $R_i$ is Noetherian for each $i \geq 0$.
%\end{definition}

%In \S \ref{NoethTilt}, we examine about certain Noetherian towers.

%We will discuss the relationship between perfect towers and some towers of characteristic $p>0$.

\subsection{Purely inseparable towers and inverse perfection}\label{PureInsep}
In this subsection, we define \textit{inverse perfection for towers}, which assigns a perfect tower to a tower by arranging a certain type of inverse limits of rings. 
For this, we introduce the following class of towers that admit distinguished inverse systems of rings.

\begin{definition}[Purely inseparable towers]
\label{invqperf}
Let $R$ be a ring, and let $I\subseteq R$ be an ideal. 
\begin{enumerate}
\item
A tower $(\{R_i\}_{i \geq 0},\{t_i\}_{i \geq 0})$ is called a \textit{$p$-purely inseparable tower arising from $(R, I)$} if it satisfies the following axioms. 
%%%%%%%%%%%%%%%%%%%%%%%%
\begin{itemize}
\item[(a)]
$R_{0}=R$ and $p\in I$. 
\item[(b)]
For any $i \geq 0$, the ring map $\overline{t_i}: R_i/IR_i\to R_{i+1}/IR_{i+1}$ induced by $t_i$ is injective. 
\item[(c)]
For any $i\geq 0$, the image of the Frobenius endomorphism on $R_{i+1}/IR_{i+1}$ is contained in the image of $\overline{t_{i}}: R_{i}/IR_{i}\to R_{i+1}/IR_{i+1}$.  
\end{itemize}
%%%%%%%%%%%%%%%%%%%%%%%%
\item\label{Frobprodef}
Let $(\{R_i\}_{i \ge 0}, \{t_i\}_{i \geq 0})$ be a $p$-purely inseparable tower arising from $(R, I)$. 
For any $i\geq 0$, we denote by $F_i :R_{i+1}/IR_{i+1} \to R_i/IR_i$ the ring map (which uniquely exists by the axioms (b) and (c)) such that the following diagram commutes:
\begin{equation}\label{frobeniusdecompose}
\vcenter{
\xymatrix{
R_{i+1}/IR_{i+1}\ar@{->}[rrd]_{F_i}\ar@{->}[rr]^{F_{R_{i+1}/IR_{i+1}}}&&R_{i+1}/IR_{i+1}\\
 &&R_{i}/IR_{i}\ar[u]^{\overline{t_{i}}}. 
}}
\end{equation}
We call $F_i$ \textit{the $i$-th Frobenius projection} (of $(\{R_i\}_{i \ge 0}, \{t_i\}_{i \geq 0})$ associated to $(R, I)$). 
\end{enumerate}
\end{definition}

Hereafter, we leave out `$p$-' from `$p$-purely inseparable towers' if no confusion occurs (i.e.\ we call them simply `purely inseparable towers'). 
Similarly, we omit the phrase in parentheses subsequent to `the $i$-th Frobenius projection' (but we should be careful in some situations; cf.\ Remark \ref{rm304SatN}). 

Throughout this paper, when a purely inseparable tower $(\{R_i\}_{i\geq 0}, \{t_i\}_{i\geq 0})$ is given and its starting layer $(R, I)$ is clear from the context, we denote $R_i/IR_i$ by $\overline{R_i}$ for every $i\geq 0$.

\begin{example}
\label{perfpurelytower}
%Let $R$ be a reduced $\mathbb{F}_p$-algebra.
Any perfect tower is a purely inseparable tower.
More precisely, $(\{ R \}_{i \geq 0},  \{  F_{R}  \}_{i \geq 0})$ appearing in Definition \ref{ptower} is a purely inseparable tower arising from $(R,(0))$.
%arising from some $\bigl(R,(0)\bigr)$.
%Let $(\{ R^{1/p^i} \}_{i \geq 0},  \{  \iota_i  \}_{i \geq 0})$ be the tower defined in Definition \ref{ptower}. 
Indeed, the axioms (a) and (c) are obvious, and the axiom (b) follows from reducedness of $R$. 
The $i$-th Frobenius projection is given by the identiry map on $R$. 
\end{example}

To develop the theory of perfectoid towers, we often use a combination of the diagram (\ref{frobeniusdecompose}) in Definition \ref{invqperf} and the diagram (\ref{frobeniusdecompose2}) in the following lemma. 

\begin{lemma}\label{secondCD}
Let $(\{R_{i}\}_{i \geq 0}, \{t_i\}_{i \geq 0})$ be a purely inseparable tower arising from some pair $(R,I)$. 
Then for every $i\geq 0$, the following assertions hold. 
\begin{enumerate}
\item
$\ker(F_{i})=\ker(F_{\overline{R_{i+1}}})$. 
In particular, $F_{i}$ is injective if and only if $\overline{R_{i+1}}$ is reduced. 
\item
Any element of $\overline{R_{i+1}}$ is a root of a polynomial of the form $X^p-\overline{t_i}(a)$ with $a\in \overline{R_i}$. 
In particular, the ring map $\overline{t_{i}}: \overline{R_i}\hookrightarrow \overline{R_{i+1}}$ is integral.

\item
The following diagram commutes:
\begin{equation}\label{frobeniusdecompose2}
\vcenter{
\xymatrix{
\overline{R_{i+1}}\ar@{->}[rrd]^{F_i}&& \\
\overline{R_i}\ar[u]^{\overline{t_{i}}}\ar@{->}[rr]_{F_{\overline{R_i}}}&&\overline{R_i}.
}}
\end{equation}
\end{enumerate}
\end{lemma}
\begin{proof}
Since $\overline{t_i}$ is injective, the commutative diagram (\ref{frobeniusdecompose}) yields the assertion (1). 
Moreover, (\ref{frobeniusdecompose}) also yields the equality $x^{p}-\overline{t_{i}}(F_i(x))=0$ for every $x\in \overline{R_{i+1}}$. 
Hence the assertion (2) follows. 
To prove (3), let us recall the following equalities
$$
\overline{t_i}\circ F_{\overline{R_i}} = F_{\overline{R_{i+1}}} \circ \overline{t_i} = \overline{t_i} \circ F_i \circ \overline{t_i},
$$
where the second one follows from the commutative diagram (\ref{frobeniusdecompose}).
Since $\overline{t_i}$ is injective, we obtain the equality $F_{\overline{R_i}} = F_i \circ \overline{t_i}$, as desired. 
\end{proof}

Lemma \ref{secondCD} (3) is essential for defining inverse perfection of towers (cf.\ Definition \ref{def3.5} (2)). 
Moreover, it provides a useful tool for studying direct perfection on each layer. 
Recall that for an $\mathbb{F}_{p}$-algebra homomorphism $f: R\to S$,  there exists a unique ring map $f^{\textnormal{perf}}: R^{\textnormal{perf}}\to S^{\textnormal{perf}}$ such that the following diagram commutes (the notations are explained just before Definition \ref{ptower}): 
\[
\xymatrix{
R\ar[r]^{f}\ar[d]_{\phi_{R}}& S\ar[d]^{\phi_S}\\
R^{\textnormal{perf}}\ar[r]^{f^\textnormal{perf}}& S^{\textnormal{perf}}. 
}
\]

\begin{corollary}\label{24525N}
Keep the notation as in Lemma \ref{secondCD}. 
Then $(\overline{t_{i}})^{\textnormal{perf}}: (\overline{R_i})^{\textnormal{perf}}\to (\overline{R_{i+1}})^{\textnormal{perf}}$ is an isomorphism of rings 
whose inverse map is $(F_i)^{\textnormal{perf}}: (\overline{R_{i+1}})^{\textnormal{perf}} \to (\overline{R_i})^{\textnormal{perf}}$ up to the Frobenius automorphisms.
\end{corollary}
\begin{proof}
By Lemma \ref{secondCD} (3), $F_{(\overline{R_{i + 1}})^{\textnormal{perf}}}$ is described as 
$(F_{\overline{R_{ i+1 }}})^{\textnormal{perf}} = (\overline{t_i})^{\textnormal{perf}} \circ  F^{\textnormal{perf}}_{i}$, and it is an automorphism.
Similarly, it follows from the commutative diagram (\ref{frobeniusdecompose}) that $F^{\textnormal{perf}}_{i} \circ (\overline{t_i})^{\textnormal{perf}}$ is the Frobenius automorphism of $(\overline{R_i})^{\textnormal{perf}}$. 
Hence the assertion follows. 
\end{proof}

Now we can introduce the notion of inverse perfection for towers.

\begin{definition}[Inverse perfection of towers]\label{def3.5}
Let $(\{R_i\}_{i \ge 0}, \{t_i\}_{i \geq 0})$ be a ($p$-)purely inseparable tower arising from some pair $(R, I)$. 

\begin{enumerate}
\item
For any $j\geq 0$, we define the \textit{j-th inverse quasi-perfection of $(\{R_i\}_{i \ge 0},\{t_i\}_{i \geq 0})$ associated to $(R, I)$} as a limit:
$$
(R_j)^{q.\textnormal{frep}}_{I}:=\varprojlim \{\cdots \to \overline{R_{j+i+1}} \xrightarrow{F_{j+i}} \overline{R_{j+i}} \to \cdots \xrightarrow{F_j} \overline{R_j}\}. 
$$
\item\label{tiqfrepdef}
For any $j\geq 0$, we define an injective ring map $(t_j)^{q.\textnormal{frep}}_I : (R_{j})^{q.\textnormal{frep}}_I \hookrightarrow (R_{j+1})^{q.\textnormal{frep}}_{I}$ by the rule: 
$$
(t_j)^{q.\textnormal{frep}}_I ((a_i)_{i \geq 0}) := (\overline{t_{j+i}}(a_i))_{i \geq 0}. 
$$
Moreover, we call the resulting tower $(\{(R_{i})^{q.\textnormal{frep}}_{I}\}_{i \geq 0} , \{ (t_{i})^{q.\textnormal{frep}}_{I}\}_{i \geq 0} )$ \textit{the inverse perfection of 
$(\{R_i\}_{i \ge 0}, \{t_i\}_{i \geq 0})$ associated to $(R, I)$}.  
\item
For any $j\geq 0$, we define a ring map $(F_j)^{q.\textnormal{frep}}_I : (R_{j+1})^{q.\textnormal{frep}}_{I} \to (R_j)^{q.\textnormal{frep}}_{I}$ by the rule:  
\begin{equation}\label{FIqfrep}
(F_j)^{q.\textnormal{frep}}_I((a_i)_{i \geq 0}) : = (F_{j+i}(a_i))_{i \geq 0}. 
\end{equation}
\item
For any $j\geq 0$ and for any $m\geq 0$, we denote by $\Phi^{(j)}_{m}$ the $m$-th projection map: 
$$
(R_{j})^{q.\textnormal{frep}}_{I}\to \overline{R_{j+m}}\ ;\  (a_i)_{i\geq 0}\mapsto a_{m}. 
$$
\end{enumerate}
If no confusion occurs, we also denote by $R^{q.\textnormal{frep}}_{j}$ (resp.\ $t^{q.\textnormal{frep}}_{j}$, resp.\ $F^{q.\textnormal{frep}}_j$) the symbol $(R_{j})^{q.\textnormal{frep}}_{I}$ (resp.\ $(t_j)^{q.\textnormal{frep}}_I$, resp.\ $(F_j)^{q.\textnormal{frep}}_{I}$) as an abbreviated form. 
\end{definition}

\begin{example}
Let $R$ be an $\mathbb{F}_{p}$-algebra. Set $R_{i}:=R$ and $t_{i}:=\id_{R}$ for every $i\geq 0$. Then the tower $(\{R_{i}\}_{i\geq 0}, \{t_{i}\}_{i\geq 0})$ is a purely inseparable tower arising from $(R, (0))$. 
Moreover, for every $j\geq 0$, the attached $j$-th inverse quasi-perfection is a limit 
$$
R_j^{q.\textnormal{frep}}=\varprojlim \{\cdots  \xrightarrow{F_R} R \xrightarrow{F_R} R\xrightarrow{F_R}R\}, 
$$
which is none other than the inverse perfection of $R$.
\end{example}

In the situation of Definition \ref{def3.5}, we have the commutative diagram:
%-----------------------
\begin{equation}\label{FrFtatt0}
\vcenter{
\xymatrix{
(R_{j+1})^{q.\textnormal{frep}}_{I}\ar@{->}[rrd]_{(F_j)^{q.\textnormal{frep}}_I}\ar@{->}[rr]^{F_{(R_{j+1})^{q.\textnormal{frep}}_I}}&&(R_{j+1})^{q.\textnormal{frep}}_{I}\\
&&(R_{j})^{q.\textnormal{frep}}_{I}\ar[u]_{(t_j)^{q.\textnormal{frep}}_I}.%&& (\mathcal{Q}^{(n)})_\mathbb{Q}
}}
\end{equation}
%----------------------
Hence the tower $(\{(R_{i})^{q.\textnormal{frep}}_{I}\}_{i \geq 0} , \{ (t_{i})^{q.\textnormal{frep}}_{I}\}_{i \geq 0} )$ is also a purely inseparable tower associated to $((R_0)^{q.\textnormal{frep}}_{I} , (0))$.

In the rest of this subsection, we fix a purely inseparable tower $(\{R_i\}_{i\geq 0},\{t_i\}_{i \geq 0})$ arising from some pair $(R, I)$. 
Keep in mind that the inverse perfection $(\{(R_{i})^{q.\textnormal{frep}}_{I}\}_{i \geq 0}, \{ (t_{i})^{q.\textnormal{frep}}_{I}\}_{i \geq 0} )$ is given in Definition \ref{def3.5} (2), and its Frobenius projections $\{(F_i)^{q.\textnormal{frep}}_I\}_{i\geq 0}$ are described in Definition \ref{def3.5} (3). 
Some basic properties of inverse quasi-perfection are contained in the following proposition.

\begin{proposition}
\label{invqperfprop}
The following assertions hold. 
\begin{enumerate}
\item
For any $j\geq0$, the following assertions hold. 
\begin{enumerate}
\item
Let $J\subseteq (R_j)^{q.\textnormal{frep}}_{I}$ be a finitely generated ideal such that $J^k\subseteq \ker(\Phi^{(j)}_0)$ for some $k>0$ (see Definition \ref{def3.5} (4) for $\Phi^{(j)}_{0}$). 
Then $(R_j)^{q.\textnormal{frep}}_{I}$ is $J$-adically complete and separated. 

\item
Let $x=(x_{i})_{i \geq 0}$ be an element of $(R_j)^{q.\textnormal{frep}}_{I}$. Then $x$ is a unit if and only if $x_{0} \in R_{j}/IR_{j}$ is a unit.

\item
The ring map $(F_j)^{q.\textnormal{frep}}_I$ is an isomorphism.
\end{enumerate}

\item
$(\{(R_{i})^{q.\textnormal{frep}}_{I}\}_{i \geq 0}, \{ (t_{i})^{q.\textnormal{frep}}_{I}\}_{i \geq 0} )$ is a perfect tower. 
In particular, each $(R_i)^{q.\textnormal{frep}}_{I}$ is reduced.
\end{enumerate}
\end{proposition}

\begin{proof}
(1): Since $(\{(R_{j+i})^{q.\textnormal{frep}}\}_{i\geq 0}, \{(t_{j+i})^{q.\textnormal{frep}}_{I}\}_{i\geq 0})$ is the inverse perfection of $(\{R_{j+i}\}_{i\geq 0}, \{t_{j+i}\}_{i\geq 0})$, we are reduced to showing the assertions in the case when $j=0$.

(a): By definition, $(R_0)^{q.\textnormal{frep}}_{I}$ is complete and separated with respect to the linear topology induced by the descending filtration
$$
\ker(\Phi^{(0)}_0)\supseteq \ker(\Phi^{(0)}_1)\supseteq \ker(\Phi^{(0)}_2)\supseteq\cdots.
$$
Moreover, since $J^k\subseteq \ker(\Phi^{(0)}_0)$, we have $(J^k)^{[p^i]}\subseteq \ker(\Phi^{(0)}_i)$ for every $i\geq 0$ by the commutative diagram (\ref{frobeniusdecompose}).\footnote{The symbol $I^{[p^n]}$ for an ideal $I$ in an $\mathbb{F}_p$-algebra $A$ is the ideal generated by the elements $x^{p^n}$ for $x \in I$.} 
On the other hand, since $J^{k}$ is finitely generated, $(J^k)^{p^ir}\subseteq (J^k)^{[p^i]}$ for some $r>0$. 
Thus the assertion follows from \cite[Lemma 2.1.1]{FGK11}.

(b): It is obvious that $x_0 \in \overline{R_0}$ is a unit if $x \in (R_0)^{q.\textnormal{frep}}_I$ is a unit. Conversely, assume that $x_0 \in \overline{R_0}$ is a unit. 
Then for every $i\geq 0$, $x_{i}^{p^{i}}$ is a unit because it is the image of $x_0$ in $\overline{R_i}$. Hence $x_{i}$ is also a unit. 
Therefore, we have isomorphisms $R_{i}/IR_{i}\xrightarrow{\times x_{i}}R_{i}/IR_{i}$ ($i\geq 0$) that are compatible with the Frobenius projections. 
Thus we obtain the isomorphism between inverse limits $(R_{0})^{q.\textnormal{frep}}_{I}\xrightarrow{\times x}(R_{0})^{q.\textnormal{frep}}_{I}$, which yields the assertion.

(c): Consider the shifting map $s_{0}: (R_{0})^{q.\textnormal{frep}}_{I}\to (R_{1})^{q.\textnormal{frep}}_{I}$ defined by the rule 
$s_{0}((a_i)_{i \geq 0}):= (a_{i+1})_{i\geq 0}$. 
Then one can easily check that $s_{0}$ is the inverse map of $(F_{0})^{q.\textnormal{frep}}_{I}$.

(2): Define $F^{q.\textnormal{frep}}_{0,i} : (R_i)^{q.\textnormal{frep}}_I\to (R_{0})^{q.\textnormal{frep}}_I$ as the composite map $(F_0)^{q.\textnormal{frep}}_{I} \circ \cdots \circ (F_{i-1})^{q.\textnormal{frep}}_{I}$ (if $i\geq 1$) or the identity map (if $i=0$). 
Then the collection $\{F^{q.\textnormal{frep}}_{0,i}\}_{i\geq 0}$ gives a morphism of towers from 
$(\{(R_i)^{q.\textnormal{frep}}_I \}_{i \geq 0}  , \{(t_i)^{q.\textnormal{frep}}_I  \}_{i \geq 0})$ to  
$\bigl\{
(R_0)^{q.\textnormal{frep}}_I \xrightarrow{F_{(R_0)^{q.\textnormal{frep}}_I}}
(R_0)^{q.\textnormal{frep}}_I \xrightarrow{F_{(R_0)^{q.\textnormal{frep}}_I}}
\cdots
\bigr\}$. 
Hence by the assertion (c) of (1) and Lemma \ref{secondCD} (1), we complete the proof. 
\end{proof}

The operation of inverse quasi-perfection preserves the locality of rings and ring maps.

\begin{lemma}\label{localqinvperf}
Assume that $R_i$ is a local ring for any $i \geq 0$, and $I\neq R$. 
Then for any $j \geq 0$, the following assertions hold.
\begin{enumerate}
\item
The ring maps $t_j$, $\overline{t_{j}}$, and $F_{j}$ are local. 

\item
$(R_j)^{q.\textnormal{frep}}_{I}$ is a local ring.

\item
The ring map $(t_j)^{q.\textnormal{frep}}_I : (R_j)^{q.\textnormal{frep}}_{I} \to (R_{j+1})^{q.\textnormal{frep}}_{I}$ is local.

\end{enumerate}
\end{lemma}

\begin{proof}
As in Proposition \ref{invqperfprop} (1), it suffices to show the assertions in the case when $j=0$.

(1): Since the diagrams (\ref{frobeniusdecompose}) and (\ref{frobeniusdecompose2}) are commutative, $F_0\circ \overline{t_0}$ and $\overline{t_0}\circ F_0$ are local. 
Hence $\overline{t_0}$ and $F_0$ are local. 
In particular, the composition $R_{0}\twoheadrightarrow \overline{R_0}\xrightarrow{\overline{t_0}}\overline{R_1}$ is local. 
Since this map factors through $t_0$, $t_0$ is also local, as desired.

(2):
Let $\fm_0$ be the maximal ideal of $R_{0}$. 
Consider the ideal $(\fm_0)^{q.\textnormal{frep}}_{I} = \{ (x_{i})_{i \geq 0} \in (R_0)^{q.\textnormal{frep}}_{I}~|~x_{0} \in \fm_0/IR_0 \}$, where $\fm_0/IR_0$ is the maximal ideal of $\overline{R_0}$. 
Then by Proposition \ref{invqperfprop} (1)-(b), $(\fm_0)^{q.\textnormal{frep}}_{I}$ is a unique maximal ideal of $(R_0)^{q.\textnormal{frep}}_{I}$. Hence the assertion follows. 

(3):
By the assertion (2), $( \{ (R_i)^{q.\textnormal{frep}}_I  \}_{i \geq 0},  \{  (t_i)^{q.\textnormal{frep}}_I  \}_{i \geq 0} )$ is a purely inseparable tower of local rings. 
Hence by the assertion (1), $(t_0)^{q.\textnormal{frep}}_I$ is local. 
\end{proof}

A purely inseparable tower also satisfies the following amusing property. This is well-known in positive characteristic, in which case $R_i \to R_{i+1}$ gives a universal homeomorphism (i.e.\ the induced morphism of schemes $\Spec R_{i+1}\to \Spec R_i$ is a universally homeomorphism). See also Corollary \ref{finiteetaleequivalence}. 

\begin{lemma}
\label{categoryfiniteetale}
For every $i \ge 0$, assume that $R_{i}$ is $I$-adically Henselian.\footnote{This condition is realized if $R_{0}$ is $I$-adically Henselian and each $t_{i}: R_{i}\to R_{i+1}$ is integral. } 
%\item
%The Frobenius projection $F_i$ (cf.\ Definition \ref{invqperf} (2)) is surjective. 
%\end{enumerate}
Then the ring map $t_i$ induces an equivalence of categories: 
$$
{\bFEt}(R_i) \xrightarrow{\cong} {\bFEt}(R_{i+1}),
$$
where ${\bFEt}(A)$ is the category of finite \'etale $A$-algebras for a ring $A$.
\end{lemma}

\begin{proof}
By Corollary \ref{24525N}, we obtain the commutative diagram of rings: 
\begin{equation}
\label{standardcomm3}
\vcenter{\xymatrix{
R_{i}\ar[d]_{\pi_i} \ar[r]^{t_i}&R_{i+1}\ar[d]^{\pi_{i+1}}\\
\overline{R_i} \ar[d]_{\phi_{\overline{R_i}}} \ar[r]^{\overline{t_i}}& \overline{R_{i+1}} \ar[d]^{\phi_{\overline{R_{i+1}}}} \\
(\overline{R_i})^{\textnormal{perf}} \ar[r]^{(\overline{t_i})^{\textnormal{perf}}}& (\overline{R_{i+1}})^{\textnormal{perf}}
}}
\end{equation}
where $\pi_j$ ($j\in \{i, i+1\}$) is the natural projection, and the bottom map is an isomorphism. 
Since the Frobenius endomorphism on any $\mathbb{F}_{p}$-algebra gives a universal homeomorphism (\cite[Tag 0CC6]{Stacks}), 
so does $\phi_{\overline{R_j}}$ by \cite[Tag 01YW]{Stacks} and \cite[Tag 01YZ]{Stacks}. 
Hence $\phi_{\overline{R_j}}$ induces an equivalence of categories of finite \'etale algebras over respective rings in view of \cite[Tag 0BQN]{Stacks}. 
The same assertion holds for $\pi_{j}$ by the lifting property of a henselian pair (\cite[Tag 09ZL]{Stacks}). 
By going around the diagram $(\ref{standardcomm3})$, we finish the proof. 
\end{proof}

\subsection{Axioms for perfectoid towers}\label{ssaxiomsPerf}
\subsubsection{Remarks on torsion}
In the subsequent \S\ref{ss-sec.ptp}, we introduce the class of \emph{perfectoid towers} as a generalization of perfect towers. 
For this purpose, we need to deal with a purely inseparable tower arising from $(R, I)$ in the case when $I=(0)$ at least, and hence plenty of $I$-torsion elements. 
Thus we begin by giving several preliminary lemmas on torsion of modules over rings.

\begin{definition}\label{def1231SatN}
Let $R$ be a ring, and let $M$ be an $R$-module. 
\begin{enumerate}
\item
Let $x\in R$ be an element. We say that an element $m\in M$ is \emph{$x$-torsion} if $x^{n}m=0$ for some $n>0$. 
We denote by $M_{x\textnormal{-tor}}$ the $R$-submodule of $M$ consisting of all $x$-torsion elements in $M$. 
\item
Let $I\subseteq R$ be an ideal. 
We say that an element $m\in M$ is \emph{$I$-torsion} if $m$ is $x$-torsion for every $x\in I$. 
We denote by $M_{I\textnormal{-tor}}$ the $R$-submodule of $M$ consisting of all $I$-torsion elements in $M$. 
Note that $M_{(x)\textnormal{-tor}} = M_{x\textnormal{-tor}}= M_{x^{n}\textnormal{-tor}}$ for every $n>0$.

\item
For an element $x\in R$ (resp.\ an ideal $I\subseteq R$), we say that \emph{$M$ has bounded $x$-torsion} (resp.\ \emph{bounded $I$-torsion}) if there exists some $l>0$ such that 
$x^{l}M_{x\textnormal{-tor}}=(0)$ ($I^{l}M_{I\textnormal{-tor}}=(0)$).

\item
For an ideal $I\subseteq R$, we denote by $\varphi_{I, M}: M_{I\textnormal{-tor}}\to M/IM$ the composition of natural $R$-linear maps: 
\begin{equation}\label{94SunN}
M_{I\textnormal{-tor}}\hookrightarrow M\twoheadrightarrow M/IM. 
\end{equation}
\end{enumerate}
\end{definition}

First we record the following fundamental lemma.

\begin{lemma}\label{102Sun}
Let $R$ be a ring, and let $M$ be an $R$-module. Let $x\in R$ be an element. Then for every $n>0$, we have 
$$
M_{x\textnormal{-tor}}\cap x^{n}M= x^{n}M_{x\textnormal{-tor}}.
$$ 
\end{lemma}
\begin{proof}
Pick an element $m\in M_{x\textnormal{-tor}}\cap x^{n}M$. Then $m=x^{n}m_{0}$ for some $m_{0}\in M$, and $x^{l}m=0$ for some $l>0$. Hence $x^{l+n}m_{0}=0$, which implies that 
$m_{0}\in M_{x\textnormal{-tor}}$ and thus $m\in x^{n}M_{x\textnormal{-tor}}$. The containment $x^{n}M_{x\textnormal{-tor}}\subseteq M_{x\textnormal{-tor}}\cap x^{n}M$ is clear. 
\end{proof}

\begin{corollary}\label{96TueN}
Keep the notation as in Lemma \ref{102Sun}, and suppose further that $xM_{x\textnormal{-tor}}=(0)$. 
Then the map $\varphi_{(x), M}: M_{x\textnormal{-tor}}\to M/xM$ (see Definition \ref{def1231SatN} (4)) is injective. 
\end{corollary}
\begin{proof}
It is clear from Lemma \ref{102Sun}. 
\end{proof}

Lemma \ref{102Sun} is also applied to show a half part of the following useful result. 

\begin{lemma}\label{1227TueN}
Keep the notation as in Lemma \ref{102Sun}, and suppose further that $M$ has bounded $x$-torsion. 
Let $\widehat{M}$ be the $x$-adic completion of $M$, and let $\psi: M\to \widehat{M}$ be the natural map. 
Then the restriction $\psi_{\textnormal{tor}}: M_{x\textnormal{-tor}}\to (\widehat{M})_{x\textnormal{-tor}}$ of $\psi$ is an isomorphism of $R$-modules. 
\end{lemma}
\begin{proof}
By assumption, there exists some $l > 0$ such that $x^{l}M_{x\textnormal{-tor}} = (0)$. 
On the other hand, $\ker (\psi_{\textnormal{tor}})=M_{x\textnormal{-tor}}\cap \bigcap_{n=0}^{\infty}x^{n}M$ is contained in $M_{x\textnormal{-tor}}\cap x^lM$, which is equal to $x^{l}M_{x\textnormal{-tor}}$ by Lemma \ref{102Sun}. 
Hence $\psi_{\textnormal{tor}}$ is injective. 

Let us prove the surjectivity. Let $\widehat{N}$ denote the $x$-adic completion of $N$ for every $R$-module $N$. Then we obtain the commutative diagram of $R$-modules: 
\begin{equation}\label{421SunN}
\vcenter{
\xymatrix{
0\ar[r]& M_{x\textnormal{-tor}}\ar[d]_{\psi_{M_{x\textnormal{-tor}}}} \ar[r]^{\iota} & M \ar[d]^{\psi}\ar[r]^{\pi\ \ \ \ \ } &  M/M_{x\textnormal{-tor}} \ar[d]\ar[r] & 0 \\
0\ar[r]& \widehat{M_{x\textnormal{-tor}} }\ar[r]^{\widehat{\iota}} & \widehat{M} \ar[r]^{\widehat{\pi}\ \ \ \ \ } &  \widehat{M/M_{x\textnormal{-tor}}} \ar[r]& 0
}}
\end{equation}
where $\iota$ is the inclusion map and $\pi$ is the natural projection. 
Since $\psi\circ\iota$ factors through $\psi_{\textnormal{tor}}$, it suffices to show that $(\widehat{M})_{x\textnormal{-tor}} \subseteq \im(\widehat{\iota} \circ \psi_{M_{x\textnormal{-tor}}})$. 
First, $\psi_{M_{x\textnormal{-tor}}}$ is bijective because it is isomorphic to the canonical isomorphism 
$M_{x\textnormal{-tor}}/(x^{l}) \xrightarrow{\cong} \widehat{ M_{x\textnormal{-tor}} }/(x^{l})$. 
To show that $(\widehat{M})_{x\textnormal{-tor}} \subseteq \im(\widehat{\iota})$, note that the top row of (\ref{421SunN}) forms an exact sequence, and it consists of $R$-modules that have bounded $x$-torsion. 
Then by \cite[Tag 0923]{Stacks} and right exactness of derived completion functors, $\ker (\widehat{\pi}) = \im (\widehat{\iota})$ (in fact, the bottom sequence is also exact because $\psi_{\textnormal{tor}}$ is injective). 
Moreover, since $\widehat{M/M_{x\textnormal{-tor}}}$ is $x$-torsion free by \cite[Chapter \textbf{II}, Lemma 1.1.5]{FK18}, $(\widehat{M})_{x\textnormal{-tor} }\subseteq \ker (\widehat{\pi})$.  
Hence the assertion follows. 
\end{proof}

The following lemma is used for proving Main Theorem \ref{mt1} (cf.\ Lemma \ref{lem228242027}). 

\begin{lemma}\label{1230FriN}
Let $R$ be a ring, and let $M$ be an $R$-module. Let $x\in R$ be an element. Then for every $n>0$, we have 
\begin{equation}\label{1229ThuN}
\Ann_{M/x^{n}M}(x)\subseteq\im(\varphi_{(x^n), M})+x^{n-1}(M/x^{n}M). 
\end{equation}
\end{lemma}
\begin{proof}
Pick an element $m\in M$ such that $xm\in x^{n}M$. Then $x(m-x^{n-1}m')=0$ for some $m'\in M$. 
In particular, $m-x^{n-1}m'\in M_{x^{n}\textnormal{-tor}}$. 
Hence $m\mod x^{n}M$ lies in the right-hand side of (\ref{1229ThuN}), as desired. 
\end{proof}

In the case when $M=R$, we can regard $M_{I\textnormal{-tor}}$ as a (possibly) non-unital subring of $R$. 
This point of view provides valuable insights. For example, ``reducedness'' for $R_{I\textnormal{-tor}}$ induces a good property on boundedness of torsions. 

\begin{lemma}\label{2294N}
Let $(R, I)$ be a pair such that $R_{I\textnormal{-tor}}$ does not contain any non-zero nilpotent element of $R$. Then $IR_{I\textnormal{-tor}}=(0)$. 
\end{lemma}
\begin{proof}
It suffices to show that $xR_{I\textnormal{-tor}}=0$ for every $x\in I$. 
Pick an element $a\in R_{I\textnormal{-tor}}$. Then for a sufficiently large $n>0$, $x^na=0$. 
Hence $(xa)^{n}=x^na\cdot a^{n-1}=0$. 
Thus we have $xa=0$ by assumption, as desired. 
\end{proof}

\begin{corollary}\label{cor306MonN}
Let $(\{R_i\}_{i\geq 0},\{t_i\}_{i \geq 0})$ be a purely inseparable tower arising from some pair $(R, I)$. Then for every $i\geq0$ and every ideal $J\subseteq (R_{i})^{q.\textnormal{frep}}_{I}$, we have $J((R_i)^{q.\textnormal{frep}}_{I})_{J\textnormal{-tor}}=(0)$. 
\end{corollary}
\begin{proof}
Since $(R_i)^{q.\textnormal{frep}}_{I}$ is reduced by Proposition \ref{invqperfprop} (2), the assertion follows from Lemma \ref{2294N}. 
\end{proof}

Furthermore, we can treat $R_{I\textnormal{-tor}}$ as a positive characteristic object (in the situation of our interest), even if $R$ is not an $\mathbb{F}_{p}$-algebra. 

\begin{lemmadefinition}\label{2295N}
Let $(R, I)$ be a pair such that $p\in I$ and $IR_{I\textnormal{-tor}}=(0)$. 
Then the multiplicative map:  
\begin{equation}\label{1213TueN}
R_{I\textnormal{-tor}}\to R_{I\textnormal{-tor}};\ x\mapsto x^{p}
\end{equation}
is also additive. We denote by $F_{R_{I\textnormal{-tor}}}$ the map (\ref{1213TueN}). 
\end{lemmadefinition}
\begin{proof}
It immediately follows from the binomial theorem. 
\end{proof}

\subsubsection{Perfectoid towers and pillars}\label{ss-sec.ptp}

Now, we define perfectoid towers.

\begin{definition}{(Perfectoid towers)}
\label{stilt}
Let $R$ be a ring, and let $I_{0}\subseteq R$ be an ideal. 
A tower $(\{R_{i}\}_{i\geq 0}, \{t_{i}\}_{i\geq 0})$ is called a \textit{(p-)perfectoid tower arising from} $(R, I_{0})$ if it is a $p$-purely inseparable tower arising from $(R,I_0)$ (cf.\ Definition \ref{invqperf} (1)) 
and satisfies the following additional axioms.
\begin{itemize}
\item[(d)]
For every $i\geq 0$, the $i$-th Frobenius projection $F_i : R_{i+1}/I_{0}R_{i+1} \to R_{i}/I_{0}R_{i}$  (cf.\ Definition \ref{invqperf} (2)) is surjective. 
\item[(e)]
For every $i \geq 0$, $R_i$ is an $I_0$-adically Zariskian ring (i.e.\ $I_0R_i$ is contained in the Jacobson radical of $R_i$).
\item[(f)]
$I_{0}$ is a principal ideal, and $R_{1}$ contains a principal ideal $I_{1}$ that satisfies the following axioms. 
\begin{itemize}
\item[(f-1)]
$I_1^p = I_0R_1$. 
\item[(f-2)] 
For every $i \geq 0$, $\ker (F_i)=I_1(R_{i+1}/I_{0}R_{i+1})$.
\end{itemize}
\item[(g)]
For every $i\geq 0$, $I_{0}(R_{i})_{I_{0}\textnormal{-tor}}=(0)$. Moreover, there exists a (unique) bijective map 
$(F_{i})_{\textnormal{tor}}: (R_{i+1})_{I_{0}\textnormal{-tor}}\to (R_{i})_{I_{0}\textnormal{-tor}}$ such that the diagram: 
\begin{equation}\label{24519N}
\vcenter{
\xymatrix{
(R_{i+1})_{I_0\textnormal{-tor}}\ar[d]_{(F_i)_{\textnormal{tor}}}\ar[r]^{\varphi_{I_{0}, R_{i+1}}}&R_{i+1}/I_{0}R_{i+1}\ar[d]^{F_{i}}\\
(R_{i})_{I_{0}\textnormal{-tor}}\ar[r]_{\varphi_{I_{0}, R_i}}&R_{i}/I_{0}R_{i}
}}
\end{equation}
commutes (see Definition \ref{def1231SatN} for the notation; see also Corollary \ref{96TueN}). 
\end{itemize}
\end{definition}

\begin{remark}\label{remark2022705}
If $I_0$ is generated by an element whose image in $R_i$ is a non-zerodivisor for every $i \geq 0$, then the axiom (g) is satisfied automatically. If $R_1$ is reduced and $I_{0}=(0)$, then the axiom (g) follows from the axioms (d) and (f). 
Consequently, if every $t_i$ is injective and $\varinjlim_{i\geq 0}R_i$ is a domain, one can ignore the axiom (g) when checking that $(\{R_i\}_{i \geq 0}, \{t_i\}_{i \geq 0})$ is a perfectoid tower. 
\end{remark}

We have some examples of perfectoid towers.
\begin{example}
\begin{enumerate}
\item (\cf \cite[Definition 4.4]{Sh11})
Let $(R, \fm, k)$ be a $d$-dimensional unramified regular local ring of mixed characteristic $p>0$ whose residue field is perfect. Then we have
$$
R \cong W(k)\llbracket x_2, \ldots, x_d\rrbracket.
$$
For every $i\geq 0$, set $R_i := R[p^{1/p^i}, x_2^{1/p^i}, \ldots, x_d^{1/p^i}]$, and let $t_i : R_i \to R_{i+1}$ be the inclusion map. 
Then the tower $(\{ R_i \}_{i \geq 0}, \{ t_i \}_{i \geq 0})$ is a perfectoid tower arising from $(R, (p))$. Indeed, the Frobenius projection $F_i : R_{i+1}/pR_{i+1} \to R_{i}/pR_i$ is given as the $p$-th power map.\footnote{The axiom (f-2) follows from the normality of $R_i$. The other axioms are clearly satisfied. }

\item
For some generalization of (1), one can build a perfectoid tower arising from a complete local log-regular ring. For details, see \S\ref{smalltiltlog}.
\item
We note that $t_i$ (resp. $F_i$) of a perfectoid tower is not necessarily the inclusion map (resp. the $p$-th power map). For instance, let $R$ be a reduced $\mathbb{F}_p$-algebra. Set $R_i := R$, $t_i := F_R$, and $F_i := \id_R$ for every $i \geq 0$. Then $(\{ R_i \}_{ i \geq 0} , \{ t_i \}_{i \geq 0})$ is a perfectoid tower arising from $(R, (0))$.
\end{enumerate}
\end{example}

The class of perfectoid towers is a generalization of perfect towers. 
\begin{lemma}\label{perfdperftower}
Let $(\{R_{i}\}_{i\geq 0}, \{t_{i}\}_{i\geq 0})$ be a tower of $\mathbb{F}_p$-algebras. Then the following conditions are equivalent. 
\begin{enumerate}
\item
$(\{R_{i}\}_{i\geq 0}, \{t_{i}\}_{i\geq 0})$ is a perfect $\mathbb{F}_{p}$-tower (cf. Definition \ref{ptower}).
\item
$(\{R_{i}\}_{i\geq 0}, \{t_{i}\}_{i\geq 0})$ is a $p$-perfectoid tower arising from $(R_{0}, (0))$. 
\end{enumerate}
\end{lemma}

\begin{proof}
First we verify the implication $(1) \Rightarrow (2)$. For this, we may assume that $(\{R_{i}\}_{i\geq 0}, \{t_{i}\}_{i\geq 0})$ is of the form $R\xrightarrow{F_{R}} R \xrightarrow{F_{R}} R \xrightarrow{F_R} \cdots$ (see Definition \ref{ptower}). 
By Example \ref{perfpurelytower}, it is a purely inseparable tower arising from  $(R, (0))$. 
The axiom (e) in Definition \ref{stilt} is obvious.
Moreover, the Frobenius projection $F_i$ (cf.\ Example  \ref{perfpurelytower}) is an isomorphism for any $ i \geq 0$. 
Hence the axioms (d), (f), and (g) are also satisfied, which yields the assertion.

%Assume that $(\{R_{i}\}_{i\geq 0}, \{t_{i}\}_{i\geq 0})$ is a perfect $\mathbb{F}_{p}$-tower. Then the axioms (a), (b), (e), (f-1), (g) are obvious.
%Thus it suffices to show that there exists an isomorphism $F_i : R^{1/ p^{i+1}} \to R^{1 / p^i}$ such that $\iota_i \circ F_i =F_{R^{1/p^{i+1}}} $ for any $ i \geq 0$.
%We define the ring map $F_i : R^{1/p^{i+1}} \to R^{1/p^i}$ as the colimit of the following ladder:
%
%
%Then $F_i$ is an isomorphism.
%Thus the assertion $(1) \Rightarrow (2)$ holds.
Conversely, assume that $(\{R_{i}\}_{i\geq 0}, \{t_{i}\}_{i\geq 0})$ is a perfectoid tower arising from $(R_0, (0))$.
Since $F_{i}$ is identified with $(F_i)_\textnormal{tor}$ in the axiom (g), the injectivity of $(F_i)_\textnormal{tor}$ implies that $F_{i}$ is injective. 
In other words, $R_i$ is reduced by Lemma \ref{secondCD} (1). 
Furthermore, $F_i$ is an isomorphism by the axiom (d) or the surjectivity of $(F_i)_{\textnormal{tor}}$. 
Hence we obtained the desired isomorphism of towers:
\begin{equation}\label{6430TueN}
\vcenter{
\xymatrix{
R_0 \ar[d]^{\id_{R_0}} \ar[r]^{t_0}&
R_1 \ar[d]^{F_{0}} \ar[r]^{t_1}&
R_2 \ar[d]^{F_{0}\circ F_1} \ar[r]^{t_2} &
R_3 \ar[d]^{F_{0}\circ F_1\circ F_2} \ar[r]^{t_3} & \cdots \\
R_0 \ar[r]_{F_{R_0}} &
R_0 \ar[r]_{F_{R_0}} &
R_0 \ar[r]_{F_{R_0}} &
R_0 \ar[r]_{F_{R_0}} & \cdots.
}}
\end{equation}
Thus the assertion follows. 
\end{proof}

For a perfectoid tower $(\{R_{i}\}_{i\geq 0}, \{t_{i}\}_{i\geq 0})$ arising from $(R, I_{0})$, an ideal $I_{1}\subseteq R_{1}$ appearing in the axiom (f) in Definition \ref{stilt} is unique. Indeed, it contains $I_{0}R_{1}$, and  its image via the projection 
$R_{1}\to \overline{R_1}$ is a fixed ideal $\ker(F_{0})$. 

\begin{definition}
We call $I_{1}$ the \emph{first perfectoid pillar of $(\{R_{i}\}_{i\geq 0}, \{t_{i}\}_{i\geq 0})$ arising from $(R, I_{0})$}. 
\end{definition}

The relationship between $I_{0}$ and $I_{1}$ can be observed also in higher layers (see Proposition \ref{fn-generating} below). 
In the rest of this section, we fix a perfectoid tower $(\{R_{i}\}_{i \geq 0}, \{t_i\}_{i \geq 0})$ arising from some pair $(R, I_{0})$, and let $I_{1}$ denote the first perfectoid pillar. 

\begin{proposition}
\label{fn-generating}
The following assertions hold. 
\begin{enumerate}
\item
For a sequence of principal ideals $\{I_i \subseteq R_i\}_{i \geq 2}$, the following conditions are equivalent. 
\begin{enumerate}
\item
$F^{-1}_{i}({I_i}\overline{R_i})=I_{i+1}\overline{R_{i+1}}$ for every $i\geq 0$. 

\item
$F_i(I_{i+1}\overline{R_{i+1}}) = I_i\overline{R_{i}}$ for every $i\geq 0$. 

\end{enumerate}

\item
Each one of the equivalent conditions in (1) implies that $I_{i+1}^p = I_{i}R_{i+1}$ for every $i\geq 0$.

\item
There exists a unique sequence of principal ideals $\{I_i \subseteq R_i\}_{i \geq 0}$ that satisfies one of the equivalent conditions in (1). Moreover, there exists a sequence of elements 
$\{\overline{f_i}\in \overline{R_i}\}_{i\geq 0}$ such that $I_{i}\overline{R_i}=(\overline{f_i})$ and $F_i(\overline{f_{i+1}})=\overline{f_i}$ for every $i\geq 0$. 
\end{enumerate}
\end{proposition}

\begin{proof}
(1): 
Since the implication $(\textnormal{a})\Rightarrow (\textnormal{b})$ follows from the axiom (d) in Definition \ref{stilt}, it suffices to show the converse. 
Assume that the condition (b) is satisfied. 
Then for every $i\geq 0$, the compatibility 
$\overline{t_i}\circ F_i=F_{\overline{R_{i+1}}}$ implies 
\begin{equation}\label{eq302Thu}
I_{i+1}^p\overline{R_{i+1}}=I_i\overline{R_{i+1}}
\end{equation}
because $I_{i+1}$ is principal. 
In particular, $\ker (F_i)=I_{1}\overline{R_{i+1}}\subseteq I_{i+1}\overline{R_{i+1}}$ (cf.\ the axiom (f-2)). 
On the other hand, by the surjectivity of $F_i$ and the assumption again, we have $F_i(F^{-1}_i({I_i}\overline{R_i}))={I_i}\overline{R_i}=F_i(I_{i+1}\overline{R_{i+1}})$. 
Hence 
$$
F^{-1}_i({I_i}\overline{R_i})\subseteq {I_{i+1}}\overline{R_{i+1}}+\ker (F_i)
\subseteq I_{i+1}\overline{R_{i+1}}\subseteq F^{-1}_i({I_i}\overline{R_i}), 
$$
which yields the assertion. 

(2):
Let us deduce the assertion from (\ref{eq302Thu}) by induction. 
By definition, $I^{p}_{1}=I_0R_1$. 
We then fix some $i\geq1$. Suppose that for every $1\leq k\leq i$, $I^{p}_{k}=I_{k-1}R_{k}$. 
Then $I_0R_{i}= I^{p^i}_i$. In particular, $R_i$ is $I_i$-adically Zariskian by the axiom (e). 
Moreover, by (\ref{eq302Thu}), we have the equalities of $R_i$-modules: 
$$
I_iR_{i+1}=I_{i+1}^p + I_0R_{i+1}=I_{i+1}^p + I_i^{p^i-1}(I_iR_{i+1}). 
$$
Hence by Nakayama's lemma, we obtain $I^p_{i+1}=I_iR_{i+1}$ as desired.

(3): By the axiom of (dependent) choice, the existence follows from the axiom (d) in Definition \ref{stilt}. 
Let us show the uniqueness of $\{I_i \subseteq R_i\}_{i \geq 0}$ that satisfies the condition (a) in (1). 
For every $i\geq 0$, $I_{i}R_{i+1} \subseteq I_{i+1}$ by (2), and hence $I_{i+1}$ is the inverse image of $F_{i}^{-1}(I_{i}\overline{R_{i}})$ via the projection $R_{i+1}\to \overline{R_{i+1}}$. Since $I_{0}$ is fixed, the assertion follows. 
\end{proof}

\begin{definition}
In the situation of Proposition \ref{fn-generating} (3), 
we call $I_{i}$ the \emph{$i$-th perfectoid pillar of $(\{R_{i}\}_{i\geq 0}, \{t_{i}\}_{i\geq 0})$ arising from $(R_{0}, I_{0})$}.  
\end{definition}

The following property of perfectoid pillars is applied to prove our main result. 

\begin{lemma}\label{1018TueN}
Let $\{I_i\}_{i\geq 0}$ denote the system of perfectoid pillars of $(\{R_{i}\}_{i\geq 0}, \{t_{i}\}_{i\geq 0})$, and 
let $\pi_i: R_i/I_0R_i\to R_i/I_iR_i$ $(i\geq 0)$ be the natural projections. 
Then for every $i\geq 0$, there exists a unique isomorphism of rings: 
$$
F'_i: R_{i+1}/I_{i+1}R_{i+1}\xrightarrow{\cong}R_i/I_iR_i
$$
such that 
$\pi_i \circ F_i=F'_i\circ \pi_{i+1}$. 
\end{lemma}
\begin{proof}
Since $F_i$ and $\pi_i$ are surjective, the assertion immediately follows from $\ker(\pi_i\circ F_i)=F_i^{-1}({I_i}(R_i/I_0R_i))=I_{i+1}(R_{i+1}/I_0R_{i+1})$. 
\end{proof}

\subsection{Tilts of perfectoid towers}\label{subsecTilt}
\subsubsection{Invariance of some properties}
Here we establish tilting operation for perfectoid towers. For this, we first introduce the notion of \emph{small tilt}, which originates in \cite{Sh11}. 
\begin{definition}[Small tilts]\label{defstilt}
Let $(\{R_i\}_{i\geq 0},\{t_i\}_{i \geq 0})$ be a perfectoid tower arising from some pair $(R, I_0)$.  
\begin{enumerate}
\item
For any $j\geq 0$, we define the \textit{j-th small tilt} of $(\{R_{i}\}_{i \geq 0}, \{t_i\}_{i \geq 0})$ associated to $(R,I_0)$ as the $j$-th inverse quasi-perfection of $(\{R_{i}\}_{i \geq 0}, \{t_i\}_{i \geq 0})$ associated to $(R, I_0)$ and denote it by $(R_j)_{I_0}^{s.\flat}$.
%in distinction from $(R_j)_{I_0}^{q.\textnormal{frep}}$. 
\item
Let the notation be as in Lemma \ref{1018TueN}. 
Then we define $I^{s.\flat}_i:=\ker (\pi_{i}\circ \Phi_0^{(i)})$ for every $i\geq 0$.  
\end{enumerate}
\end{definition}
%By definition, $\Phi_{0}^{i}$ induces an isomorphism $R^{s.\flat}_{i}/I^{s.\flat}_{i}\xrightarrow{\cong }R_{i}/I_{i}$. 

Note that the ideal $I^{s.\flat}_{i}\subseteq R^{s.\flat}_{i}$ has the following property. 

\begin{lemma}\label{129SunN}
Keep the notation as in Definition \ref{defstilt}. 
Then for every $i\geq 0$ and $j\geq 0$, we have $\Phi^{(j)}_i(I^{s.\flat}_j)=I_{j+i}\overline{R_{j+i}}$. 
\end{lemma}
\begin{proof}
Since $\Phi^{(j)}_{0}$ is surjective, we have $\Phi^{(j)}_{0}(I^{s.\flat}_{j})=I_{j}\overline{R_j}$. 
On the other hand, since $\Phi^{(j)}_0=F_j\circ\Phi^{(j)}_1$, we have 
$$
F_j^{-1}(\Phi^{(j)}_{0}(I^{s.\flat}_{j}))=\Phi^{(j)}_{1}(I^{s.\flat}_{j})+\ker (F_{j})=\Phi^{(j)}_{1}(I^{s.\flat}_{j}). 
$$
Hence by the condition (a) in Proposition \ref{fn-generating} (1), $\Phi^{(j)}_{1}(I^{s.\flat}_{j})=I_{j+1}\overline{R_{j+1}}$. 
By repeating this procedure recursively, we obtain the assertion. 
\end{proof}

The next lemma provides some completeness of the small tilts attached to a perfectoid tower of characteristic $p>0$ (see also Remark \ref{303FrN}).

\begin{lemma}
\label{stiltcompletion}
Let $(\{R_i\}_{i \geq 0}, \{t_i\}_{i \geq 0})$ be a perfectoid tower arising from $(R,(0))$.
Then, for any element $f \in R$ and any $j \geq 0$, the inverse limit $\vpl \{ \cdots \xrightarrow{\overline{F_{j+1}}} R_{j+1}/fR_{j+1} \xrightarrow{\overline{F_j}} R_j/fR_j \}$ is isomorphic to the $f$-adic completion of $R_j$.
\end{lemma}

\begin{proof}
It suffices to show the assertion when $j=0$. 
Let $(\{R'_{i}\}_{i\geq 0}, \{t'_i\}_{i\geq 0})$ denote the standard perfect tower (\ref{FrobTower}) arising from $R$. 
By Lemma \ref{perfdperftower}, (\ref{6430TueN}) gives a canonical isomorphism $(\{R_i\}_{i \geq 0}, \{t_i\}_{i \geq 0})\xrightarrow{\cong} (\{R'_{i}\}_{i\geq 0}, \{t'_i\}_{i\geq 0})$. 
If we put $J_{0}= fR'_{0}$, then $R'_i/J_{0}R'_i = R/f^{p^{i}}R$ for every $i\geq 0$. 
Hence we have canonical isomorphisms: 
$$
\vpl \{ \cdots \xrightarrow{\overline{F_{1}}} R_{1}/fR_{1} \xrightarrow{\overline{F_0}} R_0/fR_0 \} \xrightarrow{\cong} \varprojlim_{n \geq 0} R/f^{p^n}R \xrightarrow{\cong} \varprojlim_{n \geq 0}R / f^nR, 
$$
as desired. 
\end{proof}

\begin{example}\label{widehatrmk}
Let $S$ be a perfect $\mathbb{F}_{p}$-algebra. 
Pick an arbitrary $f\in S$, and let $\widehat{S}$ denote the $f$-adic completion. 
Applying the above proof to the tower 
$$
S\xrightarrow{\textnormal{id}_{S}}S\xrightarrow{\textnormal{id}_{S}}S\xrightarrow{\textnormal{id}_{S}}\cdots, 
$$
we obtain a canonical isomorphism $\varprojlim_{\textnormal{Frob}}S/fS\xrightarrow{\cong}\widehat{S}$. 
\end{example}

\begin{remark}\label{303FrN}
In the situation of Lemma \ref{stiltcompletion}, assume further that $\overline{t_i} : R_i/fR_i \to R_{i+1}/fR_{i+1}$ is injective for every $i\geq 0$. 
Then $(\{R_i\}_{i \geq 0},\{t_i\}_{i \geq 0})$ is a purely inseparable tower arising from $(R,(f))$ with Frobenius projections $\{\overline{F_i}:  R_{i+1}/fR_{i+1} \to R_i / fR_i\}_{i \geq 0}$. 
Furthermore, it satisfies the axioms (d), (f), and (g) in Definition \ref{stilt}. 
To check this, we may assume that $(\{R_i\}_{i \geq 0},\{t_i\}_{i \geq 0})$ is the standard perfect tower (\ref{FrobTower}). 
Then $\overline{F_i}$ is the natural projection $R/f^{p^{i+1}}R\twoheadrightarrow R/f^{p^{i}}R$. It is clearly surjective, and its kernel is $f^{p^{i}}(R/f^{p^{i+1}}R)$. 
Let $I_i$ be the ideal of $R_i$ generated by $f \in R_i$ ($=R$). 
Then $I_0R_{i} = f^{p^i}R$ and $I_1R_{i+1} = f^{p^i}R$. Hence $I^p_1 = I_{0}R_{1}$ and $\ker(\overline{F_i}) = I_1\overline{R_{i+1}} $. 
Finally, note that $(R_i)_{I_0\textnormal{-tor}} = R_{f\textnormal{-tor}}$. 
Then $I_{0}(R_{i})_{I_0\textnormal{-tor}} = f^{p^i}R_{f\textnormal{-tor}} = (0)$ by Lemma \ref{2294N}, and we can take $\id_{R_{f\textnormal{-tor}}}$ as the bijection 
$(\overline{F_i})_{\textnormal{tor}}$ fitting into the diagram (\ref{24519N}). 
\end{remark}

Now we define \emph{tilts of perfectoid towers}.

\begin{definition}[Tilts of perfectoid towers]\label{tiltperfdtower}
Let $(\{R_{i}\}_{i \geq 0}, \{t_i\}_{i \geq 0})$ be a perfectoid tower arising from some pair $(R,I)$. 
Then we define \textit{the tilt of $(\{R_{i}\}_{i \geq 0}, \{t_i\}_{i \geq 0})$} associated to $(R, I)$ as the inverse perfection of $(\{R_{i}\}_{i \geq 0}, \{t_i\}_{i \geq 0})$ associated to $(R, I)$, 
and denote it by $(\{(R_{i})^{s.\flat}_{I}\}_{i \geq 0}, \{(t_i)^{s.\flat}_{I}\}_{i \geq 0})$.
%in distinction from 
%$(\{(R_{i})^{q.\textnormal{frep}}_{I}\}_{i \geq 0}, \{(t_i)^{q.\textnormal{frep}}_{I}\}_{i \geq 0})$. 
If no confusion occurs, we also denote by $R^{s.\flat}_{i}$ (resp.\ $t^{s.\flat}_{i}$) the symbol $(R_{i})^{s.\flat}_{I}$ (resp.\ $(t_i)^{s.\flat}_I$) as an abbreviated form. \end{definition}

After discussing several basic properties of this tilting operation, we illustrate how to compute the tilts of perfectoid towers in some specific cases; when they consist of \textit{log-regular rings} (see Theorem \ref{TiltingLogRegular} and Example \ref{examplelogstilt}). 

We should remark that all results on the perfection of purely inseparable towers (established in \S\ref{PureInsep}) can be applied to the tilts of perfectoid towers.

%In the rest of this subsection, we fix a perfectoid tower $(\{R_i\}_{i\geq 0}, \{t_i\}_{i\geq 0})$ arising from some pair $(R, I_{0})$, and denote by $\{I_i\}_{i\geq 1}$ the system of perfectoid pillars. 
%Moreover, we also denote by $(\{R^{s.\flat}_{i}\}_{i\geq 0}, \{t^{s.\flat}_{i}\}_{i\geq 0})$ the tilt $(\{(R_{i})^{s.\flat}_I\}_{i \geq 0}, \{(t_i)^{s.\flat}_I\}_{i \geq 0})$ as an abbreviated form. 

Let us state Main Theorem \ref{mt1} in a more refined form. 
This is an important tool when one wants to see that a certain correspondence holds between Noetherian rings of mixed characteristic and those of positive characteristic.

\begin{theorem}
\label{exactstilt}
Let $(\{R_i\}_{i\geq 0}, \{t_i\}_{i\geq 0})$ be a perfectoid tower arising from some pair $(R, I_{0})$, and let $\{I_i\}_{i\geq 0}$ be the system of perfectoid pillars. 
Let $(\{R^{s.\flat}_{i}\}_{i\geq 0}, \{t^{s.\flat}_{i}\}_{i\geq 0})$ denote the tilt associated to $(R, I_{0})$. 
Then the following assertions hold. 
\begin{enumerate}
\item
For every $j\geq 0$ and every element $f^{s.\flat}_{j}\in R^{s.\flat}_{j}$, the following conditions are equivalent. 
\begin{enumerate}
\item
$f_{j}^{s.\flat}$ is a generator of $I_{j}^{s.\flat}$. 
\item
For every $i\geq 0$, $\Phi^{(j)}_{i}(f^{s.\flat}_{j})$ is a generator of $I_{j+i}\overline{R_{j+i}}$. 
\end{enumerate}
In particular, $I^{s.\flat}_j$ is a principal ideal, and $(I^{s.\flat}_{j+1})^{p}=I^{s.\flat}_jR_{j+1}^{s.\flat}$. 
\item
We have isomorphisms of (possibly) non-unital rings $(R_j^{s.\flat})_{I^{s.\flat}_0\textnormal{-tor}}\cong (R_j)_{I_{0}\textnormal{-tor}}$ that are compatible with $\{t_j\}_{j\geq 0}$ and $\{t^{s.\flat}_j\}_{j\geq 0}$. 
\end{enumerate}
\end{theorem}

We give its proof in the subsequent \S\ref{SSSectMT1}. Before that, let us observe that this theorem induces many good properties of tilting. 
In the rest of this subsection, we keep the notation as in Theorem \ref{exactstilt}.

\begin{lemma}\label{1215ThuN}
For every $i\geq 0$, $R^{s.\flat}_i$ is $I^{s.\flat}_{0}$-adically complete and separated. 
\end{lemma}
\begin{proof}
By Theorem \ref{exactstilt}, the ideal $I_0^{s.\flat}R^{s.\flat}_i\subseteq R^{s.\flat}_i$ is principal. Hence one can apply Proposition \ref{invqperfprop} (1)-(a) to deduce the assertion. 
\end{proof}

To discuss perfectoidness for the tilt $(\{R^{s.\flat}_{i}\}_{i\geq 0}, \{t^{s.\flat}_{i}\}_{i\geq 0})$, we introduce the following maps.

\begin{definition}\label{Fsflat}
For every $i\geq 0$, we define a ring map $(F_{i})^{s.\flat}_{I_0} : (R_{i+1})^{s.\flat}_{I_0}/I^{s.\flat}_0(R_{i+1})^{s.\flat}_{I_0} \to (R_{i})^{s.\flat}_{I_0}/I_0^{s.\flat}(R_{i})^{s.\flat}_{I_0}$ by the rule: 
$$
(F_{i})^{s.\flat}_{I_0}(\alpha_{i+1}\mod I^{s.\flat}_{0}(R_{i+1})_{I_0}^{s.\flat}) = (F_i)^{q.\textnormal{frep}}_{I_0}(\alpha_{i+1})\mod I^{s.\flat}_{0}(R_i)_{I_0}^{s.\flat}
$$
where $\alpha_{i+1}\in (R_{i+1})_{I_0}^{s.\flat}$. 
If no confusion occurs, we also denote by $F^{s.\flat}_{i}$  the symbol $(F_{i})^{s.\flat}_{I_0}$ 
as an abbreviated form. 
\end{definition}
\begin{remark}\label{rm304SatN}
Although the symbols $(\ \cdot\ )^{s.\flat}$ and 
$(\ \cdot\ )^{q.\textnormal{frep}}$ had been used interchangeably before Definition \ref{Fsflat}, $(F_{i})^{s.\flat}_{I_0}$ differs from $(F_i)^{q.\textnormal{frep}}_{I_0}$ in general.
\end{remark}

The following lemma is an immediate consequence of Theorem \ref{exactstilt} (1), but quite useful.
\begin{lemma}\label{921WedN}
For every $j\geq 0$, $\Phi^{(j)}_0$ induces an isomorphism 
\begin{equation}\label{1212MonN}
\overline{\Phi^{(j)}_0}: R^{s.\flat}_j/I^{s.\flat}_0R^{s.\flat}_j\xrightarrow{\cong} R_j/I_{0}R_j;\ \  a\mod I^{s.\flat}_0R^{s.\flat}_j \mapsto 
\Phi^{(j)}_{0}(a). 
\end{equation}
Moreover, $\{\overline{\Phi^{(i)}_0}\}_{i\geq 0}$ is compatible with 
$\{t_i\}_{i\geq 0}$ (resp.\ $\{F_{R^{s.\flat}_i/I^{s.\flat}_0R^{s.\flat}_i}\}_{i\geq 0}$, resp.\ $\{F^{s.\flat}_i\}_{i\geq 0}$) and $\{t^{s.\flat}_i\}_{i\geq 0}$ (resp.\ $\{F_{R_{i}/I_{0}R_{i}}\}_{i\geq 0}$, resp.\ $\{F_i\}_{i\geq 0}$). 
\end{lemma}
\begin{proof}
By the axiom (d) in Definition \ref{stilt}, (\ref{1212MonN}) is surjective. 
Let us check the injectivity.
By Theorem \ref{exactstilt} (1), $I^{s.\flat}_{0}$ is generated by an element $f^{s.\flat}_{0}\in R^{s.\flat}_{0}$ such that $\Phi^{(0)}_{i}(f^{s.\flat}_{0})$ is a generator of $I_{i}\overline{R_i}$ ($i\geq 0$). 
Note that $(\{R_{j+i}\}_{i\geq 0}, \{t_{j+i}\}_{i\geq 0})$ is a perfectoid tower arising from $(R_{j}, I_{0}R_{j})$. 
Moreover, $\{I_{i}R_{j+i}\}_{i\geq 0}$ is the system of perfectoid pillars associated to $(R_{j}, I_{0}R_{j})$ (cf. the condition (b) in Proposition \ref{fn-generating} (1)). 
Put $J_{0}:=I_{0}R_{j}$. 
Then by Theorem \ref{exactstilt} (1) again, we find that $J^{s.\flat}_{0}=f^{s.\flat}_{0}R^{s.\flat}_{j}=I^{s.\flat}_{0}R_{j}^{s.\flat}$. 
Since $J^{s.\flat}_{0}=\ker \Phi^{(j)}_0$, we obtain the first assertion. 

One can deduce that $\{\overline{\Phi^{(i)}_0}\}_{i\geq 0}$ is compatible with the Frobenius projections from the commutativity of 
(\ref{frobeniusdecompose}), because the other compatibility assertions immediately follow from the construction. 
\end{proof}

\begin{remark}\label{rMonNm306}
Theorem \ref{exactstilt} (2) and Lemma \ref{921WedN} can be interpreted as a correspondence of homological invariants between $R_i$ and $R_i^{s.\flat}$ by using Koszul homologies. 
Indeed, for any generator $f_0$ (resp. $f_0^{s.\flat}$) of $I_0$ (resp. $I_0^{s.\flat}$), the Koszul homology $H_q(f_0^{s.\flat}; R_i^{s.\flat})$ is isomorphic to $H_q(f_0; R_i)$ for any $q \geq 0$ as an abelian group.\footnote{Note that $(R_i)_{I_0\textnormal{-tor}}=\Ann_{R_i}(I_0)$ by the axiom (g), and $(R_i^{s.\flat})_{I_0^{s.\flat}\textnormal{-tor}}=\Ann_{R_i^{s.\flat}}(I_0^{s.\flat})$ by Corollary \ref{cor306MonN}. } 
\end{remark}

Now we can show the invariance of several properties of perfectoid towers under tilting. 
The first one is perfectoidness, which is most important in our framework. 

\begin{proposition}
\label{stiltstilt}
$(\{R_{i}^{s.\flat}\}_{i \geq 0},\{t_{i}^{s.\flat}\}_{i \geq 0})$ is a perfectoid tower arising from $(R_0^{s.\flat}, I^{s.\flat}_0)$. 
\end{proposition}

\begin{proof}
By Lemma \ref{921WedN} and Remark \ref{303FrN}, $(\{R_{i}^{s.\flat}\}_{i \geq 0},\{t_{i}^{s.\flat}\}_{i \geq 0})$ is a purely inseparable tower arising from $(R_0^{s.\flat}, I^{s.\flat}_0)$ that also satisfies the axioms (d), (f), and (g). 
Moreover, the axiom (e) holds by Lemma \ref{1215ThuN}. 
Hence the assertion follows. 
\end{proof}

Next we focus on finiteness properties. ``Small" in the name of small tilts comes from the following fact. 

\begin{proposition}
\label{smalltiltproperty2}
For every $j\geq 0$, the following assertions hold. 
\begin{enumerate}
\item
If $t_{j}: R_{j}\to R_{j+1}$ is module-finite, then so is $t^{s.\flat}_{j}: R^{s.\flat}_{j}\to R^{s.\flat}_{j+1}$. 
Moreover, the converse holds true when $R_j$ is $I_0$-adically complete and separated. 

\item
If $R_j$ is a Noetherian ring, then so is $R_j^{s.\flat}$. 
Moreover, the converse holds true when $R_j$ is $I_0$-adically complete and separated. 
\item
Assume that $R_j$ is a Noetherian local ring, and a generator of $I_0R_{j}$ is regular.
Then the dimension of $R_j$ is equal to that of $R_j^{s.\flat}$. 

\end{enumerate}
\end{proposition}

\begin{proof}
(1): 
By Lemma \ref{921WedN}, $\overline{t_j}: R_j/I_0R_j\to R_{j+1}/I_0R_{j+1}$ is module-finite if and only if 
$\overline{t^{s.\flat}_j}: R_j^{s.\flat}/I_0^{s.\flat}R_j^{s.\flat}\to R_{j+1}^{s.\flat}/I_0^{s.\flat}R_{j+1}^{s.\flat}$ is so. Thus by Lemma  \ref{1215ThuN} and \cite[Theorem 8.4]{Ma86}, the assertion follows. 

(2):
One can prove this assertion by applying Lemma \ref{1215ThuN}, Lemma \ref{921WedN}, and \cite[Tag 05GH]{Stacks}.

(3):
By Theorem \ref{exactstilt}, $I_0^{s.\flat}R^{s.\flat}_j$ is also generated by a regular element.
Thus we obtain the equalities $\dim R_j = \dim R_j/I_0R_j +1$ and $\dim R_j^{s.\flat} = \dim R_j^{s.\flat}/I_0^{s.\flat}R_j^{s.\flat} + 1$.
By combining these equalities with Lemma \ref{921WedN}, we deduce assertion. 
\end{proof}

Proposition \ref{smalltiltproperty2} (2) says that ``Noetherianness'' for a perfectoid tower is preserved under tilting. 

%The third is the invariance of the Noetherianness of perfectoid towers.
%One of the reasons for establishing the theory of perfectoid towers is to develop a method of applying perfectoid theory to Noetherian rings.
%Now, 
%Before that, we give the definition of Noetherian towers.

\begin{definition}\label{Noethtower}
We say that $(\{R_i \}_{i \geq 0} , \{ t_i\}_{i \geq 0})$ is \textit{Noetherian} if $R_i$ is Noetherian for each $i \geq 0$.
\end{definition}

\begin{corollary}
If $(\{R_i \}_{i \geq 0} , \{ t_i\}_{i \geq 0})$ is Noetherian, then so is the tilt $(\{R_i^{s.\flat}\}_{i \geq 0}, \{t_i^{s.\flat}\}_{i \geq 0} )$.
Moreover, the converse holds true when $R_i$ is $I_0$-adically complete and separated for each $i \geq 0$. 
\end{corollary}
\begin{proof}
It immediately follows from Proposition \ref{smalltiltproperty2} (2). 
\end{proof}

%The second is the equivalence of categories of finite \'etale algebras over each layer of perfectoid towers and of their tilts.

Finally, let us consider perfectoid towers of henselian rings. Then we obtain the equivalence of categories of finite \'etale algebras over each layer.

\begin{proposition}\label{finiteetaleequivalence}
Assume that $R_{i}$ is $I_{0}$-adically Henselian for any $i \geq 0$. 
Then we obtain the following equivalences of categories:
$$
\bFEt(R_i^{s.\flat}) \xrightarrow{\cong} \bFEt(R_i).
$$
\end{proposition}
\begin{proof}
This follows from Lemma \ref{1215ThuN}, Lemma \ref{921WedN} and \cite[Tag 09ZL]{Stacks}.
%By Theorem \ref{exactstilt} and \cite[Tag 09ZL]{Stacks}, we obtain the desired equivalences of categories.
\end{proof}

\subsubsection{Proof of Main Theorem \ref{mt1}}\label{SSSectMT1}
We keep the notation as above. Furthermore, we set 
$\overline{I_i}:=I_i\overline{R_i}$ for every $i\geq 0$. 
To prove Theorem \ref{exactstilt}, we investigate some relationship between $(R_{i})_{I_{0}\textnormal{-tor}}$ and $\Ann_{\overline{R_{i}}}(\overline{I_{i}})$. 
First recall that we can regard $(R_{i})_{I_{0}\textnormal{-tor}}$ as a non-unital subring of $\overline{R_i}$ by Corollary \ref{96TueN}. 
Moreover, the map $\overline{t_{i}}$ naturally restricts to $(R_{i})_{I_{0}\textnormal{-tor}}\hookrightarrow (R_{i+1})_{I_{0}\textnormal{-tor}}$, as follows.

\begin{lemma}\label{1120SunN}
For every $i\geq 0$, let $(t_i)_{\textnormal{tor}}: (R_{i})_{I_{0}\textnormal{-tor}}\to (R_{i+1})_{I_{0}\textnormal{-tor}}$ be the restriction of $t_{i}$. Then the following assertions hold. 
\begin{enumerate}
\item
$(t_i)_{\textnormal{tor}}$ is the unique map such that 
$\varphi_{I_{0}, R_{i+1}}\circ (t_i)_{\textnormal{tor}}=\overline{t_{i}}\circ\varphi_{I_{0}, R_i}$. 
\item
$(t_i)_{\textnormal{tor}}\circ (F_i)_{\textnormal{tor}}=(F_{i+1})_{\textnormal{tor}}\circ (t_{i+1})_{\textnormal{tor}}=F_{(R_{i+1})_{I_0\textnormal{-tor}}}$. 
\end{enumerate}
\end{lemma}
\begin{proof}
Since $\varphi_{I_{0}, R_i}$ is injective by Corollary \ref{96TueN}, the assertion (1) is clear from the construction. 
Hence we can regard $(t_i)_\textnormal{tor}$ and $(F_{i})_\textnormal{tor}$ as the restrictions of $\overline{t_{i}}$ and $F_{i}$,  respectively. 
Thus the assertion (2) follows from the compatibility $\overline{t_{i}}\circ F_{i}=F_{i+1}\circ\overline{t_{i+1}}=F_{\overline{R_{i+1}}}$ induced by Lemma \ref{secondCD} (3). 
\end{proof}

The map $\varphi_{I_{0}, R_i}: (R_{i})_{I_{0}\textnormal{-tor}}\hookrightarrow R_{i}/I_{0}R_{i}$ restricts to $\Ann_{R_{i}}(I_{i})\hookrightarrow \Ann_{\overline{R_{i}}}(\overline{I_{i}})$. On the other hand, $\Ann_{R_{i}}(I_{i})$ turns out to be equal to $(R_{i})_{I_{0}\textnormal{-tor}}$ by the following lemma. 

\begin{lemma}\label{1121MonN}
For every $i\geq 0$, $I_{i}(R_i)_{I_0\textnormal{-tor}}=0$. In particular, 
$\im(\varphi_{I_0, R_i})\subseteq \Ann_{\overline{R_{i}}}(\overline{I_{i}})$. 
\end{lemma}
\begin{proof}
By Lemma \ref{1120SunN} (2) and the axiom (g) in Definition \ref{stilt}, we find that $F_{(R_i)_{I_0\textnormal{-tor}}}$ is injective. 
In other words, $(R_{i})_{I_{0}\textnormal{-tor}}$ does not contain any non-zero nilpotent element. 
Moreover, $(R_{i})_{I_{0}\textnormal{-tor}}=(R_{i})_{I_i\textnormal{-tor}}$. Hence the assertion follows from Lemma \ref{2294N}. 
\end{proof}

The following lemma is essential for proving Theorem \ref{exactstilt}. 

\begin{lemma}\label{lem228242027}
For every $i\geq 0$, $F_{i}$ restricts to a $\mathbb{Z}$-linear map $\Ann_{\overline{R_{i+1}}}(\overline{I_{i+1}})\to \Ann_{\overline{R_i}}(\overline{I_{i}})$. Moreover, the resulting inverse system $\{\Ann_{\overline{R_i}}(\overline{I_i})\}_{i\geq 0}$ has the following properties. 
\begin{enumerate}
\item
For every $j\geq 0$, 
$\varprojlim_{i\geq 0}^1\Ann_{\overline{R_{j+i}}}(\overline{I_{j+i}})=(0)$. 
\item
There are isomorphisms of $\mathbb{Z}$-linear maps $\varprojlim_{i\geq 0}\Ann_{\overline{R_{j+i}}}(\overline{I_{j+i}})\cong (R_j)_{I_{0}\textnormal{-tor}}$ $(j\geq 0)$ that are multiplicative, and compatible with 
$\{t^{s,\flat}_{j}\}_{j\geq 0}$ and $\{t_j\}_{j\geq 0}$. 
\end{enumerate}
\end{lemma}
\begin{proof}
Since $F_{i}(\overline{I_{i+1}})=\overline{I_{i}}$, $F_{i}$ restricts to a $\mathbb{Z}$-linear map $(F_{i})_{\textnormal{ann}}: \Ann_{\overline{R_{i+1}}}(\overline{I_{i+1}})\to \Ann_{\overline{R_i}}(\overline{I_{i}})$. 
Let $\varphi_i: (R_{i})_{I_{0}\textnormal{-tor}}\hookrightarrow \Ann_{\overline{R_{i}}}(\overline{I_{i}})$ be the restriction of $\varphi_{I_0, R_i}$. By Lemma \ref{1230FriN} and Lemma \ref{1121MonN}, we can write 
$\Ann_{\overline{R_i}}(\overline{I_i})=\im(\varphi_i)+\overline{I_{i}}^{p^i-1}$. 
Moreover, $\im(\varphi_i)\cap\overline{I_{i}}^{p^i-1}=(0)$ by Lemma \ref{102Sun} and Lemma \ref{1121MonN}. 
Hence we have the following ladder with exact rows: 
\begin{equation}\label{20220824}
\vcenter{
\xymatrix{
0\ar[r]&(R_{i+1})_{I_{0}\textnormal{-tor}}\ar[r]^{\varphi_{i+1}}\ar[d]^{(F_{i})_{\textnormal{tor}}}&\Ann_{\overline{R_{i+1}}}(\overline{I_{i+1}})\ar[r]\ar[d]&\overline{I_{i+1}}^{p^{i+1}-1}\ar[r]\ar[d]&0\\
0\ar[r]&(R_i)_{I_{0}\textnormal{-tor}}\ar[r]^{\varphi_i}&\Ann_{\overline{R_i}}(\overline{I_i})\ar[r]&\overline{I_{i}}^{p^i-1}\ar[r]&0
}}
\end{equation}
where the second and third vertical maps are the restrictions of $F_{i}$. 
Since $F_i(\overline{I_{i+1}}^{p^{i+1}-1})=0$, both functors $\varprojlim_{i\geq 0}$ and 
$\varprojlim^{1}_{i\geq 0}$ assign $(0)$ to the inverse system $\{\overline{I_{j+i}}^{p^{j+i}-1}\}_{i\geq 0}$. 
Moreover, since $(F_{i})_\textnormal{tor}$ is bijective, $\varprojlim_{i\geq 0}(R_{j+i})_{I_{0}\textnormal{-tor}}\cong (R_{j})_{I_{0}\textnormal{-tor}}$ and 
$\varprojlim^{1}_{i\geq 0}(R_{j+i})_{I_{0}\textnormal{-tor}}=(0)$. 
Hence we find that $\varprojlim_{i\geq 0}^{1}\Ann_{\overline{R_{j+i}}}(\overline{I_{j+i}})=(0)$, 
which is the assertion (1). 
Furthermore, we obtain the isomorphisms of $\mathbb{Z}$-modules: 
\begin{equation}\label{1130WedN}
(R_j)_{I_{0}\textnormal{-tor}}\xleftarrow{(\Phi^{(j)}_{0})_{\textnormal{tor}}}\varprojlim_{i\geq 0}(R_{j+i})_{I_{0}\textnormal{-tor}}\xrightarrow{\varprojlim_{i\geq 0}\varphi_{j+i}}
\varprojlim_{i\geq 0}\Ann_{\overline{R_{j+i}}}(\overline{I_{j+i}})
\end{equation}
(where $(\Phi^{(j)}_{0})_{\textnormal{tor}}$ denotes the $0$-th projection map), which are also multiplicative. 
Let us deduce (2) from it. 
Since we have $t_{j}^{s.\flat}=\varprojlim_{i\geq 0}\overline{t_{j+i}}$ by definition, the maps $\varprojlim_{i\geq 0}\varphi_{j+i}$ $(j\geq 0)$ 
are compatible with $\{\varprojlim_{i\geq 0}(t_{j+i})_{\textnormal{tor}}\}_{j\geq 0}$ (induced by Lemma \ref{1120SunN} (2)) and $\{t^{s.\flat}_j\}_{j\geq 0}$ by Lemma \ref{1120SunN} (1). 
On the other hand, the projections $(\Phi^{(j)}_{0})_\textnormal{tor}$ $(j\geq 0)$ are compatible with $\{\varprojlim_{i\geq 0}(t_{j+i})_{\textnormal{tor}}\}_{j\geq 0}$ and  $\{(t_j)_{\textnormal{tor}}\}_{j\geq 0}$. 
Hence the assertion follows. 
\end{proof}

Let us complete the proof of Theorem \ref{exactstilt}.

\begin{proof}[Proof of Theorem \ref{exactstilt}]
(1): 
The implication (a) $\Rightarrow$ (b) follows from Lemma \ref{129SunN}. 
Let us show the converse (b)$\Rightarrow$(a). 
For every $i\geq 0$, put $\overline{f_{j+i}}:=\Phi^{(j)}_{i}(f^{s.\flat}_{j})$, and let $\pi_{i}$ and $F'_{i}$ be as in Lemma \ref{1018TueN}. 
Then, by the assumption, we have the following commutative ladder with exact rows: 
\[\xymatrix{
0\ar[r]&(\overline{f_{i+1}})\ar[rr]^{\iota_{i+1}}\ar[d]&& \overline{R_{i+1}}\ar[r]^{\pi_{i+1}}\ar[d]^{F_i}&R_{i+1}/I_{i+1}\ar[r]\ar[d]^{F'_i}&0\\
0\ar[r]&(\overline{f_i})\ar[rr]^{\iota_i}&& \overline{R_i}\ar[r]^{\pi_i}&R_i/I_{i}\ar[r]&0
}\]
where $\iota_i$ is the inclusion map. 
Let us consider the exact sequence obtained by taking inverse limits for all columns of the above ladder. 
Then, since each $F'_{i}$ is an isomorphism, the map $\varprojlim_{i\geq 0}\pi_{j+i}: R^{s.\flat}_{j}\to\varprojlim_{i\geq 0}R_{j+i}/I_{j+i}$ is isomorphic to $\pi_{j}\circ\Phi^{(j)}_{0}$. 
Thus we find that $I^{s.\flat}_j=\im (\varprojlim_{i\geq 0}\iota_{j+i})$. 
Let us show that the ideal $\im (\varprojlim_{i\geq 0}\iota_{j+i})\subseteq R^{s.\flat}_{j}$ is generated by $f^{s.\flat}_{j}$. 
For $i\geq 0$, let $\mu_i: \overline{R_i}\to (\overline{f_i})$ be the $\overline{R_i}$-linear map induced by multiplication by $\overline{f_i}$. 
Then we obtain the commutative ladder: 
\[\xymatrix{
\overline{R_{i+1}}\ar[d]^{F_{i}}\ar[r]^{\mu_{i+1}}&(\overline{f_{i+1}})\ar[r]^{\iota_{i+1}}\ar[d]&\overline{R_{i+1}}\ar[d]^{F_{i}}\\
\overline{R_i}\ar[r]^{\mu_i}&(\overline{f_i})\ar[r]^{\iota_i}&\overline{R_i}. 
}\]
Then, since $\ker \mu_{i}=\Ann_{\overline{R_i}}(\overline{I_i})$ for every $i\geq 0$, $\varprojlim_{i\geq 0}\mu_{j+i}$ is surjective by Lemma \ref{lem228242027} (1). 
Hence we have $\im(\varprojlim_{i\geq 0}\iota_{j+i})=\im(\varprojlim_{i\geq 0}(\iota_{j+i}\circ\mu_{j+i}))$, where the right hand side is the ideal of $R_{j}^{s.\flat}$ generated by $f^{s.\flat}_{j}$. 
Thus we obtain the desired implication.   
Finally, note that by Proposition \ref{fn-generating} (3), we can take a system of elements $\{f^{s.\flat}_j\in R^{s.\flat}_j\}_{j\geq 0}$ satisfying the condition (b) such that $(f^{s.\flat}_{j+1})^{p}=f_{j}^{s.\flat}$ $(j\geq 0)$. 

(2): We have 
$I^{s.\flat}_j(R_j^{s.\flat})_{I^{s.\flat}_j\textnormal{-tor}}=(0)$ by Corollary \ref{cor306MonN}. 
Hence by the assertion (1), 
$$
(R_j^{s.\flat})_{I^{s.\flat}_0\textnormal{-tor}}=
(R_j^{s.\flat})_{I^{s.\flat}_j\textnormal{-tor}}=
\Ann_{R^{s.\flat}_{j}}(I^{s.\flat}_{j})= \ker(\varprojlim_{i\geq 0}\mu_{j+i})= \varprojlim_{i\geq 0}\Ann_{\overline{R_{j+i}}}(\overline{I_{j+i}}). 
$$
Thus by Lemma \ref{lem228242027} (2), we obtain an isomorphism $(R_j^{s.\flat})_{I^{s.\flat}_0\textnormal{-tor}}\cong (R_j)_{I_{0}\textnormal{-tor}}$ with the desired property. 
\end{proof}

\subsection{Relation with perfectoid rings}\label{perfd-perfd}
In the rest of this paper, for a ring $R$, we use the following notation. 
Set the inverse limit 
$$
R^{\flat} := \varprojlim\{\cdots \to R/pR \to R/pR \to \cdots \to R/pR\}, 
$$
where each transition map is the Frobenius endomorphism on $R/pR$.
It is called the \textit{tilt} (or \textit{tilting}) of $R$. 
Moreover, we denote by $W(R)$ the ring of $p$-typical Witt vectors over $R$. 
If $R$ is $p$-adically complete and separated, we denote by $\theta_{R}: W(R^\flat)\to R$ the ring map such that the diagram: 
\begin{equation}\label{1218SunN}
\vcenter{\xymatrix{
W(R^\flat)\ar[r]^{\theta_{R}}\ar[d]&R\ar[d]\\
R^{\flat}\ar[r]&R/pR
}}
\end{equation}
(where the vertical maps are induced by reduction modulo $p$ and the bottom map is the first projection) commutes.
%; 
%see \cite[??]{BMS18} for the existence and the uniqueness of $\theta_{R}$. 

Recall the definition of perfectoid rings. 

\begin{definition}{(\cite[Definition 3.5]{BMS18})}
\label{integralperfectoid}
A ring $S$ is \textit{perfectoid} if the following conditions hold. 
\begin{enumerate}
\item
$S$ is $\varpi$-adically complete and separated for some element $\varpi \in S$ such that $\varpi^{p}$ divides $p$.
\item
The Frobenius endomorphism on $S/pS$ is surjective.
\item
The kernel of $\theta_{S}: W(S^{\flat}) \to S$ is principal.
\end{enumerate}
\end{definition}

We have a connection between perfectoid towers and perfectoid rings. 
To see this, we use the following characterization of perfectoid rings. 

\begin{theorem}[cf.\ {\cite[Corollary 16.3.75]{GR22}}]\label{FontainePerfectoid}
Let $S$ be a ring. Then $S$ is a perfectoid ring if and only if $S$ contains an element $\varpi$ with the following properties. 
\begin{enumerate}
\item
$\varpi^{p}$ divides $p$, and $S$ is $\varpi$-adically complete and separated. 
\item
The ring map $S/\varpi S \to S/\varpi^pS$ induced by the Frobenius endomorphism on $S/\varpi^pS$ is an isomorphism. 
\item
The multiplicative map 
\begin{equation}\label{torsFrob}
S_{\varpi\textnormal{-tor}}\to S_{\varpi\textnormal{-tor}}\ ;\ s\mapsto s^{p}
\end{equation}
is bijective. 
\end{enumerate}
\end{theorem}
\begin{proof}
(``if'' part): It follows from \cite[Corollary 16.3.75]{GR22}. 

(``only if'' part): Let $\varpi\in S$ be as in Definition \ref{integralperfectoid}. Then, such $\varpi$ clearly has the property (1) (in Theorem \ref{FontainePerfectoid}), and also has the property (2) by \cite[Lemma 3.10 (i)]{BMS18}. 
To show the remaining part, we set $\widetilde{S}:=S/S_{\varpi\textnormal{-tor}}$. 
By \cite[\S 2.1.3]{KS20}, the diagram of rings: 
\[\xymatrix{
S\ar[rr]^{\pi_{2}}\ar[d]_{\pi_{1}}&&(S/\varpi S)_{\textnormal{red}}\ar[d]^{\pi_{4}}\\
\widetilde{S}\ar[rr]_{\pi_{3}}&&(\widetilde{S}/\varpi\widetilde{S})_{\textnormal{red}}
}\]
(where $\pi_{i}$ is the canonical projection map for $i=1,2,3,4$) is cartesian. Hence $S_{\varpi\textnormal{-tor}}$ ($=\ker (\pi_{1})$) is isomorphic to $\ker (\pi_{4})$ as a (possibly) non-unital ring. 
Since $(S/\varpi S)_{\textnormal{red}}$ is a perfect $\mathbb{F}_{p}$-algebra, it admits the Frobenius endomorphism and the inverse Frobenius. 
Moreover, $\ker (\pi_{4})$ is closed under these operations because $(\widetilde{S}/\varpi\widetilde{S})_{\textnormal{red}}$ is reduced. 
Consequently, it follows that one has a bijection (\ref{torsFrob}). Hence $\varpi$ has the property (3), as desired.  
\end{proof}

\begin{remark}\label{rmk1218SunN}
In view of the above proof, the ``only if'' part of Theorem \ref{FontainePerfectoid} can be refined as follows. 
For a perfectoid ring $S$, an element $\varpi\in S$ such that $p\in \varpi^{p}S$ and $S$ is $\varpi$-adically complete and separated 
satisfies the properties (2) and (3) in Theorem \ref{FontainePerfectoid}. 
\end{remark}

\begin{corollary}
\label{smalltiltproperty1}
Let $(\{R_{i}\}_{i \geq 0},\{t_i\}_{i \geq 0})$ be a perfectoid tower arising from some pair $(R_0,I_0)$. 
Let $\widehat{R_{\infty}}$ denote the $I_{1}$-adic completion of $R_{\infty}$. 
Then $\widehat{R_\infty}$ is a perfectoid ring. %(where $f_1$ is an element defined as in Definition \ref{stilt} (e)).
\end{corollary}

\begin{proof}
Since we have $\varinjlim_{i\geq 0}F_{R_{i}/I_{0}R_{i}}=(\varinjlim_{i\geq 0}\overline{t_{i}})\circ (\varinjlim_{i\geq 0}F_i)$ and 
$\varinjlim_{i\geq 0}\overline{t_{i}}$ is a canonical isomorphism, 
the Frobenius endomorphism on $\widehat{R_{\infty}}/I_0\widehat{R_{\infty}}$ can be identified with $\varinjlim_{i\geq 0}F_{i}$. 
Hence one can immediately deduce from the axioms in Definition \ref{stilt} that any generator of $I_{1}\widehat{R_{\infty}}$ 
has the all properties assumed on $\varpi$ in Theorem \ref{FontainePerfectoid}. 
\end{proof}

In view of Theorem \ref{FontainePerfectoid}, one can regard perfectoid rings as a special class of perfectoid towers. 

\begin{example}\label{eg1226MonN}
Let $S$ be a perfectoid ring. 
Let $\varpi\in S$ be such that $p\in \varpi^{p} S$ and $S$ is $\varpi$-adically complete and separated. 
Set $S_{i}=S$ and $t_{i}=\id_{S}$ for every $i\geq 0$, and $I_{0}=\varpi^{p}S$. 
Then by Remark \ref{rmk1218SunN}, the tower $(\{S_{i}\}_{i\geq 0}, \{t_{i}\}_{i\geq 0})$ is a perfectoid tower arising from $(S, I_{0})$. 
In particular, $I_{0}S_{I_{0}\textnormal{-tor}}=(0)$, and $F_{S_{I_0\textnormal{-tor}}}$ is bijective. 
\end{example}

Moreover, we can treat more general rings in a tower-theoretic way. 

\begin{example}[Zariskian preperfectoid rings]\label{eg1227TueN}
Let $R$ be a ring that contains an element $\varpi$ such that $p\in \varpi^{p} R$, $R$ is $\varpi$-adically Zariskian, and $R$ has bounded $\varpi$-torsion. 
Assume that the $\varpi$-adic completion $\widehat{R}$ is a perfectoid ring. 
Set $R_{i}=R$ and $t_{i}=\id_{R}$ for every $i\geq 0$, and $I_{0}=\varpi^{p}R$. 
Then the tower $(\{R_{i}\}_{i\geq 0}, \{t_{i}\}_{i\geq 0})$ is a perfectoid tower arising from $(R, I_{0})$. 
Indeed, the axioms (a) and (e) are clear from the assumption. Moreover, since $\widehat{R}$ is perfectoid and $R/\varpi^{p}R\cong \widehat{R}/\varpi^{p}\widehat{R}$, the axioms (b), (c), (d) and (f) hold by Example \ref{eg1226MonN}. 
Similarly, the axiom (g) holds by Lemma \ref{1227TueN} (the map $\psi_{\textnormal{tor}}: R_{I_{0}\textnormal{-tor}}\to (\widehat{R})_{I_{0}\textnormal{-tor}}$ is also an isomorphism of non-unital rings). 
\end{example}

Recall that we have two types of tilting operation at present; one is defined for perfectoid rings, and the other is for perfectoid towers. 
The following result asserts that they are compatible.

\begin{lemma}\label{lem1218SunN}
Let $(\{ R_i^{s.\flat} \}_{i \geq 0} , \{ t_i^{s.\flat} \}_{i \geq 0} )$ be the tilt of $(\{R_{i}\}_{i \geq 0},\{t_i\}_{i \geq 0})$ associated to $(R_{0}, I_0)$. 
Let $\widehat{R_{\infty}^{s.\flat}}$ be the $I^{s.\flat}_0$-adic completion of $R_{\infty}^{s.\flat}:=\varinjlim_{i\geq 0}R_i^{s.\flat}$. 
Let $( I_{0}R_{\infty} )^{\flat}$ be the ideal of $R^{\flat}_{\infty}$ that is the inverse image of $I_{0}R_{\infty} \mod pR_{\infty}$ via the first projection. 
Then there exist canonical isomorphisms 
$$
R_\infty^{\flat}\xleftarrow{\cong}  \varprojlim_{\textnormal{Frob}} R_{\infty}^{s.\flat} /I_{0}^{s.\flat}R^{s.\flat}_{\infty}  \xrightarrow{\cong}\widehat{R_{\infty}^{s.\flat}}
$$
under which $( I_0R_\infty )^{\flat} \subseteq R^{\flat}_{\infty}$ corresponds to $I^{s.\flat}_{0}\widehat{R^{s.\flat}_\infty} \subseteq \widehat{R_{\infty}^{s.\flat}}$. 

\end{lemma}
\begin{proof}
Note that $R_{\infty}^{s.\flat}$ is perfect. 
By Lemma \ref{921WedN} and Example \ref{widehatrmk}, we obtain the following commutative diagram of rings: 
\begin{equation}
\xymatrix{
\varprojlim_{\textnormal{Frob}}R_\infty /I_0R_\infty \ar[d] & \ar[d]\ar[l]_{\cong} \varprojlim_{\textnormal{Frob}} R_{\infty}^{s.\flat} /I_{0}^{s.\flat}R^{s.\flat}_{\infty} \ar[r]^{\ \ \ \ \ \ \cong}&  \widehat{R_{\infty}^{s.\flat}}\ar[d] \\ 
R_\infty/I_0R_\infty& \ar[l]_{\cong}R^{s.\flat}_\infty/I^{s.\flat}_0R^{s.\flat}_\infty\ar@{=}[r]&R^{s.\flat}_\infty/I^{s.\flat}_0R^{s.\flat}_\infty
}
\end{equation}
where the vertical arrows denote the first projection maps. 
By \cite[Lemma 3.2 (i)]{BMS18}, we can identify $R^{\flat}_{\infty}$ with $\varprojlim_{\textnormal{Frob}}R_{\infty}/I_{0}R_{\infty}$, and the ideal $( I_{0}R_{\infty} )^{\flat}\subseteq R_\infty^\flat$ corresponds to the kernel of the leftmost vertical map. 
Thus, since the kernel of the rightmost vertical map is $I_{0}^{s.\flat}\widehat{R_{\infty}^{s.\flat}}$, the assertion follows. 
\end{proof}

\subsection{Examples: complete local log-regular rings}\label{smalltiltlog}
\subsubsection{Calculation of the tilts}
As an example of tilts of Noetherian perfectoid towers, we calculate them for certain towers of local log-regular rings. Firstly, we review a perfectoid tower constructed in \cite{GR22}. 

\begin{construction}\label{logtower}

Let $(R,\mathcal{Q},\alpha)$ be a \textit{complete} local log-regular ring with perfect residue field of characteristic $p>0$.
Assume that $\mathcal{Q}$ is fine, sharp, and saturated (see Remark \ref{rmk2125}). Let $I_{\alpha}\subseteq R$ be the ideal defined in Definition \ref{LogSchemeStr2}. Set $A:=R/I_\alpha$. Let $(f_1,\ldots,f_r)$ be a sequence of elements of $R$ whose image in $A$ is \emph{maximal} (see Definition \ref{maximal}).
Since the residue field of $R$ is perfect, $r$ is the dimension of $A$ (see \S \ref{AppendixA}).
For every $i\geq 0$, we consider the ring 
$$
A_i:=A[T_1,\ldots, T_r]/(T_1^{p^i}-\overline{f_1},\ldots, T_r^{p^i}-\overline{f_r}),
$$
where each $\overline{f_j}$ denotes the image of $f_j$ in $A$ ($j=1,\ldots, r$). Notice that $A_i$ is regular by Theorem \ref{criterion-regularity}. 
Moreover, we set $\mathcal{Q}^{(i)}:= \mathcal{Q}^{(i)}_p$ (see Definition \ref{cthpowermonoid}).
%\{\gamma\in \mathcal{Q}_\mathbb{Q}\ |\ \gamma^{p^n}\in \mathcal{Q}\}.
Furthermore, we define 
\begin{equation}\label{eq305SunN1}
R'_i:=\mathbb{Z}[\mathcal{Q}^{(i)}]\otimes_{\mathbb{Z}[\mathcal{Q}]}R,\ R''_i:=R[T_1,\ldots, T_r]/(T_1^{p^i}-f_1,\ldots, T_r^{p^i}-f_r),
\end{equation}
and 
\begin{equation}\label{eq305SunN2}
R_i:=R'_i\otimes_{R}R''_i.
\end{equation}
Let $t_i :R_i \rightarrow R_{i+1}$ be the ring map that is naturally induced by the inclusion map $\iota^{(i)}: \mathcal{Q}^{(i)}\hookrightarrow \mathcal{Q}^{(i+1)}$. 
Since $R''_{i+1}$ is a free $R''_i$-module, $t_i$ is universally injective by Lemma \ref{FineSharpSat} and the condition (e) in Proposition \ref{Ogus1.4.2.7} (2).

\end{construction}

\begin{proposition}
\label{claimlog}
Keep the notation as in Construction \ref{logtower}.
Let $\alpha_i: \mathcal{Q}^{(i)}\to R_i$ be the natural map. 
Then $(R_i, \mathcal{Q}^{(i)}, \alpha_i)$ is a local log-regular ring. 
%Furthermore, for every $i \geq 0$, the following diagram commutes:
%\begin{equation}\label{5311016}
%\vcenter{
%\xymatrix{
%C(k_i)\llbracket  \mathcal{Q}_0^{(i)} \oplus \mathbb{N}^r \rrbracket \ar@{^(->}[r] \ar@{->>}[d]& C(k_{i+1})\llbracket  \mathcal{Q}_0^{(i+1)} \oplus \mathbb{N}^r \rrbracket \ar@{->>}[d] \\
%R_i \ar@{^(->}[r]^{t_i} & R_{i+1},
%}}
%\end{equation}
%where the top arrow is the inclusion map.

\end{proposition}
\begin{proof}
%For the first assertion, 
We refer the reader to \cite[17.2.5]{GR22}. 
%The second assertion follows from \cite[Lemma 17.2.4]{GR22}.
\end{proof}

By construction, we obtain the tower of rings $(\{R_i\}_{i \geq 0},\{t_i\}_{i \geq 0})$ (see Definition \ref{towerdef}).

\begin{proposition}
\label{smalltilt}
Keep the notation as in Construction \ref{logtower}.
Then the tower $(\{R_i\}_{i \geq 0}, \{t_i\}_{i \geq 0})$ of local log-regular rings defined above is a perfectoid tower arising from $(R,(p))$.
\end{proposition}

\begin{proof}
We verify (a)-(g) in Definition \ref{invqperf} and Definition \ref{stilt}.
The axiom (a) is trivial. Since $t_i$ is universally injective, the axiom (b) follows. 
The axioms (c) and (d) follow from \cite[(17.2.10) and Lemma 17.2.11]{GR22}. 
Since $R$ is of residual characteristic $p$, the axiom (e) follows from the locality. 
Since $t_i$ is injective and $R_i$ is a domain for any $i \geq 0$, the axiom (g) holds by Remark \ref{remark2022705}.
Finally, let us check that the axiom (f) holds. In the case when $p=0$, it follows from \cite[Theorem 17.2.14 (i)]{GR22}. Otherwise, there exists an element $\varpi\in R_{1}$ that satisfies $\varpi^p = pu$ for some unit $u \in R_1$ by \cite[Theorem 17.2.14 (ii)]{GR22}.
Set $I_1 := (\varpi)$.
Then the axiom (f-1) holds.
%Otherwise, there exists a regular element $\varpi\in R_{1}$ that satisfies the axiom (f-1) by \cite[Theorem 17.2.14 (ii)]{GR22}. 
Moreover, the axiom (f-2) follows from \cite[Theorem 17.2.14 (iii)]{GR22}.
Thus the assertion follows.
\end{proof}

For calculating the tilt of the perfectoid tower constructed above, the following lemma is quite useful. 

\begin{lemma}\label{structurelog}
Keep the notation as in Proposition \ref{claimlog}. 
Let $k$ be the residue field of $R$. 
Then there exists a family of ring maps $\{ \phi_i : C(k)\llbracket  \mathcal{Q}^{(i)} \oplus (\mathbb{N}^r)^{(i)} \rrbracket \to R_i \}_{ i \geq 0 }$ which is compatible with the log structures of $\{(R_i, \mathcal{Q}^{(i)}, \alpha_i)\}_{i\geq 0}$ such that the following diagram commutes for every $i \geq 0$:
\begin{equation}\label{5311016}
\vcenter{
\xymatrix{
C(k)\llbracket  \mathcal{Q}^{(i)} \oplus (\mathbb{N}^r)^{(i)} \rrbracket \ar@{^(->}[r] \ar@{->>}[d]^{\phi_i} & C(k)\llbracket  \mathcal{Q}^{(i+1)} \oplus (\mathbb{N}^r)^{(i+1)} \rrbracket \ar@{->>}[d]^{\phi_{i+1}} \\
R_i \ar@{^(->}[r]^{t_i} & R_{i+1}
}}
\end{equation}
(where the top arrow is the natural inclusion). 
Moreover, there exists an element $\theta \in C(k)\llbracket  \mathcal{Q} \oplus \mathbb{N}^r \rrbracket$ whose constant term is $p$ such that the kernel of $\phi_i$ is generated by $\theta$ for every $i \geq 0$.
\end{lemma}

\begin{proof}
First we remark the following. 
Let $k_i$ be the residue field of $R_{i}$. 
Then by Lemma \ref{localqinvperf} (1) and Lemma \ref{secondCD} (2), the transition maps induce a purely inseparable extension $k \hookrightarrow k_i$. 
Moreover, this extension is trivial because $k$ is perfect. 
Therefore, we can identify $k_i$ (resp. the Cohen ring of $R_i$) with $k$ (resp. $C(k)$).

Next, let us show the existence of a family of ring maps $\{ \phi_i \}_{i \geq 0}$ with the desired compatibility. 
Since $(R_i, \mathcal{Q}^{(i)} , \alpha_i)$ is a complete local log-regular ring, we can take a surjective ring map $\psi_i: C(k)\llbracket \mathcal{Q}^{(i)} \oplus \mathbb{N}^r\rrbracket\to R_{i}$ as in Theorem \ref{CohenLogReg}; its kernel is generated by an element $\theta_{i}$ whose constant term is $p$, and the diagram: 
\[
\xymatrix{
\mathcal{Q}^{(i)} \ar[r] \ar[rd]_{\alpha_i} & C(k)\llbracket \mathcal{Q}^{(i)} \oplus \mathbb{N}^r \rrbracket \ar@{->>}[d]^{\psi_i} \\
 & R_i
}\]
commutes. 
For $j=1,\ldots, r$, let us denote  by $f_j^{1/p^i}$ the image of $T_j\in R[T_{1},\ldots, T_{r}]$ in $R_i$ (see (\ref{eq305SunN1}) and (\ref{eq305SunN2})). 
Note that the sequence $f_{1}^{1/p^{i}}, \ldots, f_r^{1/p^i}$ in $R_{i}$ becomes a regular system of parameters of $R_i/I_{\alpha_i}$ by the reduction modulo $I_{\alpha_i}$ (see \cite[17.2.3]{GR22} and \cite[17.2.5]{GR22}).
Thus, for the set of the canonical basis $\{ {\bf e}_1, \ldots, {\bf e}_r \}$ of $\mathbb{N}^r$, we may assume $\psi_i (e^{{\bf e}_j}) = f_j^{1/p^i}$ by the construction of $\psi_i$ (see the proof of \cite[Chapter III, Theorem 1.11.2]{Ogus18}). 
Hence we can choose $\{\psi_{i}\}_{i\geq 0}$ so that the diagram: 
\begin{equation}\label{751711}
\vcenter{
\xymatrix{
C(k)\llbracket \mathcal{Q}^{(i)} \oplus \mathbb{N}^r \rrbracket \ar@{^(->}[r] \ar@{->>}[d]_{\psi_i} &
C(k)\llbracket  \mathcal{Q}^{(i+1)} \oplus \mathbb{N}^r \rrbracket \ar@{->>}[d]^{\psi_{i+1}} \\
R_i \ar@{^(->}[r]^{t_i} &  R_{i+1}
}}
\end{equation}
commutes. 
Thus it suffices to define $\phi_i : C(k)\llbracket\mathcal{Q}^{(i)} \oplus (\mathbb{N}^r)^{(i)}\rrbracket \to R_i$ as the composite map of the isomorphism 
$C(k)\llbracket  \mathcal{Q}^{(i)} \oplus (\mathbb{N}^r)^{(i)} \rrbracket  \xrightarrow{\cong} C(k)\llbracket \mathcal{Q}^{(i)} \oplus \mathbb{N}^r \rrbracket$ obtained by Lemma \ref{finite1} (3) and $\psi_{i}$. 

Finally, note that the image of $\theta_0\in \ker (\psi_{0})$ in $C(k)\llbracket \mathcal{Q}^{(i)} \oplus \mathbb{N}^r \rrbracket$ is contained in $\ker (\psi_{i})$, and its constant term is still $p$. 
Thus, by the latter assertion of Theorem \ref{CohenLogReg} (2), $\ker(\psi_i)$ is generated by $\theta_0$. 
Hence by taking $\theta_{0}$ as $\theta$, we complete the proof. 
\end{proof}

Let us consider the monoids $ \mathcal{Q}^{(i)} $ for an integral sharp monoid $\mathcal{Q}$. 
Since there is the natural inclusion $\iota^{(i)} : \mathcal{Q}^{(i)} \hookrightarrow \mathcal{Q}^{(i+1)}$ for any $i \geq 0$, we obtain a direct system of monoids $(\{ \mathcal{Q}^{(i)} \}_{i \geq 0} , \{ \iota^{(i)} \}_{i \geq 0})$. 
Moreover, the $p$-times map on $\mathcal{Q}^{(i+1)}$ gives a factorization:
\[
\xymatrix{
\mathcal{Q}^{(i+1)} \ar^{\times p}[r] \ar@{->>}_{\times p}[rd] & \mathcal{Q}^{(i+1)} \\
& \ar@{^(->}[u]_{\iota^{(i)}} \mathcal{Q}^{(i)}.
}\]
From this discussion, we define the small tilt of $\{ \mathcal{Q}^{(i)} \}_{i \geq 0}$.
\begin{definition}
Let $\mathcal{Q}$ be an integral monoid, and 
let $(\{ \mathcal{Q}^{(i)} \}_{i \geq 0}, \{\iota^{(i)}\}_{i \geq 0} )$ be as above. 
Then for an integer $j\geq0$, we define the \textit{$j$-th small tilt of $(\{ \mathcal{Q}^{(i)} \}_{i \geq 0}, \{\iota^{(i)}\}_{i \geq 0} )$} as the inverse limit
\begin{equation}\label{eq304SatN}
\mathcal{Q}^{s.\flat}_{j} := \varprojlim \{\cdots \to \mathcal{Q}^{(j+1)} \to \mathcal{Q}^{(j)}\},
\end{equation}
where the transition map $\mathcal{Q}^{(i+1)} \to \mathcal{Q}^{(i)}$ is the $p$-times map of monoids.
\end{definition}

Now we can derive important properties of the tilt of the perfectoid tower given in Construction \ref{logtower}. 
\begin{theorem}\label{TiltingLogRegular}
Keep the notation as in Lemma \ref{structurelog}.
Then the following assertions hold.
\begin{enumerate}
\item
The tower $(\{ (R_i)_{(p)}^{s.\flat} \}_{i \geq 0}, \{  (t_i)_{(p)}^{s.\flat} \}_{i \geq 0} )$ is isomorphic to $( \{ k\llbracket \mathcal{Q}^{(i)} \oplus (\mathbb{N}^r)^{(i)}\rrbracket \}_{i \geq 0}, \{u_i\}_{i \geq 0})$, where $u_i$ is the ring map induced by the natural inclusion $\mathcal{Q}^{(i)} \oplus (\mathbb{N}^r)^{(i)} \hookrightarrow \mathcal{Q}^{(i+1)} \oplus (\mathbb{N}^r)^{(i+1)}$.

\item
For every $j \ge 0$, there exists a homomorphism of monoids $\alpha^{s.\flat}_{j} : \mathcal{Q}^{s.\flat}_j \to (R_j)^{s.\flat}_{(p)}$ such that $((R_j)^{s.\flat}_{(p)}, \mathcal{Q}^{s.\flat}_{j}, \alpha^{s.\flat}_{j})$ is a local log-regular ring. 
\item
For every $j \ge 0$,  $(t_j)^{s.\flat}_{(p)} : (R_j)^{s.\flat}_{(p)} \to (R_{j+1})^{s.\flat}_{(p)}$ is module-finite and $(R_j)^{s.\flat}_{(p)}$ is $F$-finite.

\end{enumerate}
\end{theorem}

\begin{proof}
$(1)$: 
By Lemma \ref{structurelog}, each $R_i$ is isomorphic to $C(k)\llbracket \mathcal{Q}^{(i)} \oplus (\mathbb{N}^r)^{(i)} \rrbracket/(p-f)C(k)\llbracket \mathcal{Q}^{(i)} \oplus (\mathbb{N}^r)^{(i)} \rrbracket$ where $f$ is an element of $C(k)\llbracket \mathcal{Q} \oplus \mathbb{N}^r\rrbracket$ which has no constant term.
%%%%%
Set $S_{i} := k\llbracket \mathcal{Q}^{(i)} \oplus (\mathbb{N}^r)^{(i)}\rrbracket$ for any $i \geq 0$ and let $u_i: S_i \hookrightarrow S_{i+1}$ be the inclusion map induced by the natural inclusion $\mathcal{Q}^{(i)} \oplus (\mathbb{N}^r)^{(i)} \hookrightarrow \mathcal{Q}^{(i+1)} \oplus (\mathbb{N}^r)^{(i+1)}$.
Then the tower $(\{S_i\}_{i \geq 0},\{u_i\}_{i \geq 0})$ is a perfect tower. 
Indeed, each $S_i$ is reduced by Theorem \ref{CMnormal}; 
moreover, by the perfectness of $k$ and Lemma \ref{finite1} (3), the Frobenius endomorphism on $S_{i+1}$ factors through a surjection $G_i: S_{i+1}\to S_{i}$. 
In particular, $(\{S_i\}_{i \geq 0},\{u_i\}_{i \geq 0})$ is a perfectoid tower arising from $(S_{0}, (0))$ and $G_{i}$ is the $i$-th Frobenius projection (cf.\ Lemma \ref{perfdperftower}).

Put $\overline{f} := f \mod pC(k)\llbracket \mathcal{Q} \oplus \mathbb{N}^r  \rrbracket \in S_0$. Then each $S_{i}$ is $\overline{f}$-adically complete and separated by \cite[Lemma 2.1.1]{FGK11}. 
Moreover, the commutative diagram (\ref{5311016}) yields the commutative squares  ($i\geq 0$): 
\[
\xymatrix{
S_{i+1}/\overline{f}S_{i+1}  \ar[d]^{\overline{G_i}}\ar^{\cong}[r] & R_{i+1}/pR_{i+1}   \ar[d]^{F_i} \\
S_{i}/\overline{f}S_i   \ar^{\cong}[r] & R_{i}/pR_i   
}\]
that are compatible with $\{\overline{u_i}: S_{i}/\overline{f}S_{i}\to S_{i+1}/\overline{f}S_{i+1}\}_{i\geq 0}$ and $\{\overline{t_i}\}_{i\geq 0}$. 
Hence by Lemma \ref{stiltcompletion}, we obtain the isomorphisms 
\begin{equation}
\label{isomorphism2}
(R_j)^{s.\flat}_{(p)} \xleftarrow{\cong} 
\vpl\{ \cdots \xrightarrow{\overline{G_{j+1}}} S_{j+1}/\overline{f}S_{j+1}  \xrightarrow{\overline{G_{j}}} S_j/\overline{f}S_j  \} \xrightarrow{\cong} S_j \ \ \ \ \ \  (j\geq 0)
\end{equation}
that are compatible with the transition maps of the towers. 
Thus the assertion follows.

(2): Considering the inverse limit of the composite maps $\mathcal{Q}^{(j+i)}\xrightarrow{\alpha_{j+i}}R_{j+i}\twoheadrightarrow R_{j+i}/pR_{j+i}$ ($i\geq 0$), we obtain a homomorphism of monoids $\alpha^{s.\flat}_{j} : \mathcal{Q}^{s.\flat}_{j} \to (R_j)^{s.\flat}_{(p)}$. 
On the other hand, let $\overline{\alpha_{j}} : \mathcal{Q}^{(j)} \to S_{j}$ be the natural inclusion. 
Then, since $S_j$ is canonically isomorphic to $k\llbracket\mathcal{Q}^{(j)} \oplus \mathbb{N}^r\rrbracket$, $(S_j, \mathcal{Q}^{(j)}, \overline{\alpha_j} )$ is a local log-regular ring by Theorem \ref{CohenLogReg} (1). 
Thus it suffices to show that $((R_j)^{s.\flat}_{(p)}, \mathcal{Q}^{s.\flat}_{j}, \alpha^{s.\flat}_{j})$ is isomorphic to $(S_j, \mathcal{Q}^{(j)}, \overline{\alpha_j})$ as a log ring. 
Since the transition maps in (\ref{eq304SatN}) are isomorphisms by Lemma \ref{finite1} (3), we obtain the isomorphisms of monoids 
\begin{equation}
\label{isomorphism3}
\mathcal{Q}^{s.\flat}_{j}\xleftarrow{\id_{\mathcal{Q}^{s.\flat}_{j}}}\mathcal{Q}^{s.\flat}_{j} \xrightarrow{\cong} \mathcal{Q}^{(j)}\ \ \ \ \ \ (j\geq 0). 
\end{equation}
Then one can connect (\ref{isomorphism3}) to (\ref{isomorphism2}) to construct a commutative diagram using $\alpha^{s.\flat}_j$ and $\overline{\alpha_j}$. 
Hence the assertion follows. 

$(3)$: By Lemma \ref{finite424} (1), $t_j: R_j\to R_{j+1}$ is module-finite. 
Hence by Proposition \ref{smalltiltproperty2} (1), $(t_j)^{s.\flat}_{(p)}: (R_j)^{s.\flat}_{(p)}\to (R_{j+1})^{s.\flat}_{(p)}$ is also module-finite. 
Finally let us show that $(R_j)^{s.\flat}_{(p)}$ is $F$-finite. 
By the assertion (2), $(R_j)^{s.\flat}_{(p)}$ is a complete Noetherian local ring, and the residue field is $F$-finite because it is perfect.
Thus the assertion follows from \cite[Theorem 8.4]{Ma86}.
\end{proof}

%\begin{remark}
%If a complete local ring $(R, \fm, k)$ of characteristic $p>0$ has $F$-finite residue field, then $R$ itself is $F$-finite. Therefore, $((R_j)^{s.\flat}_{(p)}, \mathcal{Q}^{s.\flat}_{j}, \alpha^{s.\flat}_{j})$ in Theorem \ref{TiltingLogRegular} is $F$-finite for every $j \geq 0$, since its residue field $k$ is perfect.
%\end{remark}

\begin{example}
\label{examplelogstilt}
\begin{enumerate}
\item
A tower of regular local rings which is treated in \cite{CK19} and \cite{KS20} is a perfectoid tower in our sense.
Let $(R, \fm , k)$ be a $d$-dimensional regular local ring whose residue field $k$ is perfect and let $x_1, \ldots, x_d$ be a regular sequence of parameters.
Let ${\bf e}_1,\ldots, {\bfe}_d$ be the canonical basis of $\mbN^d$.
Then $(R,\mathbb{N}^d, \alpha)$ is a local log-regular ring where $\alpha : \mathbb{N}^d \to R$ is a homomorphism of monoids which maps ${\bfe}_i$ to $x_i$.
Furthermore, assume that $R$ is $\fm$-adically complete.
Then, by Cohen's structure theorem, $R$ is isomorphic to
$$
W(k)\llbracket x_1,\ldots,x_d\rrbracket/(p-f)
$$
where $f=x_1$ or $f \in (p,x_1,\ldots,x_d)^2$ (the former case is called \textit{unramified}, and the latter is called \textit{ramified}).
Let us construct a perfectoid tower arising from $(R, (p))$ along Construction \ref{logtower}.
Since $k$ is perfect, ${\bf \Omega}_{k}$ is zero by the short exact sequences (\ref{modifieddifferential}) and the definition of itself.
%and  (\ref{modifieddifferential2}).
This implies that the image of the empty subset of $R$ in $k$ forms a maximal sequence.
%$R$ has no sequences of non-zero elements whose image in $k$ is maximal.
Hence $R''_i$ in Construction \ref{logtower} is equal to $R$.
Moreover, $(\mathbb{N}^d)^{(i)}$ is generated by 
$
\frac{1}{p^i}\bfe_1, \ldots, \frac{1}{p^i}\bfe_d.
$
Thus, applying Construction \ref{logtower}, we obtain
$$
R_i =
R_i' = 
\mathbb{Z}[(\mathbb{N}^d)^{( i )}] \otimes_{\mathbb{Z}[\mathbb{N}^d]} R \cong 
R[T_1,\ldots, T_d]/(T_1^{p^i}-x_1, \ldots, T_d^{p^i}-x_d) \cong
W(k)\llbracket  x_1^{1/p^i}, \ldots, x_d^{1/p^i}  \rrbracket/(p-f).
$$
Set the natural injection $t_i : R_i \to R_{i+1}$ for any $i \geq 0$.
Then, by Proposition $\ref{smalltilt}$, $(\{R_i\}_{i \geq 0}, \{ t_i \}_{i\geq 0})$ is a perfectoid tower arising from $(R,(p))$.
By Theorem \ref{TiltingLogRegular}, its tilt 
$(\{(R_i)^{s.\flat}_{(p)}\}_{i\geq 0}, \{(t_i)^{s.\flat}_{(p)}\}_{i\geq 0})$ is isomorphic to the tower $k\llbracket \mbN^d \rrbracket\hookrightarrow 
k\llbracket (\mbN^d)^{(1)} \rrbracket \hookrightarrow
k\llbracket (\mbN^d)^{(2)} \rrbracket \hookrightarrow \cdots$, 
which can be written as  
$$
k\llbracket x_1, \ldots, x_d\rrbracket \hookrightarrow k\llbracket x_1^{1/p}, \ldots, x_d^{1/p}\rrbracket \hookrightarrow  
k\llbracket x_1^{1/p^2}, \ldots, x_d^{1/p^2}\rrbracket \hookrightarrow \cdots.
$$

\item
Consider the surjection:
$$
S:=W(k)\llbracket x,y,z,w\rrbracket/(xy-zw) \twoheadrightarrow R:=W(k)\llbracket x,y,z,w\rrbracket/(xy-zw,p-w)=W(k)\llbracket x,y,z\rrbracket/(xy-pz).
$$
where $k$ is a perfect field. Let $\mathcal{Q} \subseteq \mathbb{N}^4$ be a saturated submonoid generated by 
$$
(1,1,0,0),(0,0,1,1),(1,0,0,1),\text{ and }(0,1,1,0).
$$
Then $S$ admits a homomorphism of monoids $\alpha_S:\mathcal{Q} \to S$ by letting $(1,1,0,0) \mapsto x, (0,0,1,1) \mapsto y, (1,0,0,1) \mapsto z$ and $(0,1,1,0) \mapsto w$. With this, $(S,\mathcal{Q},\alpha_S)$ is a local log-regular ring. The composite map $\alpha_R:\mathcal{Q} \to S \to R$ makes $R$ into a local log ring. Indeed, we can write $R\cong W(k)\llbracket\mathcal{Q}\rrbracket/(p-e^{(0,1,1,0)})$, hence $(R,\mathcal{Q},\alpha_R)$ is log-regular by Theorem \ref{CohenLogReg}. %Put $R_0 := R$. 

Next, note that $R/I_{\alpha_R} \cong k$.
Then, for the same reason in (1), $R''_i$ is equal to $R$.
Moreover, $\mathcal{Q}^{(i)}$ is generated by
$$
\Bigl(\frac{1}{p^i},\frac{1}{p^i},0,0\Bigr),
\Bigl(0,0,\frac{1}{p^i},\frac{1}{p^i}\Bigr),
\Bigl(\frac{1}{p^i},0,0\frac{1}{p^i}\Bigr),
\Bigl(0,\frac{1}{p^i},\frac{1}{p^i},0\Bigr)
.
$$
Thus, applying Construction \ref{logtower}, we obtain
\[
\begin{array}{rcl}
R_i & =  & R\llbracket \mathcal{Q}^{(i)} \rrbracket  \\
%\mathbb{Z}[\mathcal{Q}^{( n )}] \otimes_{\mathbb{Z}[\mathcal{Q}]} R \cong 
& \cong  & W(k)\llbracket \mathcal{Q}^{(i)} \rrbracket/(p-e^{(0,1,1,0)}) \\
& \cong &
%W(k)\llbracket x^{1/p^i},y^{1/p^i},z^{1/p^i},w^{1/p^i} \rrbracket/(x^{k/p^i}y^{k/p^i}-z^{k/p^i}w^{k/p^i}~|~1 \leq k \leq p^i)+(p-w).
W(k)\llbracket x^{1/p^i},y^{1/p^i},z^{1/p^i},w^{1/p^i} \rrbracket/(x^{1/p^i}y^{1/p^i}-z^{1/p^i}w^{1/p^i}, p-w).
\end{array}
\]
%Moreover, since $(R_n, \mathcal{Q}^{(n)}, \alpha_n)$ is a local log-regular ring by , we obtain the isomorphisms 
%$$
%R_n \cong 
%W(k)\llbracket \mathcal{Q}^{(n)} \rrbracket/(p-w) \cong 
%W(k)\llbracketx^{\frac{1}{p^n}},y^{\frac{1}{p^n}},z^{\frac{1}{p^n}},w^{\frac{1}{p^n}} \rrbracket/(x^{\frac{k}{p^n}}y^{\frac{k}{p^n}}-z^{\frac{k}{p^n}}w^{\frac{k}{p^n}}~|~1 \leq k \leq p^n)+(p-w).
%$$
%It is easy to see that the direct system $\{R_n\}_{n \geq 0}$ satisfies the axioms (a), (b), (c), (d), (e), (f), (g) by taking $f_0 = p = w$ and $f_1 = w^{\frac{1}{p}}$. Thus we obtain the isomorphism
Set a natural injection $t_i : R_i \to R_{i+1}$.
Then, by Proposition \ref{smalltilt}, $(\{R_i\}_{i \geq 0}, \{ t_i \}_{i \geq 0})$ is a perfectoid tower arising from $(R,(p))$. Hence 
$$
R_\infty = \vil_{i \geq 0} R_i \cong 
\displaystyle\bigcup_{i \geq 0}W(k)\llbracket x^{1/p^i},y^{1/p^i},z^{1/p^i},w^{1/p^i} \rrbracket/(x^{1/p^i}y^{1/p^i}-z^{1/p^i}w^{1/p^i}, p-w), 
$$
and its $p$-adic completion is perfectoid. 
Moreover, one can calculate the tilt $(\{R_i^{s.\flat}\}_{i\geq 0}, \{t_i^{s.\flat}\}_{i\geq 0})$ to be $k\llbracket\mathcal{Q}\rrbracket \hookrightarrow
k\llbracket\mathcal{Q}^{(1)}\rrbracket\hookrightarrow
k\llbracket\mathcal{Q}^{(2)}\rrbracket\hookrightarrow\cdots$ 
by Theorem \ref{TiltingLogRegular}, or, more explicitly, 
$$
k\llbracket x,y,z,w \rrbracket/(xy-zw)\hookrightarrow 
k\llbracket x^{1/p},y^{1/p},z^{1/p},w^{1/p} \rrbracket/(x^{1/p}y^{1/p}-z^{1/p}w^{1/p})\hookrightarrow\cdots.
$$

\end{enumerate}
\end{example}

\subsubsection{Towers of split maps and sousperfectoid rings}
Let us recall that Hansen and Kedlaya introduced a new class of topological rings that guarantees sheafiness on the associated adic spectra (see \cite[Definition 7.1]{HK20}).

\begin{definition}
%Let $A$ be a Huber ring such that a prime $p \in A$ is topologically nilpotent. 
Let $A$ be a complete and separated Tate ring such that a prime $p \in A$ is topologically nilpotent.
We say that $A$ is \textit{sousperfectoid}, if there exists a perfectoid ring $B$ in the sense of Fontaine (see \cite[Definition 2.13]{HK20}) with a continuous $A$-linear map $f:A \to B$ that splits in the category of topological $A$-modules. That is, there is a continuous $A$-linear map $\sigma:B \to A$ such that $\sigma \circ f=\id_A$.
\end{definition}

Let us show that a perfectoid tower consisting of split maps induces sousperfectoid rings. 
In view of Theorem \ref{log-splinter}, one can apply this result to the towers discussed above. 
See \cite{NS19} for detailed studies on algebraic aspects of Tate rings.

\begin{proposition}\label{prop1347314}
Let $\tower$ be a perfectoid tower arising from some pair $(R,(f_0))$.
Assume that $f_0$ is regular, $R$ is $f_0$-adically complete and separated, and $t_i$ splits as an $R_i$-linear map for every $i \geq 0$.
We equip $R[\frac{1}{f_{0}}]$ with the linear topology in such a way that $\{f_0^nR\}_{n \geq 1}$ defines a fundamental system of open neighborhoods at $0 \in R[\frac{1}{f_0}]$. 
Then $R[\frac{1}{f_0}]$ is a sousperfectoid Tate ring. In particular, it is stably uniform.
\end{proposition}

In order to prove this, we need the following lemma.

\begin{lemma}\label{splitmodule}
Keep the notations and assumptions as in Proposition \ref{prop1347314}.
Then the natural map $R_0 \to \varinjlim_{i\geq 0}R_i$ splits as an $R_0$-linear map. 
\end{lemma}

\begin{proof}
%[Proof of Claim \ref{splitmodule}]
We use the fact that each $t_i: R_i \to R_{i+1}$ splits as an $R_i$-linear map by assumption.
This implies that the short exact sequence of $R$-modules
$$
0 \to R_0 \to R_i \to R_i/R \to 0
$$
splits for any $i\geq 0$. It induces a commutative diagram of $R$-modules
$$
\begin{CD}
0 @>>> \Hom_{R_0}(R_{i+1}/R_0,R_0) @>>> \Hom_{R_0}(R_{i+1},R_0) @>>> \Hom_{R_0}(R_0,R_0) @>>> 0\\
@. @V\alpha_iVV @V\beta_iVV @| \\
0 @>>> \Hom_{R_0}(R_i/R_0,R_0) @>>> \Hom_{R_0}(R_i,R_0) @>>> \Hom_{R_0}(R_0,R_0) @>>> 0
\end{CD}
$$
where each horizontal sequence is split exact, and each vertical map forms an inverse system induced by $t_i : R_i \to R_{i+1}$. In particular, $\beta_i$ is surjective and it thus follows from the snake lemma that $\alpha_i$ is surjective as well. By taking inverse limits, we obtain the short exact sequence:
$$
0 \to \varprojlim_{i \geq 0} \Hom_{R_0}(R_i/R_0,R_0) \to \varprojlim_{i \geq 0} \Hom_{R_0}(R_i,R_0) \xrightarrow{h} \Hom_{R_0}(R_0,R_0) \to 0.
$$
It follows from \cite[Lemma 4.1]{Schenzel14} that $h$ is the canonical surjection $\Hom_{R_0}(R_\infty,R_0) \twoheadrightarrow \Hom_{R_0}(R_0,R_0)$. Then choosing an inverse image of $\id_{R_0} \in \Hom_{R_0}(R_0,R_0)$ gives a splitting of $R_0 \to R_\infty$.
\end{proof}

\begin{proof}[Proof of Proposition \ref{prop1347314}]
We have constructed an infinite extension $R \to R_\infty$ such that if $\widehat{R_\infty}$ is the $f_0$-adic completion, then the associated Tate ring $\widehat{R_\infty}[\frac{1}{f_0}]$ is a perfectoid ring in the sense of Fontaine by Theorem \ref{smalltiltproperty1} and \cite[Lemma 3.21]{BMS18}.%by \cite[Theorem 17.2.14]{GR22}. 

By Lemma \ref{splitlemma1} and  Lemma \ref{splitmodule}, it follows that the map $R[\frac{1}{f_0}] \to \widehat{R_\infty}[\frac{1}{f_0}]$ splits in the category of topological $R[\frac{1}{f_0}]$-modules (notice that $R$ is $f_0$-adically complete and separated). Thus, $R[\frac{1}{f_0}]$ is a sousperfectoid Tate ring. The combination of \cite[Corollary 8.10]{HK20}, \cite[Proposition 11.3]{HK20} and \cite[Lemma 11.9]{HK20} allows us to conclude that $R[\frac{1}{f_0}]$ is stably uniform.
\end{proof}

As a corollary, one can obtain the stable uniformity for complete local log-regular rings (see also Construction \ref{logtower} and Theorem \ref{log-splinter}).

\begin{corollary}
Let $(R,\mathcal{Q},\alpha)$ is a complete local log-regular ring of mixed characteristic with perfect residue field. We equip $R[\frac{1}{p}]$ with the structure of a complete and separated Tate ring in such a way that $\{p^nR\}_{n \geq 1}$ defines a fundamental system of open neighborhoods at $0 \in R[\frac{1}{p}]$. Then $R[\frac{1}{p}]$ is stably uniform.
\end{corollary}

\section{Applications to \'etale cohomology of Noetherian rings}
In this section, we establish several results on \'etale cohomology of Noetherian rings, as applications of the theory of perfectoid towers developed in \S\ref{sectionPerf}. 
In \S\ref{comparisontilt}, for a ring that admits a certain type of perfectoid tower, we prove that finiteness of \'etale cohomology groups on the positive characteristic side carries over to the mixed characteristic side (Proposition \ref{TiltEtaleCohHensel}). In \S\ref{divisorclassLR}, we apply this result to a problem on divisor class groups of log-regular rings. 

We prepare some notation. Let $X$ be a scheme and let $X_{\et}$ denote the category of schemes that are \'etale over $X$, and for any \'etale $X$-scheme $Y$, we specify the covering $\{Y_i \to Y\}_{i \in I}$ so that $Y_i$ is \'etale over $Y$ and the family $\{Y_i\}_{i \in I}$ covers surjectively $Y$. For an abelian sheaf $\mathcal{F}$ on $X_{\et}$, we denote by $H^i(X_{\et},\mathcal{F})$ the value of the $i$-th derived functor of $U \in X_{\et} \mapsto \Gamma(U,\mathcal{F})$. For the most part of applications, we consider \textit{torsion} sheaves, such as $\mathbb{Z}/n \mathbb{Z}$ and $\mu_n$ for $n \in \mathbb{N}$. However, for the multiplicative group scheme $\mathbb{G}_m$, we often use the following isomorphism:
$$
H^1(X_{\et},\mathbb{G}_m) \cong \Pic(X).
$$
For the basics on \'etale cohomology, we often use \cite{LF15} or \cite{Mi80} as references.

\subsection{Tilting \'etale cohomology groups}\label{comparisontilt}
%\subsubsection{Preliminaries on \'etale cohomology}
Let $A$ be a ring with an ideal $J$, let $\widehat{A}$ be the $J$-adic completion of $A$, and let $U \subseteq \Spec(A)$ be an open subset. Then we define the $J$-adic completion of $U$ to be the open subset $\widehat{U} \subseteq \Spec(\widehat{A})$, which is the inverse image of $U$ via $\Spec(\widehat{A}) \to \Spec(A)$. We will use the following result for deriving results on the behavior of \'etale cohomology under the tilting operation as well as some interesting results on the divisor class groups of Noetherian normal domains (see Proposition \ref{DivPicard} and Proposition \ref{Basechange1}).

\begin{theorem}[Fujiwara-Gabber]
\label{HenselianEtale}
Let $(A,J)$ be a Henselian pair with $X:=\Spec(A)$ and let $\widehat{A}$ be the $J$-adic completion of $A$. Then the following assertions hold.
\begin{enumerate}
\item
For any abelian torsion sheaf $\mathscr{F}$ on $X_{\et}$, we have $\mathbf{R}\Gamma(\Spec(A)_{\et},\mathscr{F}) \simeq \mathbf{R}\Gamma(\Spec(A/J)_{\et},\mathscr{F}|_{\Spec(A/J)})$.

\item
Assume that $J$ is finitely generated. Then for any abelian torsion sheaf $\mathscr{F}$ on $X_{\et}$ and any open subset $U \subseteq X$ such that $X \setminus V(J) \subseteq U$, we have $\mathbf{R}\Gamma(U_{\et},\mathscr{F}) \simeq \mathbf{R}\Gamma(\widehat{U}_{\et},\mathscr{F})$.
\end{enumerate}
\end{theorem}

\begin{proof}
The first statement is known as \textit{Affine analog of proper base change} in \cite{G94}, while the second one is known as \textit{Formal base change theorem} which is \cite[Theorem 7.1.1]{F95} in the Noetherian case, and \cite[XX, 4.4]{Book} in the non-Noetherian case.
\end{proof}

We will need the tilting invariance of (local) \'etale cohomology from \cite[Theorem 2.2.7]{KS20}. To state the theorem and establish a variant of it, we give some notations.

\begin{definition}\label{tiltopen}
Let $(A, I)$ and $(B, J)$ be pairs such that there exists an isomorphism of rings $\Phi: A/I\xrightarrow{\cong} B/J$. 
Then for any open subset $U\subseteq \Spec (B)$ containing $\Spec (B)\setminus V(J)$, we define an open subset 
$F_{A, \Phi}(U)\subseteq \Spec(A)$ as the complement of the closed subset $\Spec(\Phi)\bigl(\Spec (B)\setminus U\bigr)\subseteq \Spec (A)$. 
\end{definition}

One can define small tilts of Zariski-open subsets. 

\begin{definition}\label{openstilt}
Let $(\{R_{i}\}_{i\geq 0}, \{t_{i}\}_{i\geq 0})$ be a perfectoid tower arising from some pair $(R, I_0)$, and 
let $(\{R^{s.\flat}_{i}\}_{i\geq 0}, \{t^{s.\flat}_{i}\}_{i\geq 0})$ be the tilt associated to $(R, I_0)$. 
Recall that we then have an isomorphism of rings $\overline{\Phi^{(i)}_{0}}: R^{s.\flat}_{i}/I^{s.\flat}_{0}R^{s.\flat}_{i}\xrightarrow{\cong}R_{i}/I_{0}R_{i}$ for every $i\geq 0$. 
For every $i\geq 0$ and every open subset $U\subseteq \Spec(R_{i})$ containing $\Spec(R_{i})\setminus V(I_{0}R_{i})$, we define 
$$
U^{s.\flat}_{I_0}:=F_{R^{s.\flat}_{i}, \overline{\Phi^{(i)}_0}}(U). 
$$
We also denote $U^{s.\flat}_{I_{0}}$ by $U^{s.\flat}$ as an abbreviated form. 
\end{definition}

Note that by the compatibility described in Lemma \ref{921WedN}, the operation $U\rightsquigarrow U^{s.\flat}$ is compatible with the base extension along the transition maps of a perfectoid tower. 

Let us give some examples of $U^{s.\flat}$.

\begin{example}[Punctured spectra of regular local rings]\label{pctSpec}
Keep the notation as in Example \ref{examplelogstilt} (1). In this situation, the isomorphism $\overline{\Phi^{(0)}_{0}}: R^{s.\flat}_{0}/I^{s.\flat}_{0}\xrightarrow{\cong} R_{0}/I_{0}$ in Definition \ref{openstilt} can be written as  
\begin{equation}\label{isomregular}
k\llbracket x_1, \ldots, x_d \rrbracket/(p^{s.\flat})\xrightarrow{\cong}R/pR, 
\end{equation}
where $p^{s.\flat}\in k\llbracket x_1, \ldots, x_d \rrbracket$ is some element. 
Set $U:=\Spec(R) \setminus V(\mathfrak{m})$. 
Then, since the maximal ideal $\overline{\mathfrak{m}}\subseteq R/pR$ corresponds to the (unique) maximal ideal of $k\llbracket x_1, \ldots, x_d \rrbracket/(p^{s.\flat})$, we have 
$$
U^{s.\flat}\cong\Spec(k\llbracket x_{1},\ldots, x_{d}\rrbracket) \setminus V((x_{1},\ldots, x_{d})). 
$$
\end{example}

\begin{example}[Tilting for preperfectoid rings]\label{tiltcorrespond}
Keep the notation as in Example \ref{eg1227TueN}. 
Then by Lemma \ref{lem1218SunN}, $\overline{\Phi^{(0)}_{0}}: R^{s.\flat}_{0}/I^{s.\flat}_{0}\xrightarrow{\cong} R_{0}/I_{0}$ is identified with the isomorphism: 
\begin{equation}\label{1217SatN}
\overline{\theta_{\widehat{R}}}: (\widehat{R})^{\flat}/I_{0}^{\flat}(\widehat{R})^{\flat}\xrightarrow{\cong} \widehat{R}/I_{0}\widehat{R}
\end{equation}
which is induced by the bottom map in the diagram (\ref{1218SunN}). 
In this case, we denote $F_{R^{\flat}, \overline{\Phi^{(0)}_{0}}}(U)$ by $U^{\flat}$  in distinction from $U^{s.\flat}$. 
\end{example}

The comparison theorem we need, due to \v{C}esnavi\v{c}ius and Scholze \cite[Theorem 2.2.7]{KS20}, is stated as follows. 

\begin{theorem}[\v{C}esnavi\v{c}ius-Scholze]
\label{tiltingCoh}
Let $A$ be a $\varpi$-adically Henselian ring with bounded $\varpi$-torsion for an element $\varpi \in A$ such that $p \in \varpi^pA$. Assume that the $\varpi$-adic completion of $A$ is perfectoid.  Let $U \subseteq \Spec(A)$ be a Zariski-open subset such that $\Spec(A) \setminus V(\varpi A) \subseteq U$, and let $U^\flat \subseteq \Spec(A^\flat)$ be its tilt (see Example \ref{tiltcorrespond}). 
\begin{enumerate}
\item
For every torsion abelian group $G$, we have $\mathbf{R}\Gamma(U_{\et},G) \cong \mathbf{R}\Gamma(U^\flat_{\et},G)$ in a functorial manner with respect to $A$, $U$, and $G$.

\item
Let $Z$ be the complement of $U \subseteq \Spec(A)$. Then for a torsion abelian group $G$, we have $\mathbf{R}\Gamma_Z(\Spec(A)_{\et},G) \cong \mathbf{R}\Gamma_Z(\Spec(A^\flat)_{\et},G)$.
\end{enumerate}
\end{theorem}

%\subsubsection{Comparison results}
Now we come to the main result on tilting \'etale cohomology groups. Recall that we have fixed a prime $p>0$.

\begin{proposition}
\label{TiltEtaleCohHensel}
Let $(\{R_{j}\}_{j\geq 0}, \{t_{j}\}_{j\geq 0})$ be a perfectoid tower arising from some pair $(R, I_0)$. Suppose that $R_{j}$ is $I_{0}$-adically Henselian 
for every $j\geq 0$. 
Let $\ell$ be a prime different from $p$. 
Suppose further that 
%\begin{itemize}
%\item
%$f_{0}=p$; 
%\item
for every $j \ge 0$, $t_{j}: R_j \to R_{j+1}$ is a module-finite extension of Noetherian normal domains whose generic extension is of $p$-power degree.\footnote{The existence of such towers is quite essential for applications to \'etale cohomology, because the extension degree of each $R_j \to R_{j+1}$ is controlled in such a way that the $p$-adic completion of its colimit is a perfectoid ring.} 
%\end{itemize}
Fix a Zariski-open subset $U \subseteq \Spec(R)$ such that $\Spec (R) \setminus V(pR) \subseteq U$ and the corresponding open subset $U^{s.\flat} \subseteq \Spec(R^{s.\flat})$ (cf.\ Definition \ref{openstilt}). Then, for any fixed $i,n \ge 0$ such that $|H^i(U_{\et}^{s.\flat},\mathbb{Z}/\ell^n\mathbb{Z})|<\infty$, one has 
$$
|H^i(U_{\et},\mathbb{Z}/\ell^n\mathbb{Z})| \le |H^i(U_{\et}^{s.\flat},\mathbb{Z}/\ell^n\mathbb{Z})|. 
$$
In particular, if $H^i(U_{\et}^{s.\flat},\mathbb{Z}/\ell^n\mathbb{Z})=0$, then $H^i(U_{\et},\mathbb{Z}/\ell^n\mathbb{Z})=0$.
\end{proposition}

\begin{proof}
Since each $R_j$ is a $p$-adically Henselian normal domain, so is $R_\infty=\varinjlim_{j\geq 0}R_{j}$. 
Moreover, every prime $\ell$ different from $p$ is a unit in $R_j$ and $R_\infty$. 
Attached to the tower $(\{R_j\}_{j\geq 0}, \{t_j\}_{j\geq 0})$, we get a tower of finite (not necessarily flat) maps of normal schemes:
\begin{equation}
\label{opentower}
U=U_0 \leftarrow \cdots \leftarrow U_j \leftarrow U_{j+1} \leftarrow \cdots.
\end{equation}
More precisely, let $h_j:\Spec(R_{j+1}) \to \Spec(R_{j})$ be the associated scheme map. Then the open set $U_{j+1}$ is defined as the inverse image $h^{-1}_j(U_j)$, thus defining the map $U_{j+1} \to U_j$ in the tower $(\ref{opentower})$. Since $h_j$ is a finite morphism of normal schemes, \cite[Lemma 3.4]{Bh14} applies to yield a well-defined trace map: 
$\Tr:h_{j*}h^*_j\mathbb{Z}/\ell^n\mathbb{Z} \to \mathbb{Z}/\ell^n\mathbb{Z}$ such that
\begin{equation}
\label{etaletrace}
\mathbb{Z}/\ell^n\mathbb{Z} \xrightarrow{h^*_j} h_{j*}h^*_j\mathbb{Z}/\ell^n\mathbb{Z} \xrightarrow{\Tr} \mathbb{Z}/\ell^n\mathbb{Z}
\end{equation}
is multiplication by the generic degree of $h_j$ (=$p$-power order). Then this is bijective, as the multiplication map by $p$ on $\mathbb{Z}/\ell^n\mathbb{Z}$ is bijective. We have the natural map: $H^i(U_{j,\et},\mathbb{Z}/\ell^n\mathbb{Z}) \to H^i(U_{j+1,\et},h_{j}^*\mathbb{Z}/\ell^n\mathbb{Z})$. Since $h_j$ is affine, the Leray spectral sequence gives $ H^i(U_{j+1,\et},h_j^*\mathbb{Z}/\ell^n\mathbb{Z}) \cong H^i(U_{j,\et},h_{j*}h_{j}^*\mathbb{Z}/\ell^n\mathbb{Z})$. Composing these maps, the composite map $(\ref{etaletrace})$ induces
$$
H^i(U_{j,\et},\mathbb{Z}/\ell^n\mathbb{Z}) \to H^i(U_{j+1,\et},h_j^*\mathbb{Z}/\ell^n\mathbb{Z}) \xrightarrow{\cong}H^i(U_{j,\et},h_{j*}h_j^*\mathbb{Z}/\ell^n\mathbb{Z}) \xrightarrow{\Tr} 
H^i(U_{j,\et},\mathbb{Z}/\ell^n\mathbb{Z})
$$
and the composition is bijective. Since $h_j^*\mathbb{Z}/\ell^n\mathbb{Z} \cong \mathbb{Z}/\ell^n\mathbb{Z}$, we get an injection
\begin{equation}
\label{etaleinjective}
H^i(U_{j,\et},\mathbb{Z}/\ell^n\mathbb{Z}) \hookrightarrow H^i(U_{j+1,\et},\mathbb{Z}/\ell^n\mathbb{Z}). 
\end{equation}
Set $U_\infty=\varprojlim_{j} U_j$. Since each morphism $U_{j+1} \to U_{j}$ is affine, by using $(\ref{etaleinjective})$ and \cite[Tag 09YQ]{Stacks}, we have
$$
H^i(U_{\et},\mathbb{Z}/\ell^n\mathbb{Z}) \hookrightarrow \varinjlim_j H^i(U_{j,\et},\mathbb{Z}/\ell^n\mathbb{Z}) \cong H^i(U_{\infty,\et},\mathbb{Z}/\ell^n\mathbb{Z}).
$$
Thus, it suffices to show that $|H^i(U_{\infty,\et},\mathbb{Z}/\ell^n\mathbb{Z})| \le |H^i(U_{\et}^{s.\flat},\mathbb{Z}/\ell^n\mathbb{Z})|$. 
Hence by tilting \'etale cohomology using Theorem \ref{tiltingCoh}, we are reduced to showing 
\begin{equation}
\label{etaleinjective2}
|H^i(U_{\infty,\et}^\flat,\mathbb{Z}/\ell^n\mathbb{Z})| \le |H^i(U_{\et}^{s.\flat},\mathbb{Z}/\ell^n\mathbb{Z})|,
\end{equation}
where $U_\infty^\flat$ is the open subset of $\Spec(R_\infty^\flat)$ that corresponds to $U_\infty \subseteq \Spec(R_\infty)$ in view of 
Example \ref{tiltcorrespond}. 
On the other hand, considering the tilt of $(\{R_{j}\}_{j\geq 0}, \{t_{j}\}_{j\geq 0})$ associated to $(R_{0}, I_{0})$, we have a perfect $\mathbb{F}_p$-tower $(\{R_j^{s.\flat}\}_{j\geq 0},  \{t^{s.\flat}_{j}\}_{j\geq 0})$.  
Note that each $R_j^{s.\flat}$ is $I_0^{s.\flat}$-adically Henselian Noetherian ring\footnote{It is not obvious whether $R_j^{s.\flat}$ is normal. However, the normality was used only in the trace argument and we do not need it in the following argument.} by Lemma \ref{1215ThuN} and Proposition \ref{smalltiltproperty2} (2), and $t_j^{s.\flat}$ is module-finite by Proposition \ref{smalltiltproperty2} (1). 
Considering the small tilts of the Zariski-open subsets appearing in  $(\ref{opentower})$ (see Definition \ref{openstilt}), we get a tower of finite maps: 
$$
U^{s.\flat}=U_0^{s.\flat} \leftarrow \cdots \leftarrow U_j^{s.\flat} \leftarrow U_{j+1}^{s.\flat} \leftarrow \cdots.
$$
So let $U_\infty^{s.\flat}$ be the inverse image of $U^{s.\flat}$ under $\Spec(R_\infty^{s.\flat}) \to \Spec(R^{s.\flat})$. 
Since $U_\infty^{s.\flat} \to U^{s,\flat}$ is a universal homeomorphism, the preservation of the small \'etale sites (\cite[Tag 03SI]{Stacks}) gives an isomorphism:
\begin{equation}
\label{Final2}
H^i(U_{\et}^{s.\flat},\mathbb{Z}/\ell^n\mathbb{Z}) \cong H^i(U_{\infty,\et}^{s.\flat},\mathbb{Z}/\ell^n\mathbb{Z}).
\end{equation}
Now the combination of Lemma \ref{lem1218SunN} and Theorem \ref{HenselianEtale} (2) together with the assumption finishes the proof of the theorem.
\end{proof}

\begin{remark}
One can formulate and prove the version of Proposition \ref{TiltEtaleCohHensel} for the \'etale cohomology with support in a closed subscheme of $\Spec(R)$, using Theorem \ref{tiltingCoh}. 
Then the resulting assertion gives a generalization of \v{C}esnavi\v{c}ius-Scholze's argument in \cite[Theorem 3.1.3]{CK19} which is a key part of their proof for the absolute cohomological purity theorem. One of the advantages of Proposition \ref{TiltEtaleCohHensel} is that it can be used to answer some cohomological questions on possibly singular Noetherian schemes (e.g. log-regular schemes) in mixed characteristic. 
\end{remark}

\subsection{Tilting the divisor class groups of local log-regular rings}\label{divisorclassLR}

We need a lemma of Grothendieck on the relationship between the divisor class group and the Picard group via direct limit. Its proof is found in \cite[Proposition (21.6.12)]{Gro67} or \cite[XI Proposition 3.7.1]{GroSGA}.

\begin{lemma}
\label{Grothendieck}
Let $X$ be an integral Noetherian normal scheme, and let $\{U_i\}_{i \in I}$ be a family of open subsets of $X$. Consider the following conditions.
\begin{enumerate}
\item
$\{U_i\}_{i \in I}$ forms a filter base. 
In particular, one can define a partial order on $I$ so that it is a directed set and $\{U_i\}_{i\in I}$ together with the inclusion maps forms an inverse system. 

\item
Let $V_i:=X \setminus U_i$ for any $i\in I$. Then $\codim_X(V_i) \ge 2$.

\item
For any $x \in \bigcap_{i \in I} U_i$, the local ring $\mathcal{O}_{X,x}$ is factorial.
\end{enumerate}
If  $\{U_i\}_{i \in I}$ satisfies the condition $(2)$, then the natural map $\Pic(U_i) \to \Cl(X)$ is injective for any $i \in I$. If $\{U_i\}_{i \in I}$ satisfies the conditions $(1)$, $(2)$ and $(3)$, then $\varinjlim_{i \in I} \Pic(U_i) \cong \Cl(X)$. In particular, if $U\subseteq X$ is any open subset that is locally factorial with $\codim_X(X \setminus U) \ge 2$, then $\Pic(U) \cong \Cl(X)$.
\end{lemma}

Next we establish the following two results on the torsion part of the divisor class group of a (Noetherian) normal domain; these are a part of numerous applications of Theorem \ref{HenselianEtale} of independent interest.

\begin{proposition}
\label{DivPicard}
Let $(R,\fm,k)$ be a strictly Henselian Noetherian local normal $\mathbb{F}_p$-domain of dimension $\ge 2$, let $X:=\Spec(R)$ and fix an ideal $J \subseteq \fm$. Let $\{U_i\}_{i \in I}$ be any family of open subsets of $X$ satisfying $(1)$, $(2)$ and $(3)$ as in the hypothesis of Lemma \ref{Grothendieck} and let $U_i^\infty$ be the $\mathbb{F}_p$-scheme which is the perfection of $U_i$.
\begin{enumerate}
\item
For any prime $\ell \ne p$, 
$$
\Cl(X)[\ell^n] \cong \varinjlim_{i \in  I}H^1\big((U^\infty_i)_{\et},\mathbb{Z}/\ell^n\mathbb{Z}\big).
$$

\item
Let $\widehat{R^{1/p^{\infty}}}$ denote the $J$-adic completion of $R^{1/p^\infty}$. If moreover each $U_i$ has the property that $X \setminus V(J) \subseteq U_i$, then for any prime $\ell \ne p$,
$$
\Cl(X)[\ell^n] \cong \varinjlim_{i \in  I}H^1\big((\widehat{U^\infty_i})_{\et},\mathbb{Z}/\ell^n\mathbb{Z}\big),
$$
where $\widehat{U_i^\infty}$ is inverse image of $U_i^\infty$ via the scheme map $\Spec(\widehat{R^{1/p^{\infty}}}) \to \Spec(R^{1/p^{\infty}})$.
\end{enumerate}
\end{proposition}

\begin{proof}
Let us begin with a remark on the direct limit of \'etale cohomology groups. Note that for the transition morphism $g:U_i^\infty \to U_j^\infty$ which is affine, there is a functorial map: $H^1\big((U^\infty_j)_{\et},\mathbb{Z}/\ell^n\mathbb{Z}\big) \to H^1\big((U^\infty_i)_{\et},g^*(\mathbb{Z}/\ell^n\mathbb{Z})\big) \cong H^1\big((U^\infty_i)_{\et},\mathbb{Z}/\ell^n\mathbb{Z}\big)$, which defines the direct system of cohomology groups.

$(1)$: First we prove the following claim:
\begin{enumerate}
\item[$\bullet$]
There is an injection of abelian groups:
$$
H^1(U_{et},\mathbb{Z}/\ell^n\mathbb{Z}) \cong \Pic(U)[\ell^n] \subseteq \Cl(X)[\ell^n]
$$
for any $n \in \mathbb{N}$, where $U \subseteq X$ is an open subset whose complement is of codimension $\ge 2$.
\end{enumerate}
To prove this, consider the Kummer exact sequence
$$
0 \to \mathbb{Z}/\ell^n\mathbb{Z} \cong \mu_{\ell^n} \to \mathbb{G}_m \xrightarrow{(~)^{\ell^n}} \mathbb{G}_m \to 0,
$$
where the identification of \'etale sheaves $\mu_{\ell^n} \cong \mathbb{Z}/\ell^n\mathbb{Z}$ follows from the fact that $R$ is strict Henselian (one simply sends $1 \in \mathbb{Z}/\ell^n\mathbb{Z}$ to the primitive $\ell^n$-th root of unity in $R$). Let $U \subseteq X$ be an open subset with its complement $V=X \setminus U$ having codimension $\ge 2$. Then we have an exact sequence (\cite[Proposition 4.9; Chapter III]{Mi80}):

$$
\Gamma(U_{\et},\mathbb{G}_m) \xrightarrow{(~)^{\ell^n}} \Gamma(U_{\et},\mathbb{G}_m) \to H^1(U_{\et},\mathbb{Z}/\ell^n\mathbb{Z}) \to \Pic(U) \xrightarrow{(~)^{\ell^n}} \Pic(U).
$$
Since $R$ is strict local and $\ell \ne p$, Hensel's lemma yields that $R^\times=(R^\times)^{\ell^n}$. Moreover, since $\codim_X(V) \ge 2$ and $X$ is normal, we have $\Gamma(U_{\et},\mathbb{G}_m)=R^\times$. Thus, $H^1(U_{\et},\mathbb{Z}/\ell^n\mathbb{Z}) \cong \Pic(U)[\ell^n]$. Note that $\Pic(U) \hookrightarrow \Cl(U)$ restricts to $\Pic(U)[\ell^n] \hookrightarrow \Cl(U)[\ell^n]$. Moreover, the natural homomorphism $\Cl(X) \to \Cl(U)$ is an isomorphism, thanks to $\codim_X(V) \ge 2$. Hence $H^1(U_{\et},\mathbb{Z}/\ell^n\mathbb{Z})\cong \Pic(U)[\ell^n] \subseteq \Cl(X)[\ell^n]$, which proves the claim.

Since $R$ is normal, the regular locus has complement with codimension $\ge 2$. Using this fact, we can apply Lemma \ref{Grothendieck} to get an isomorphism $\Cl(X)[\ell^n] \cong \varinjlim_{i \in I}H^1\big((U_i)_{\et},\mathbb{Z}/\ell^n\mathbb{Z}\big)$. By \'etale invariance of cohomology under taking perfection of $\mathbb{F}_p$-schemes (\cite[Tag 03SI]{Stacks}), we get
$$
\Cl(X)[\ell^n] \cong \varinjlim_{i \in I} H^1\big((U_i)_{\et},\mathbb{Z}/\ell^n\mathbb{Z}\big) \cong \varinjlim_{i \in I} H^1\big((U^\infty_i)_{\et},\mathbb{Z}/\ell^n\mathbb{Z}\big),
$$
as desired.

$(2)$: Since $R$ is Henselian along $\fm$ and $J \subseteq \fm$, it is Henselian along $J$ by \cite[Tag 0DYD]{Stacks}. Moreover, the perfect closure of $R$ still preserves Henselian property along $J$. Theorem \ref{HenselianEtale} yields
$$
H^1\big((U^\infty_i)_{\et},\mathbb{Z}/\ell^n\mathbb{Z}\big) \cong H^1\big((\widehat{U^\infty_i})_{\et},\mathbb{Z}/\ell^n\mathbb{Z}\big)
$$
and the conclusion follows from $(1)$.
\end{proof}

\begin{proposition}
\label{Basechange1}
Let $A$ be a Noetherian ring with a regular element $t \in A$ such that $A$ is $t$-adically Henselian and $A \to A/tA$ is the natural surjection between locally factorial domains. Pick an integer $n>0$ that is invertible on $A$. Then if $\Cl(A)$ has no torsion element of order $n$, the same holds for $\Cl(A/tA)$. If moreover $A$ is a $\mathbb{Q}$-algebra and $\Cl(A)$ is torsion-free, then so is $\Cl(A/tA)$.
\end{proposition}

\begin{proof}
The Kummer exact sequence $0 \to \mu_{n} \to \mathbb{G}_m \xrightarrow{(~)^{n}} \mathbb{G}_m \to 0$ induces the following commutative diagram:
$$
\begin{CD}
H^1(\Spec(A)_{\et},\mu_n) @>\delta_1>> \Pic(A) @>(~)^n>> \Pic(A) \\
@V\alpha VV @VVV @VVV \\
H^1(\Spec(A/tA)_{\et},\mu_n) @>\delta_2>> \Pic(A/tA) @>(~)^n>> \Pic(A/tA)
\end{CD}
$$
By Theorem \ref{HenselianEtale}, the map $\alpha$ is an isomorphism. Then if $\Pic(A)$ has no torsion element of order $n$, $\delta_1$ is the zero map. This implies that $\delta_2$ is also the zero map and hence, $\Pic(A/tA)$ has no element of order $n$. Since both $A$ and $A/tA$ are locally factorial by assumption, we have $\Cl(A) \cong \Pic(A)$ and $\Cl(A/tA) \cong \Pic(A/tA)$. So the assertion follows.
\end{proof}

It is not necessarily true that $\delta_1$ (resp. $\delta_2$) is injective, because we do not assume $A$ to be strictly Henselian. 
%We use the small tilts of the perfectoid tower of local log-regular rings to prove some result on the divisor class group.

\begin{lemma}
\label{logstricthenselization}
Let $(R,\mathcal{Q},\alpha)$ be a log-regular ring. Then a strict Henselization $(R^{\rm{sh}}, \mathcal{Q}, \alpha^{\rm{sh}})$ is also a log-regular ring where $\alpha^{\rm{sh}}: \mathcal{Q} \to R \to R^{\rm{sh}}$ is the composition of homomorphisms.
\end{lemma}
\begin{proof}
Since $R \to R^{\rm{sh}}$ is a local ring map, $(R^{\rm{sh}}, \mathcal{Q}, \alpha^{\rm{sh}})$ is a local log ring by Lemma $\ref{locallogextension}$. Note that we have the equality $I_{\alpha^{\rm{sh}}}= I_{\alpha}R^{\rm{sh}}$. Since we have the isomorphism $R^{\rm{sh}}/I_{\alpha^{\rm{sh}}} \cong (R/I_{\alpha})^{\rm{sh}}$ by \cite[Tag 05WS]{Stacks} and $(R/I_{\alpha})^{\rm{sh}}$ is a regular local ring by \cite[Tag 06LN]{Stacks}, $R^{\rm{sh}}/I_{\alpha^{\rm{sh}}}$ is a regular local ring. Moreover, since the dimension of $R$ is equal to the dimension of a strict henselization $R^{\rm{sh}}$, we obtain the following equalities:
$$
\dim R^{\rm{sh}}-\dim (R^{\rm{sh}}/I_{\alpha^{\rm{sh}}})
=\dim R^{\rm{sh}}-\dim (R/I_{\alpha})^{\rm{sh}}
=\dim R - \dim (R/I_{\alpha})
=\dim \mathcal{Q}.
$$
So the local log ring $(R^{\rm{sh}}, \mathcal{Q}, \alpha^{\rm{sh}})$ is log-regular.
\end{proof}

Now we can prove the following result on the divisor class groups of local log-regular rings, as an application of the theory of perfectoid towers.

\begin{theorem}
\label{torsiondivisorclass}
Let $(R,\mathcal{Q},\alpha)$ be a local log-regular ring of mixed characteristic with perfect residue field $k$ of characteristic $p>0$, and denote by $\Cl(R)$ the divisor class group with its torsion subgroup $\Cl(R)_{\rm{tor}}$. Then the following assertions hold.
\begin{enumerate}
\item
Assume that $R \cong W(k)\llbracket\mathcal{Q}\rrbracket$ for a fine, sharp, and saturated monoid $\mathcal{Q}$, where $W(k)$ is the ring of Witt vectors over $k$. Then $\Cl(R)_{\rm{tor}} \otimes \mathbb{Z}[\frac{1}{p}]$ is a finite group. In other words, the $\ell$-primary subgroup of $\Cl(R)_{\rm{tor}}$ is finite for all primes $\ell \ne p$ and vanishes for almost all primes $\ell \ne p$.

\item
Assume that $\widehat{R^{\rm{sh}}}[\frac{1}{p}]$ is locally factorial, where $\widehat{R^{\rm{sh}}}$ is the completion of the strict Henselization $R^{\rm{sh}}$. Then $\Cl(R)_{\rm{tor}} \otimes \mathbb{Z}[\frac{1}{p}]$ is a finite group. In other words, the $\ell$-primary subgroup of $\Cl(R)_{\rm{tor}}$ is finite for all primes $\ell \ne p$ and vanishes for almost all primes $\ell \ne p$.
\end{enumerate}
\end{theorem}

\begin{proof}
The assertion $(1)$ was already proved in Proposition \ref{torsiondivisorunramified}. 
So let us prove the assertion $(2)$. 
We may assume that $\mathcal{Q}$ is fine, sharp, and saturated by Remark \ref{rmk2125}.
The proof given below works for the first case under the assumption of local factoriality of $\widehat{R^{\rm{sh}}}[\frac{1}{p}]$.

Since $R \to \widehat{R^{\rm{sh}}}$ is a local flat ring map, the induced map $\Cl(R) \to \Cl(\widehat{R^{\rm{sh}}})$ is injective by Mori's theorem (c.f. \cite[Corollary 6.5.2]{Fo17}). Thus, it suffices to prove the theorem for $\widehat{R^{\rm{sh}}}$. Moreover, $\widehat{R^{\rm{sh}}}$ is log-regular with respect to the induced log ring structure $\alpha:\mathcal{Q} \to R \to \widehat{R^{\rm{sh}}}$ by Lemma \ref{logstricthenselization}. So without loss of generality, we may assume that the residue field of $R$ is separably closed (hence algebraically closed in our case).

Henceforth, we denote $\widehat{R^{\rm{sh}}}$ by $R$ for brevity and fix a prime $\ell$ that is different from $p$. By Lemma \ref{Grothendieck} and the local factoriality of $R[\frac{1}{p}]$, we claim that there is an open subset $U \subseteq X:=\Spec(R)$ such that the following holds:
\begin{enumerate}
\item[$\bullet$]
$\Pic(U) \cong \Cl(X)$, $X \setminus V(p R) \subseteq U$ and $\codim_X(X \setminus U) \ge 2$.
\end{enumerate}
Indeed, note that $X$ is a normal integral scheme by Kato's theorem (Theorem \ref{CMnormal}) and let $U$ be the union of the regular locus of $X$ and the open $\Spec(R[\frac{1}{p}]) \subseteq X$. Then by Serre's normality criterion, we see that $\codim_X(X \setminus U) \ge 2$. We fix such an open $U \subseteq X$ once and for all. Taking the cohomology sequence associated to the exact sequence
$$
0 \to \mathbb{Z}/\ell^n\mathbb{Z}  \to \mathbb{G}_m \xrightarrow{(~)^{\ell^n}} \mathbb{G}_m \to 0
$$
on the strict local scheme $X$ and arguing as in the proof of Proposition \ref{DivPicard}, we have an isomorphism:
\begin{equation}
\label{etalePicDiv}
H^1(U_{\et}, \mathbb{Z}/\ell^n\mathbb{Z}) \cong \Pic(U)[\ell^n] \cong \Cl(X)[\ell^n].
\end{equation}
On the other hand, there is a perfectoid tower of module-finite extensions of local log-regular rings arising from $(R, (p))$:
\begin{equation}\label{ptLRPfinMT8}
(R,\mathcal{Q},\alpha)=(R_0,\mathcal{Q}^{(0)},\alpha_0) \to \cdots \to (R_j,\mathcal{Q}^{(j)},\alpha_j) \to (R_{j+1},\mathcal{Q}^{(j+1)},\alpha_{j+1}) \to \cdots. 
\end{equation}
Notice that each map is generically of $p$-power rank in view of Lemma \ref{injective1} (2) and Lemma \ref{finite424} (3). Moreover, the tilt of (\ref{ptLRPfinMT8}) (associated to $(R, (p))$) is given by 
$$
(R^{s. \flat},\mathcal{Q}^{s.\flat},\alpha^{s.\flat})=((R_0)^{s.\flat}_{(p)}, \mathcal{Q}^{s.\flat}_{0}, \alpha^{s.\flat}_{0}) \to \cdots
 \to  ((R_j)^{s.\flat}_{(p)}, \mathcal{Q}^{s.\flat}_{j}, \alpha^{s.\flat}_{j}) \to ((R_{j+1})^{s.\flat}_{(p)}, \mathcal{Q}^{s.\flat}_{j+1}, \alpha^{s.\flat}_{j+1}) \to \cdots,
$$
where $((R_j)^{s.\flat}_{(p)}, \mathcal{Q}^{s.\flat}_{j}, \alpha^{s.\flat}_{j})$ is a complete local log-regular ring of characteristic $p>0$ in view of Theorem \ref{TiltingLogRegular}.
The local ring $R^{s.\flat}$ is strictly Henselian and the complement of $U^{s.\flat}(=U^{s.\flat}_{(p)})$ has codimension $\ge 2$ in $\Spec(R^{s.\flat})$, and by repeating the proof of Proposition \ref{DivPicard}, we obtain an isomorphism
\begin{equation}
\label{Final3}
H^1(U_{\et}^{s.\flat},\mathbb{Z}/\ell^n\mathbb{Z}) \cong \Pic(U^{s.\flat})[\ell^n].
\end{equation}
By Lemma \ref{Grothendieck}, the map 
\begin{equation}
\label{Final4}
\Pic(U^{s.\flat})[\ell^n] \to \Cl(R^{s.\flat})[\ell^n]
\end{equation}
is injective. Combining $(\ref{etalePicDiv})$, $(\ref{Final3})$, $(\ref{Final4})$ and Proposition \ref{TiltEtaleCohHensel} together, it is now sufficient to check that there exists an integer $N>0$ depending only on $R^{s.\flat}$ such that
$$
\Cl(R^{s.\flat})[\ell^N]=\bigcup_{n>0} \Cl(R^{s.\flat})[\ell^n],~\mbox{and}~
\Cl(R^{s.\flat})[\ell^N]~\mbox{is finite for all}~\ell~\mbox{and zero for almost all}~\ell \ne p.
$$
Since we know that $R^{s.\flat}$ is strongly $F$-regular by Theorem \ref{TiltingLogRegular} and Lemma \ref{F-regularLog}, the aforementioned result of Polstra finishes the proof.
\end{proof}

\section{Appendix: Construction of differential modules and maximality}\label{AppendixA}
The content of this appendix is taken from Gabber-Ramero's treatise \cite{GR22} whose purpose is to supply the corrected version of Grothendieck's original presentation in EGA. So we give only a sketch of the constructions of relevant modules and maps. The readers are encouraged to look into \cite{GR22} for more details as well as proofs. We are motivated by the following specific problem.

\begin{Problem}
\label{ringextension}
Let $(A,\fm_A)$ be a Noetherian regular local ring and fix a system of elements $f_1,\ldots,f_n \in A$ and a system of integers $e_1,\ldots,e_n$ with $e_i > 1$ for every $i=1,\ldots, n$. 
We set
$$
B:=A[T_1,\ldots,T_n]/(T_1^{e_1}-f_1,\ldots,T_d^{e_n}-f_n).
$$
Then find a sufficient condition that ensures that the localization $B$ with respect to a maximal ideal $\fn$ with $\fm_A=A \cap \fn$ is regular.
\end{Problem}

From the construction, it is obvious that the induced ring map $A \to B$ is a flat finite injective extension. Let now $(A,\fm_A,k)$ be a Noetherian local ring with residue field $k_A:=A/\fm_A$ of characteristic $p>0$. Following the presentation in \cite[(9.6.15)]{GR22}, we define a certain $k_A^{1/p}$-vector space $\bold{\Omega}_A$ together with a map $\boldsymbol{d}_A: A\to \bold{\Omega}_A$ as follows.

\begin{itemize}
\item[Case I:]($p \notin \fm_A^2$)

Let $W_2(k_A)$ denote the $p$-typical ring of length $2$ Witt vectors over $k_A$. Then there is the ghost component map $\overline{\omega}_0: W_2(k_A)\to k_A$, and set 
$V_1(k_A):=\ker(\overline{\omega}_0)$. More specifically, we have $W_2(k_A)=k_A \times k_A$ as sets with addition and multiplication given respectively by 
$$
(a,b)+(c,d)=\Big(a+c,b+d+\frac{a^p+c^p-(a+c)^p}{p}\Big)~\mbox{and}~(a,b)(c,d)=(ac,a^pd+c^pb).
$$
Using this structure, we see that $V_1(k_A)=0 \times k_A$ as sets, which is an ideal of $W_2(k_A)$ and $V_1(k_A)^2=0$. This makes $V_1(k_A)$ equipped with the structure as a $k_A$-vector space by letting $x(0,a):=(x,0)(0,a)$ for $x \in k_A$. One can define the map of $k_A$-vector spaces
\begin{equation}
\label{vectoriso1}
k_A^{1/p} \to V_1(k_A)\ ;\ a \mapsto (0,a^p),
\end{equation}
which is a bijection. With this isomorphism, we may view $V_1(k_A)$ as a $k_A^{1/p}$-vector space. Next we form the fiber product ring:
$$
A_2:=A \times_{k_A} W_2(k_A).
$$
It gives rise to a short exact sequence of $A_2$-modules
\begin{equation}
\label{exactseq1}
0 \to V_1(k_A) \to A_2 \to A \to 0,
\end{equation}
where $A_2 \to A$ is the natural projection, and the $A_2$-module structure of $V_1(k_A)$ is via the restriction of rings $A_2 \to W_2(k_A)$. From $(\ref{exactseq1})$, we obtain an exact sequence of $A$-modules:
$$
V_1(k_A) \to \overline{\Omega}_A \to \Omega^1_{A/\mathbb{Z}} \to 0,
$$
where we put $\overline{\Omega}_A=\Omega^1_{A_2/\mathbb{Z}} \otimes_{A_2} A$. After applying $(~) \otimes_A k_A$ to this sequence, we have another sequence of $k_A$-vector spaces:
\begin{equation}
\label{exactseq2}
0 \to V_1(k_A) \xrightarrow{j_A} \overline{\Omega}_A \otimes_A k_A \to \Omega^1_{A/\mathbb{Z}} \otimes_A k_A \to 0.
\end{equation}
Then this is right exact. Moreover, $(\ref{vectoriso1})$ yields a unique $k_A$-linear map $\psi_A:V_1(k_A) \otimes_{k_A} k_A^{1/p} \to V_1(k_A)$. Define $\boldsymbol{\Omega}_A$ as the push-out of the diagram:
$$
V_1(k_A) \xleftarrow{\psi_A} V_1(k_A)\otimes_{k_A}k_A^{1/p} \xrightarrow{j_A \otimes k_A^{1/p}} \overline{\Omega}_{A}\otimes_{A}k_A^{1/p}.
$$
More concretely, we have
$$
\boldsymbol{\Omega}_A=\dfrac{V_1(k_A) \oplus (\overline{\Omega}_{A}\otimes_{A}k_A^{1/p})}{T},
$$
where $T=\big\{(\psi(x),-(j_A \otimes k_A^{1/p})(x))~\big|~x \in V_1(k_A)\otimes_{k_A}k_A^{1/p}\big\}$. By the universality of push-outs, we get the commutative diagram:
$$
\begin{CD}
0 @>>> V_1(k_A)\otimes_{k_A}k_A^{1/p} @>>> \overline{\Omega}_{A}\otimes_{A}k_A^{1/p} @>>> \Omega^1_{A/\mathbb{Z}}\otimes_Ak_A^{1/p} @>>> 0\\
&& @V\psi_AVV @V\boldsymbol{\psi}_AVV @|\\
0 @>>> V_1(k_A) @>>> \boldsymbol{\Omega}_A @>>> \Omega^1_{A/\mathbb{Z}}\otimes_Ak_A^{1/p} @>>> 0
\end{CD}
$$
We define the map
$$
\boldsymbol{d}_A:A \to \boldsymbol{\Omega}_A
$$
as the composite mapping
$$
A \xrightarrow{1 \times \tau_{k_A}} A_2=A \times_{k_A} W_2(k_A) \xrightarrow{d}\Omega^1_{A_2/\mathbb{Z}} \xrightarrow{\id \otimes 1} \overline{\Omega}_A=\Omega^1_{A_2/\mathbb{Z}} \otimes_A k_A^{1/p}  \xrightarrow{\boldsymbol{\psi}_A} \boldsymbol{\Omega}_A.
$$
Here, $d:A_2 \to \Omega^1_{A_2/\mathbb{Z}}$ is the universal derivation and $\tau_{k_A}:A \to k_A \to W_2(k_A)$, where the first map is the natural projection and the second one is the Teichm\"uller map.

\item[Case II:]($p\in \mathfrak{m}_A^2$)

We just set $\bold{\Omega}_A:=\Omega^1_{A/\mathbb{Z}}\otimes_Ak_A^{1/p}$, and define $\boldsymbol{d}_A: A\to \bold{\Omega}_A$ as the map induced by the universal derivation $d_A: A \to \Omega^1_{A/\mathbb{Z}}$. 
\end{itemize}

Combining both Case I and Case II together, we have a map $\boldsymbol{d}_A: A\to \bold{\Omega}_A$. Moreover, if $\phi: (A,\fm_A) \to (B,\fm_B)$ is a local ring map of local rings, it gives rise to the following commutative diagram:
$$
\begin{CD}
A @>\boldsymbol{d}_A>> \bold{\Omega}_A \\
@V\phi VV @V\bold{\Omega}_{\phi}VV \\
B @>\boldsymbol{d}_B>> \bold{\Omega}_B
\end{CD}
$$
With this in mind, one can consider the functor $A \mapsto \boldsymbol{\Omega}_A$ from the category of local rings $(A,\fm_A)$ of residual characteristic $p>0$ to the category of the $k_A^{1/p}$-vector spaces $\boldsymbol{\Omega}_A$. Some distinguished features in the construction above are expressed by the following proposition.

\begin{proposition}[{\cite[Proposition 9.6.20]{GR22}}]
\label{Differentialmodule}
Let $\phi:(A,\fm_A) \to (B,\fm_B)$ be a local ring map of Noetherian local rings such that the residual characteristic of $A$ is $p>0$. Then
\begin{enumerate}
\item
Suppose that $\phi$ is formally smooth for the $\fm_A$-adic topology on $A$ and the $\fm_B$-adic topology on $B$. Then the maps induced by $\phi$ and $\boldsymbol{\Omega}_\phi$ respectively
$$
(\fm_A/\fm_A^2) \otimes_{k_A} k_B \to \fm_B/\fm_B^2,~\boldsymbol{\Omega}_A \otimes_{K_A^{1/p}} k_B^{1/p} \to \boldsymbol{\Omega}_B
$$
are injective.

\item
Suppose that
\begin{enumerate}
\item
$\fm_A B=\fm_B$.

\item
The residue filed extension $k_A \to k_B$ is separable algebraic.

\item
$\phi$ is flat.
\end{enumerate}
Then $\boldsymbol{\Omega}_\phi$ induces an isomorphism of $k_A^{1/p}$-vector spaces:
$$
\boldsymbol{\Omega}_A \otimes_A B \cong \boldsymbol{\Omega}_B.
$$

\item
If $B=A/\fm_A^2$ and $\phi:A \to B$ is the natural map, then $\boldsymbol{\Omega}_\phi$ is an isomorphism.

\item
The functor $\boldsymbol{\Omega}_\bullet$ and the natural transformation $\boldsymbol{d}_\bullet$ commute with filtered colimits.
\end{enumerate}
\end{proposition}

We provide an answer to Problem \ref{ringextension} as follows.

\begin{theorem}[{\cite[Corollary 9.6.34]{GR22}}]
\label{criterion-regularity}
Let $f_1,\ldots, f_n$ be a sequence of elements in $A$, and let $e_1,\ldots, e_n$ be a system of integers with $e_i>1$ for every $i=1,\ldots, n$. 
Set 
$$
C:=A[T_1,\ldots, T_n]/(T_1^{e_1}-f_1,\ldots, T_n^{e^n}-f_n).
$$
Fix a prime ideal $\mathfrak{n}\subseteq C$ such that $\mathfrak{n}\cap A=\mathfrak{m}_A$, and let $B:=C_\mathfrak{n}$. 
Let $E\subseteq \bold{\Omega}_A$ be the $k_A^{1/p}$-vector space spanned by $\boldsymbol{d}_Af_1,\ldots, \boldsymbol{d}_Af_n$. Then the following conditions are equivalent. 
\begin{enumerate}
\item
$A$ is a regular local ring, and $\dim_{k_A^{1/p}}E=n$.
\item
$B$ is a regular local ring. 
\end{enumerate}
\end{theorem}

In particular, in the situation of the above theorem, $B$ is a regular local ring if $A$ is a regular local ring and $f_1,\ldots, f_n$ is \emph{maximal} in the sense of the following definition.

\begin{definition}
\label{maximal}
Let $(A,\fm_A,k_A)$ be a local ring with residual characteristic $p>0$. Then we say that a sequence of elements $f_1,\ldots, f_n$ in $A$ is \emph{maximal} if $\boldsymbol{d}_Af_1,\ldots, \boldsymbol{d}_Af_n$ forms a basis of the $k_A^{1/p}$-vector space $\boldsymbol{\Omega}_A$.
\end{definition}

In general, we have the following fact.

\begin{lemma}\label{maximalityRLR}
Let $(A,\fm_A,k_A)$ be a regular local ring of mixed characteristic and assume that
$f_1,\ldots,f_d$ is a regular system of parameters of $A$. Then the following hold:
\begin{enumerate}
\item
$f_1,\ldots,f_d$ satisfies the condition (1) of Theorem \ref{criterion-regularity}.

\item
If the residue field $k_A$ of $A$ is perfect, then the sequence $f_1,\ldots,f_d$ is maximal.
\end{enumerate}
\end{lemma}

\begin{proof}
(1): In the case that $p \notin \fm_A^2$, \cite[Proposition 9.6.17]{GR22} gives a short exact sequence:
\begin{equation}
\label{modifieddifferential}
0 \to \fm_A/\fm_A^2 \otimes_{k_A} k_A^{1/p} \to \boldsymbol{\Omega}_A \to \Omega^1_{k_A/\mathbb{Z}} \otimes_{k_A} k_A^{1/p} \to 0.
\end{equation}
Then the images $\overline{f_1},\ldots,\overline{f_d}$ form a basis of the $k_A^{1/p}$-vector space $\fm_A/\fm_A^2 \otimes_{k_A} k_A^{1/p}$. The desired claim follows from the left exactness of $(\ref{modifieddifferential})$.

In the case that $p \in \fm_A^2$, \cite[Lemma 9.6.6]{GR22} gives a short exact sequence
\begin{equation}
\label{modifieddifferential2}
0 \to \fm_A/(\fm_A^2+p\fm_A) \to \Omega_A \to \Omega^1_{k_A/\mathbb{Z}} \to 0.
\end{equation}
and we can argue as in the case $p \notin \fm_A^2$.

(2): If $k_A$ is perfect, then $\Omega^1_{k_A/\mathbb{Z}}=0$. Therefore, $(\ref{modifieddifferential})$ and $(\ref{modifieddifferential2})$ (in the latter case, one tensors it with $k_A^{1/p}$ over $k_A$) gives the desired conclusion.
\end{proof}


\begin{thebibliography}{99}

\bibitem{Aberbach01}
I. Aberbach, \emph{Extension of weakly and strongly $F$-regular rings by flat maps}, J. Algebra {\bf241} (2001), 799--807. 


\bibitem{Andre18}
Y. Andr\'e, \emph{La conjecture du facteur direct}, Publ. Math. I.H.\'E.S. \textbf{127} (2018), 71--93.



\bibitem{Bh14}
B. Bhatt, \emph{On the non-existence of small Cohen-Macaulay algebras}, 
J. Algebra {\bf411} (2014), 1--11.

\bibitem{Bh18}
B. Bhatt, \emph{On the direct summand conjecture and its derived variant}, Invent. math {\bf 212} (2018), 297--317, \url{https://doi.org/10.1007/s00222-017-0768-7}.


\bibitem{BGO20}
B. Bhatt, O. Gabber and M. Olsson,
\emph{Finiteness of \'etale fundamental groups by reduction modulo $p$},
\url{https://arxiv.org/abs/1705.07303}.


\bibitem{BMS18}
B. Bhatt, M. Morrow, and P. Scholze, 
\emph{Integral $p$-adic Hodge theory}, Publ. Math. I.H.\'E.S. \textbf{128} (2018), 219--397.


\bibitem{BG04}
W. Bruns and J. Gubeladze,
\emph{Semigroup algebras and discrete geometry}, 
Soci\'et\'e Math\'ematique de France. S\'emin. Congr. {\bf6}
(2002), 43--127.

\bibitem{CLMSK22}
H. Cai, S. Lee, L. Ma, K. Schwede, and K. Tucker, Perfectoid signature, perfectoid Hilbert-Kunz multiplicity, and an application to local fundamental groups, \url{https://arxiv.org/abs/2209.04046v2}.


\bibitem{CRST18}
J. Carvajal-Rojas, K. Schwede and K. Tucker,
\emph{Fundamental groups of $F$-regular singularities via $F$-signature},
Ann. Sci. \'Ec. Norm. Sup\'er. (4) {\bf51} (2018), 993--1016.


\bibitem{CK19}
K. \v{C}esnavi\v{c}ius, \emph{Purity for the Brauer group}, Duke Math. J. {\bf168} (2019), 1461--1486.


\bibitem{KS20}
K. \v{C}esnavi\v{c}ius and P. Scholze, \emph{Purity for flat cohomology},
Ann. of Math. {\bf199} (2024), 51--180.


\bibitem{Cla66}
L. Claborn, \emph{Every abelian group is a class group},
Pacific J. Math. {\bf18} (1966), 219--222.


\bibitem{Da78}
V. I. Danilov, \emph{The geometry of toric varieties},
Russian Math. Surveys {\bf33} (1978), 97--154.


\bibitem{DaTu20}
R. Datta and K. Tucker, \emph{On some permanence properties of (derived) splinters}, Michigan  Math. J. {\bf73} (2023), 371--400.


\bibitem{Fa02}
G. Faltings, \emph{Almost \'etale extensions},
Cohomologies $p$-adiques et applications arithm\'etiques, II. Ast\'erisque \textbf{279} (2002), 185--270. 


\bibitem{Fo17}
T. J. Ford, \emph{Separable algebras}, Graduate Studies in Mathematics {\bf183}, American Math. Society.


\bibitem{LF15}
L. Fu, \emph{\'Etale cohomology theory}, 2nd Revised ed.
Nankai Tracts in Mathematics {\bf14} (2015).


\bibitem{F95}
K. Fujiwara, \emph{Theory of tubular neighborhood in \'etale topology}, 
Duke Math. J. \textbf{80} (1995), 15--57.


\bibitem{F02}
K. Fujiwara, \emph{A proof of the absolute purity conjecture (after Gabber)}, 
Algebraic geometry 2000, Azumino (Hotaka), Adv. Stud. Pure Math., vol. 36, Math. Soc. Japan, Tokyo,  2002, 153--183. 



\bibitem{FGK11}
K. Fujiwara, O. Gabber and F. Kato. \emph{On Hausdorff completions of commutative rings in rigid geometry}, J. Algebra \textbf{332} (2011), 293-321.


\bibitem{FK18}
K. Fujiwara and F. Kato, \emph{Foundations of Rigid Geometry I}, EMS Monogr. Math. (2018).


\bibitem{Ful93}
W. Fulton, \emph{Introduction to toric varieties},
Princeton Univ. Press (1993).


\bibitem{G94}
O. Gabber, \emph{Affine analog of the proper base change theorem},
Israel J. Math. {\bf87} (1994), 325--335.


\bibitem{GR03}
O. Gabber and L. Ramero, \emph{Almost ring theory}, Lecture Notes in Mathematics \textbf{1800}, Springer.


\bibitem{GR22}
O. Gabber and L. Ramero, \emph{Almost rings and perfectoid spaces}, 
\url{https://pro.univ-lille.fr/fileadmin/user_upload/pages_pros/lorenzo_ramero/hodge.pdf}.


\bibitem{GrWe94}
P. Griffith and D. Weston, \emph{Restrictions of torsion divisor classes to hypersurfaces}, J. Algebra {\bf167} (1994), 473--487.


\bibitem{Gro67}
A. Grothendieck, \emph{\'El\'ements de g\'eom\'etrie alg\'ebrique : IV. \'Etude locale des sch\'emas et des morphismes de sch\'emas, Quatri\`eme partie}, Publ. Math. I.H.\'E.S. \textbf{32} (1967).


\bibitem{GroSGA}
A. Grothendieck, \emph{Cohomologie locale des faisceaux coh\'erents et th\'eor\`emes de Lefschetz locaux et globaux}, \url{https://arxiv.org/abs/math/0511279}.


\bibitem{HK20}
D. Hansen and K. Kedlaya, \emph{Sheafiness criteria for Huber rings}, preprint, \url{https://kskedlaya.org/papers/criteria.pdf}. 

\bibitem{Ho72}
M. Hochster, \emph{Rings of invariants of tori, Cohen-Macaulay rings generated by monomials, and polytopes}, Ann. of Math. {\bf96} (1972), 318--337.


\bibitem{HH89}
M. Hochster and C. Huneke, \emph{Tight closure and strong F-regularity},
Mem. Soc. Math. France (N. S.) {\bf38} (1989), 119--133.


\bibitem{Book}
L. Illusie, Y. Laszlo, and F. Orgogozo, \emph{Travaux de Gabber sur l'uniformisation locale et la cohomologie \'etale des sch\'emas quasi-excellents}, Ast\'erisque No. {\bf363-364}  (2014). Soci\'et\'e Math\'ematique de France, Paris, 2014.



\bibitem{Ish24}
S. Ishiro, \emph{Local log-regular rings vs toric rings
}, Communications in Algebra (2024), 1--16, \url{https://doi.org/10.1080/00927872.2024.2418377}.

\bibitem{Ka89}
K. Kato, \emph{Logarithmic structures of Fontaine-Illusie},
Algebraic analysis, geometry, and number theory (J.-I. Igusa, ed.), 
Johns Hopkins University Press, Baltimore, 1989, 191--224.


\bibitem{Ka94}
K. Kato, \emph{Toric singularities},
Amer. J. Math. {\bf116} (1994), 1073--1099.


\bibitem{MP}
L. Ma and T. Polstra, \emph{F-singularities: a commutative algebra approach}, \url{https://www.math.purdue.edu/~ma326/F-singularitiesBook.pdf}.

\bibitem{Ma86}
H. Matsumura, \emph{Commutative ring theory}, Cambridge
Studies in Advanced Mathematics $\bf{8}$, Cambridge University Press, Cambridge (1986).


\bibitem{MS04}
E. Miller and B. Sturmfels,
\emph{Combinatorial commutative algebra},
Graduate Texts in Mathematics {\bf227}, Springer (2004).


\bibitem{Mi80}
J. Milne, \emph{\'Etale cohomology}, Princeton Univ. Press (1980).


\bibitem{Mor17}
M. Morrow, \emph{Foundations of perfectoid spaces}, \url{https://webusers.imj-prg.fr/}\verb|~|\url{matthew.morrow/Espaces/Perf%20spaces%20Harvard%203.pdf}.


\bibitem{Murayama21}
T. Murayama, \emph{Relative vanishing theorems for $\mathbf{Q}$-schemes}, Algebraic Geometry \textbf{12} (2025), 84--144.


\bibitem{NS19}
K. Nakazato and K. Shimomoto, \emph{Finite \'etale extension of Tate rings and decompletion of perfectoid algebras}, J. Algebra  \textbf{589} (2022), 114--158.


\bibitem{Oda88}
T. Oda, \emph{Convex bodies and algebraic geometry: An introduction to the theory of toric varieties}, Springer (1988).


\bibitem{Ogus18}
A. Ogus, \emph{Lectures on logarithmic geometry},
Cambridge studies in advanced mathematics {\bf178}, Cambridge University Press.


\bibitem{Pol20}
T. Polstra, \emph{A theorem about maximal Cohen-Macaulay modules}, International Mathematics Research Notices, Volume 2022, Issue 3, February 2022, Pages 2086--2094, \url{https://doi.org/10.1093/imrn/rnaa154}.

\bibitem{Schenzel14}
P. Schenzel, \emph{A criterion for I-adic completeness},
Archiv der Mathematik \textbf{102} (2014), 25--33.


\bibitem{Sch12}
P. Scholze,  \emph{Perfectoid spaces}, Publ. Math. I.H.\'E.S. \textbf{116} (2012), 245--313.


\bibitem{Sch15}
P. Scholze, \emph{On torsion in the cohomology of locally symmetric varieties},
Ann. of Math. $\bf{182}$ (2015), 945--1066.

\bibitem{Sh11}
K. Shimomoto, \emph{Almost Cohen-Macaulay algebras in mixed characteristic via Fontaine rings}, Illinois J. Math. {\bf55} (2011), 107--125.

\bibitem{Stacks}
The Stacks Project Authors, \url{http://stacks.math.columbia.edu/browse}.

\bibitem{Xu14}
C. Xu, \emph{Finiteness of algebraic fundamental groups},
Compos. Math. {\bf150} (2014), 409--414.


\end{thebibliography}
\end{document}